\documentclass[a4paper]{amsart}
\usepackage{amsmath,amssymb,amsthm}

 \usepackage{url} 
 \usepackage{framed}
\usepackage{amsmath, color}
\usepackage{amssymb}
\usepackage{amssymb}
\usepackage{latexsym}
 \usepackage{graphicx, epstopdf, verbatim}
\usepackage{psfrag}

\newcommand{\rhode}{\sigma}

\newcommand{\ctext}{}

\newcommand{\R}{\mathbb R}
\newcommand{\Z}{\mathbb Z}
\newcommand{\N}{\mathbb N}
\newcommand{\HOX}[1]{\marginpar{\footnotesize #1}}
\newcommand{\Cone}{C_1}
\newcommand{\Ctwo}{C_2}
\newcommand{\Cfour}{C_3}
\newcommand{\Cfive}{C_4}
\newcommand{\cone}{C_5}
\newcommand{\ctwo}{C_6}
\newcommand{\cthree}{C_7}
\newcommand{\cfour}{C_8}
\newcommand{\Csix}{C_9}
\newcommand{\Cseven}{C_{10}}
\newcommand{\Ceight}{C_{11}}
\newcommand{\Cnine}{C_{12}}
\newcommand{\Cten}{C_{13}}
\newcommand{\cfive}{C_{14}}
\newcommand{\csix}{C_{15}}
\newcommand{\Cthirteen}{C_{16}}
\newcommand{\Celeven}{C_{17}}
\newcommand{\Ctwelve}{C_{18}}
\newcommand{\Csixteen}{C_{19}}
\newcommand{\Cnineteen}{C_{20}}
\newcommand{\Cseventeen}{C_{21}}
\newcommand{\Ctwenty}{C_{22}}
\newcommand{\Ctwentytwo}{C_{23}}
\newcommand{\Ctwentythree}{C_{24}}
\newcommand{\Ctwentyfour}{C_{25}}
\newcommand{\Ctwentyfive}{C_{26}}
\newcommand{\Ctwentysix}{C_{27}}
\newcommand{\Ctwentyseven}{C_{28}}
\newcommand{\Ctwentyeight}{C_{29}}
\newcommand{\Ctwentynine}{C_{30}}
\newcommand{\Cthirtythree}{C_{31}}
\newcommand{\ctwelve}{C_{32}}
\newcommand{\Cthirtysix}{C_{33}}
\newcommand{\Cthirtyseven}{C_{34}}
\newcommand{\Cthirtyeight}{C_{35}}
\newcommand{\Cthirtynine}{C_{36}}
\newcommand{\Cforty}{C_{37}}
\newcommand{\Cfortyone}{C_{38}}
\newcommand{\Cfortyfour}{C_{39}}
\newcommand{\Cfortyfive}{C_{40}}
\newcommand{\Cfortysix}{C_{41}}
\newcommand{\Cfortyseven}{C_{42}}
\newcommand{\Cfortyeight}{C_{43}}
\newcommand{\Cfortynine}{C_{44}}
\newcommand{\Cfifty}{C_{45}}
\newcommand{\Cfiftyone}{C_{46}}
\newcommand{\Cfiftytwo}{C_{47}}
\newcommand{\Cfiftythree}{C_{48}}
\newcommand{\Cfiftyfour}{C_{49}}
\newcommand{\Cfiftyfive}{C_{50}}
\newcommand{\Cfiftysix}{C_{51}}
\newcommand{\Cfiftyseven}{C_{52}}
\newcommand{\cfourteen}{C_{53}}
\newcommand{\Clast}{{C_{54}}}
\newcommand{\Cbracetthree}{{C_{55}}}
\newcommand{\Cbracetfour}{{C_{56}}}
\newcommand{\cbracetone}{{C_{57}}}
\newcommand{\Cbracetfive}{{C_{58}}}
\newcommand{\cbracettwo}{{C_{59}}}
\newcommand{\cbracetthree}{{C_{60}}}

\newcommand{\MM}{{\mathcal{M}}}

\newcommand{\figcommented}[1]{}
\newcommand{\epscommented}[1]{#1}

\def \beq {\begin {eqnarray}}
\def \eeq {\end {eqnarray}}
\def \ba {\begin {eqnarray*}}
\def \ea {\end  {eqnarray*}}

\renewcommand{\hat}{\widehat}
\renewcommand{\tilde}{\widetilde}

\newcommand{\ep}{\varepsilon}
\newcommand{\de}{\delta}
\newcommand{\la}{\lambda}

\newcommand{\Om}{\Omega}

\newcommand{\co}{\colon}
\newcommand{\pd}{\partial}

\newcommand{\be}{\begin{equation}}
\newcommand{\ee}{\end{equation}}

\DeclareMathOperator{\dist}{dist}
\DeclareMathOperator{\length}{length}

\DeclareMathOperator{\diam}{diam}
\DeclareMathOperator{\inj}{inj}
\DeclareMathOperator{\Reach}{Reach}
\DeclareMathOperator{\dis}{dis}

\numberwithin{equation}{section}

\newtheorem{theorem}{Theorem}

\newtheorem{lemma}{Lemma}[section]

\newtheorem{proposition}[lemma]{Proposition}
\newtheorem{corollary}[lemma]{Corollary}

\theoremstyle{definition}
\newtheorem{definition}[lemma]{Definition}
\newtheorem*{convention}{Convention}
\newtheorem*{notation}{Notation}

\theoremstyle{remark}
\newtheorem{remark}[lemma]{Remark}
\newtheorem{example}[lemma]{Example}

\newcommand{\Sec}{\operatorname{Sec}}

\begin{document}
\title[Geometric Whitney problem]{Reconstruction and interpolation of manifolds I: The geometric Whitney problem
}

\author
[Fefferman, Ivanov,  Kurylev, 
Lassas, Narayanan]
{Charles Fefferman, Sergei Ivanov,  Yaroslav Kurylev, \\
Matti Lassas, Hariharan Narayanan}

\address{\hspace{2cm} \linebreak
Charles Fefferman, Princeton University, Mathematics Department,
Fine Hall, Washington Road,
Princeton NJ, 08544-1000, USA.\hspace{7cm} 
\linebreak
 Sergei Ivanov,
St.~Petersburg Department of Steklov Institute of Mathematics, Russian Academy of Sciences, 
27 Fontanka, 191023 St.~Petersburg, Russia.
\hspace{6cm} 
\linebreak
Yaroslav Kurylev, 
University College London,
Department of Mathematics,
Gower Street WC1E  6BT, London, UK.
\hspace{8cm} 
\linebreak
Matti Lassas,
University of Helsinki,
Department of Mathematics and Statistics, P.O. Box 68,
00014, Helsinki, Finland.x
\hspace{8cm} 
 \linebreak
Hariharan Narayanan, 
School of Technology and Computer Science, 
Tata Institute for Fundamental Research, 
Mumbai 400005, India.
\hspace{24cm} 
 \linebreak 
 \linebreak {\bf Contact information of the corresponding author:}\hspace{8cm} 
  \linebreak 
 Matti Lassas\hspace{16cm}   \linebreak 
University of Helsinki,
Department of Mathematics and Statistics, P.O. Box 68  (Gustaf Hallstromin katu 2b),
00014, Helsinki, Finland.\hspace{24cm} 
 \linebreak 
 phone: -358-50-5674417, fax  -358-9-19151400, \hspace{24cm} 
 \linebreak 
e-mail: {Matti.Lassas@helsinki.fi}\hspace{8cm} 
 }


\begin{abstract}
We study the geometric Whitney problem on
how a Riemannian manifold $(M,g)$ can be constructed to approximate a metric space $(X,d_X)$.  
This  problem is closely related to  manifold interpolation (or manifold {reconstruction})
 where a smooth $n$-dimensional {submanifold} $S\subset {\mathbb R}^m$, $m>n$ needs to be constructed
to approximate a point cloud in ${\mathbb R}^m$.
These questions are  encountered in differential geometry, machine learning, and
in many inverse problems encountered in applications.
The determination of a Riemannian manifold includes the construction of its topology, differentiable 
structure, and metric. 
 
We give constructive solutions to the above problems. 
Moreover, we
characterize the metric spaces that can be approximated,
by Riemannian
manifolds with bounded geometry:
We give sufficient conditions to ensure that
 a metric space  can be approximated,
in the Gromov-Hausdorff or quasi-isometric sense, 
by a Riemannian
manifold of a fixed dimension and with bounded 
 diameter, sectional curvature,
and injectivity radius. Also, we show that similar conditions,
with modified values of parameters, are necessary.

 {As an application of the main results we give a new characterisation of Alexandrov spaces with two-sided curvature bounds.}
 Moreover, we characterise the subsets of Euclidean spaces that can be
approximated in the Hausdorff metric by submanifolds of a fixed dimension
and with bounded principal curvatures and normal injectivity radius.

 {We  develop {algorithmic procedures} that solve the geometric Whitney problem for a metric space and the
manifold {reconstruction} problem in Euclidean space, and estimate the computational complexity of these {procedures}}.

 The above interpolation problems   are also studied for unbounded metric sets and manifolds.
The results for Riemannian manifolds are based on a generalisation of the Whitney embedding construction where
approximative  coordinate charts are embedded in  ${\mathbb R}^m$ and interpolated to a smooth {submanifold}.


%
%
\end{abstract}



\maketitle


\tableofcontents

\section{Introduction and the main results}\label{section: Introduction}

\subsection{Geometrization of Whitney's extension problem}\label{subsec Geometrization}

      In this paper we develop a geometric version of Whitney's extension problem. 
Let $f : K \rightarrow  \R$ be a function defined on a given (arbitrary) set $K \subset  \R^n$, and let $m \geq 1$ be
a given integer. The classical Whitney problem is the question  
 whether $f$ extends to a function $F \in C^m(\R^n)$ and  if such an
$F$ exists, what is the optimal $C^m$ norm of the extension. Furthermore, one is interested in
the questions
if the derivatives
of $F$, up to order $m$, at a given point can be estimated, or if 
one can construct extension $F$ so that it depends linearly on $f$.

{\color{black}
These questions go back to the {work} of H.\ Whitney \cite{W1,W11,W12} in 1934.
In the decades 
since
Whitney's seminal work, fundamental progress was made by G.\ Glaeser \cite{G}, Y.\ Brudnyi and
P.\ Shvartsman \cite{Br1,Br2,Br3,Br4,BS1,BS2} and \cite{Shv1,Shv2,Shv3}, and E.\ Bierstone-P.\ Milman-W.\ Pawluski \cite{BMP}. (See also N.\
Zobin \cite{ZO1,ZO2} for the solution of a closely related problem.)

The above questions have been answered in the last few years, thanks to work of E.\ Bierstone,
Y.\ Brudnyi, C.\ Fefferman, P.\ Milman, W.\ Pawluski, P.\ Shvartsman and others, (see
\cite{BMP, Brom, Br1,Br3,Br4,BS2,F1,F2,F3,F4,F5,F6}). 
  Along the way,
the analogous problems with $C^m(\R^n)$ replaced by $C^{m,\omega }(\R^n)$, the space of functions whose
$m^{th}$ derivatives have a given modulus of continuity $\omega $, (see \cite{F5,F6}), were also solved.

The solution of Whitney's problems has led to a new algorithm for interpolation
of data, due to C.\ Fefferman and B.\ Klartag \cite{FK1,FK2}, where the authors show how to compute efficiently an interpolant $F(x),$ whose $C^m$ norm
lies within a factor $C$ of least possible,  where $C$ is a constant depending only on $m$ and $n.$ }

In recent years, the focus of attention in this problem has moved to the  direction when
the measurements $\tilde f:K\to \R$ on
the function $f$  are given with errors bounded by
$\ep>0$. Then, the task is to find a function $F:\R^n\to \R$ 
such that  $\sup_{x\in K} |F(x)-\tilde f(x)| \leq \ep$. Since the solution is not unique, one wants
to find {the extensions that have} the optimal norm in $C^m(\R^n)$,  see e.g. \cite{FK1, FK2}.
Finding $F$ can be considered as the task of finding a graph $\Gamma(F)=\{(x,F(x)):\ x\in \R^n\}\subset \R^{n+1}$
of a function in $C^m(\R^n)$
that {\ctext passes near} the points $\{(x,\tilde f(x)):\ x\in K\}$.
To formulate the above problems in geometric (i.e.\ coordinates invariant) terms,
instead of a graph set 
$\Gamma(F)$, we aim to construct a general {submanifold} or a Riemannian manifold
that approximates the given data. Also, instead of the $C^m(\R^n)$-norms,  we will measure
the optimality of the solution in terms of  invariant bounds for the curvature and the injectivity radius.

%
%

In this paper we consider the following two geometric  Whitney
problems:
\begin{itemize}

\item[A.] Let $E$ be a separable Hilbert space, e.g.\ $\R^N$, and assume that
we are given a set $X\subset E$.  When can one  construct a smooth
$n$-dimensional {submanifold} $M\subset E$  that approximates $X$ with given  bounds for 
 the geometry of $M$ and the Hausdorff distance between $M$ and $X$?
How {can}  the {submanifold} $M$  be efficiently constructed when $X$ is given? 

\item [B.] Let $(X,d_X)$ be a metric space. When there exists 
a Riemannian manifold $(M,g)$ that has given  bounds for geometry 
and approximates well~$X$? How can the manifold $(M,g)$ be
constructed when $X$  is given? {What is an algorithm that constructs $(M,g)$
when $X$ is finite?}
\end{itemize}

In Question B, by `approximation' we mean Gromov-Hausdorff
or quasi-isometric approximation, see definitions in Def.\ \ref{d:QI} and
Section \ref{subsec:gh}.

We answer  the Question $A$ 
 in Theorem \ref{t:surface} below, by showing that  if
  $X\subset E$ is locally  (i.e., at a certain small scale) close to affine $n$-dimensional planes, see Def.\ \ref{d:delta-flat},
there is a {submanifold} $M\subset E$  such that 
the Hausdorff distance of $X$  and $M$ is small and the second fundamental
form and the normal injectivity radius of $M$ are bounded.

The answer to the Question $B$ is given in Theorem \ref{t:manifold} below.
Roughly speaking, it asserts that the following natural conditions on $X$
are necessary and sufficient:
locally, $X$ should be close to $\R^n$,
and globally, the metric of $X$ should be almost intrinsic.

The conditions  
in Theorem \ref{t:manifold} are optimal, up to multiplying the obtained bounds  by a
constant factor depending on $n$.
Theorem \ref{t:manifold} gives sufficient conditions for metric spaces that 
approximate smooth manifolds.
In Corollary \ref{cor:compact class} we show that  similar conditions,
with modified values of parameters, are  necessary.

The result of 
 Theorem \ref{t:surface} is optimal, up to multiplication the obtained bounds  by a
 constant factor depending on $n$.

The proofs of  Theorems \ref{t:manifold} and \ref{t:surface} are constructive
and give raise to  {manifold reconstruction} algorithms when $X$ is a finite set.
Moreover, we give  algorithms
that verify if a finite data set $X$ satisfies the characterisations given in 
 Theorems \ref{t:manifold} and  \ref{t:surface}. {We analyse in subsection \ref{subsec :complexity 1}
the computational complexity of these algorithms, but emphasize that to keep the algorithms simple, the
algorithms have not been optimized  in this paper  to have the minimal complexity.}

{\ctext Finally, we note that  this paper is related to
 two rather different theorems of Whitney. One is the
Whitney embedding theorem of smooth manifolds into an Euclidean space and the other is the Whitney
extension theorem for functions in $\R^n$.} 

Next we formulate the definitions needed to state the results rigorously.

\begin{notation}
For a metric space $X$ and sets $A,B\subset X$,
we denote by $d_H^X(A,B)$, or just by $d_H(A,B)$,
the Hausdorff distance between $A$ and $B$ in $X$.

By $d_{GH}(X,Y)$ we denote the Gromov-Hausdorff (GH) distance
between metric spaces $X$ and~$Y$.
For the reader's convenience, we collect definitions
and elementary facts about the GH distance in section \ref{subsec:gh}.
For more detailed account of the topic, see e.g.\ \cite {BBI,Pe,Shiohama}.
In most cases we work with \emph{pointed} GH distance
between pointed metric spaces $(X,x_0)$ and $(Y,y_0)$,
where $x_0\in X$ and $y_0\in Y$ are distinguished points.
For the definition of pointed GH distance, 
see \cite[\S1.2 in Ch.~10]{Pe}) or section~\ref{subsec:gh}.

For a metric space $X$, $x\in X$ and $r>0$,
we denote by $B^X_r(x)$ or $B_r(x)$ 
the ball of radius $r$ centered at~$x$.
For $X=\R^n$, we use notation $B_r^n(x)=B_r^{\R^n}(x)$
and $B_r^n=B_r^n(0)$.
For a set $A\subset X$ and $r>0$, we denote by ${\mathcal U}_r^X(A)$ or ${\mathcal U}_r(A)$ the
metric neighborhood of $A$ of radius $r$, that is the set points
within distance $r$ from $A$.

When speaking about GH distance between metric balls $B_r^X(x)$
and $B_r^Y(y)$,
we always mean the pointed GH distance where the centers $x$ and $y$
are distinguished points of the balls.
We abuse notation and write $d_{GH}(B_r^X(x),B_r^Y(y))$
to denote this pointed GH distance.

For a Riemannian manifold $M$, we denote by $\Sec_M$ its sectional curvature
and by $\inj_M$ its injectivity radius.
\end{notation}

Small metric balls in a Riemannian manifold
are GH close to Euclidean balls.
{More precisely, let $M$ be a Riemannian $n$-manifold with $|\Sec_M|<K$
where $K$ is a positive constant,
and $0<r\le\frac12\min\{\frac\pi{\sqrt K},\inj_M\}$.
Then for every $x\in M$,
the metric balls $B^M_r(x)$ in $M$ and $B_r^n$ in $\R^n$ satisfy}
\be\label{eq: GH distance of balls}
 d_{GH}(B_r^M(x),B_r^n)\le {\tfrac14} Kr^3 .
\ee
For a proof of this estimate, see section~\ref{sec:injrad}.

If $M$ is a submanifold of $\R^N$,
one can write a similar estimate for the Hausdorff distance in $\R^N$.
Namely if the principal curvatures of $M$ are bounded by $\kappa>0$,
then $M$ deviates from its tangent space by at most $\tfrac12 \kappa r^2$ 
within a ball of radius $r$. 
Thus the Hausdorff distance between $r$-ball $B^M_r(x)$ in $M$ 
and the ball $B_r^{T_xM}(x)=B_r^N(x)\cap T_xM$ of the affine tangent space of $M$ at $x$ satisfy
\beq\label{eq: Hausdorff distance of balls}
 d_{H}(B^M_r(x),B_r^{T_xM}(x))\le \tfrac12 \kappa r^2 .
\eeq
Note the different order of the above estimates for the intrinsic distances
\eqref{eq: GH distance of balls} and  the extrinsic distances
\eqref{eq: Hausdorff distance of balls}.

With \eqref{eq: GH distance of balls} in mind,
we give the following definition.

\begin{definition}
\label{d:closetoRn}
Let $X$ be a metric space, $r>\delta>0, \, n \in \N$. We say that $X$ is 
\textit{$\delta$-close to $\R^n$ at scale $r$} if, for any $x \in X$,
\be \label{A}
d_{GH}(B_r^X(x), B_r^n) < \delta .
\ee
\end{definition}

Condition \eqref{A} can be effectively verified, up to a constant factor, see {\it Algorithm GHDist} below. The condition
can be also formulated for finite subsets:  If  sequences 
$(y_j)_{j=1}^N\subset  B_r^n$ and $(x_j)_{j=1}^N\subset B_r^X(x)$ 
are $\frac \delta 4$-nets such that $|d_{\R^n}(y_j,y_k)-d_X(x_j,x_k)|<\frac \de 4$ for all $j,k=1,2,\dots,N$,
then \eqref{A} is valid  by \cite[Prop.\ 7.3.16 and Cor.\ 7.3.28]{BBI}. 
On the other hand,
if $X$ is $\frac \delta{16}$-close to $\R^n$ at scale $r$, then such $\frac \delta 4$-nets exists.

In a Riemannian manifold, large-scale distances are determined
by small-scale ones through the lengths of paths.
However Definition \ref{d:closetoRn} does 
not impose any restrictions on distances larger that $2r$ in~$X$.
To rectify this, we need to make the metric `almost intrinsic'
as explained below.

\begin{definition}
\label{d:almost-intrinsic}
Let $X=(X,d)$ be a metric space and $\de>0$.
A \textit{$\de$-chain} in $X$ is a finite sequence
$x_1,x_2,\dots,x_N\in X$ such that $d(x_i,x_{i+1})<\de$ for all $1\le i\le N-1$.
A sequence $x_1,x_2,\dots,x_N\in X$ is said to be \textit{$\de$-straight}
if 
\be
\label{e:delta-straight}
d(x_i,x_j)+d(x_j,x_k)<d(x_i,x_k)+\de
\ee
for all $1\le i<j<k\le N$.
We say that $X$ is \textit{$\de$-intrinsic}
if for every pair of points $x,y\in X$ 
there is a $\de$-straight $\de$-chain
$x_1,\dots,x_N$ with $x_1=x$ and $x_N=y$.
\end{definition}

Clearly every Riemannian manifold (more generally, every length space)
is $\de$-intrinsic for any $\de>0$.
Moreover, if $X$ lies within GH distance $\de$ from a length space,
then $X$ is $C\de$-intrinsic.
In fact, this property characterizes $\de$-intrinsic metrics,
see Lemma \ref{l:intrinsic metric}.


\medbreak

In order to conveniently compare metric spaces at both
small scale and large scale, we need the notion
of quasi-isometry.

\begin{definition}\label{d:QI}
Let $X,Y$ be metric spaces, $\ep>0$ and $\la\ge 1$.
A (not necessarily continuous) map $f\co X\to Y$ 
is said to be a \textit{$(\la,\ep)$-quasi-isometry}
if the image $f(X)$ is an $\ep$-net in $Y$ and
\be\label{e:QI}
 \la^{-1} d_X(x,y)-\ep < d_Y(f(x),f(y)) < \la d_X(x,y)+\ep
\ee
for all $x,y\in X$, where $d_X$ and $d_Y$ denote the distances
in $X$ and $Y$, resp.
\end{definition}

Unlike the use of quasi-isometries in e.g.\ geometric group theory,
in this paper we consider quasi-isometries with parameters
$\ep\approx 0$ and $\la\approx 1$.
The quasi-isometry relation is almost symmetric:
if there is a $(\la,\ep)$-quasi-isometry from $X$ to $Y$,
then there exists a $(\la,{3}\la\ep)$-quasi-isometry from $Y$ to $X$.
We say that metric spaces $X$ and $Y$ are \textit{$(\la,\ep)$-quasi-isometric}
if there are $(\la,\ep)$-quasi-isometries in {both directions}.

The existence of a $(\la,\ep)$-quasi-isometry $f\co X\to Y$
implies that
\be\label{e:GH-via-QI}
 d_{GH}(X,Y) \le \tfrac12(\la-1)\diam(X)+{\tfrac 32}\ep .
\ee
See section \ref{subsec:gh} for the proof.

\medbreak
Now we formulate our main result.


\begin{theorem}
\label{t:manifold}
For every $n\in\N$
there exist {$\rhode_1=\rhode_1(n)>0$} 
and $\Cone=\Cone(n),\Ctwo=\Ctwo(n)>0$
such that the following holds.
Let {\ctext $r>0$, $X$ be a metric space with $\diam(X)>r$}, and
\begin{equation}
\label{delta-condition}
0<\de<\rhode_1r.
\end{equation}
Suppose that $X$ is $\de$-intrinsic and
$\delta$-close to $\R^n$ at scale $r$, 
see Definitions \ref{d:closetoRn} and \ref{d:almost-intrinsic}.
Then there exists a complete $n$-dimensional Riemannian manifold $M$
such that
\begin{enumerate}

\item $X$ and $M$ are $(1+\Cone\de r^{-1},\Cone\de)$-quasi-isometric. 
{Moreover, when the diameter of $X$ is finite, we have} 
\be
\label{e:t1-gh-estimate}
 d_{GH}(X,M) \le {2\Cone}\de r^{-1}\diam(X) .
\ee

\item The sectional curvature $\Sec_M$ of $M$ satisfies $|\Sec_M|\le \Ctwo\de r^{-3}$.

\item The injectivity radius of $M$ is bounded below by $r/2$.
\end{enumerate}
\end{theorem}

By following the steps of the proof, one can obtain
explicit formulas for the values of $\rhode_1(n)$, $\Cone(n)$, and $\Ctwo(n).$

The estimate \eqref{e:t1-gh-estimate} follows from the existence of
a $(1+\Cone\de r^{-1},\Cone\de)$-quasi-isometry from $X$ to $M$ due to \eqref{e:GH-via-QI}
and the fact that $\diam(X)>r$.
The proof of Theorem~\ref{t:manifold} is given in Section \ref{sec:manifold-proof}.

The quasi-isometry parameters and sectional curvature bound in Theorem \ref{t:manifold}
are optimal up to constant factors depending only on $n$,
see Remark \ref{rem: Thm 1 converse}.

{
\begin{remark}
The assumption that $X$ is $\de$-intrinsic in Theorem \ref{t:manifold}
is not crucial.
Without this assumption the following more technical variant of the theorem
holds:

{\it
If a metric space $X$ is $\de$-close to $\R^n$ at scale $r$,
where $\de/r$ is bounded above by a constant depending on $n$,
then there exists a complete (possibly not connected)
Riemannian $n$-manifold $M$ which satisfies properties (2) and (3) from
Theorem~\ref{t:manifold}
and approximates $X$ in the following sense:
there is a map $f\co X\to M$ such that
$$
  |d_M(f(x),f(y){)}-d_X(x,y)| < C\de
$$
for all $x,y\in X$ such that $d_X(x,y)<r$ or $d_M(f(x),f(y))<r$.
}

This variant follows from Theorem \ref{t:manifold} and the fact that
one can modify ``large'' distances in $X$ so that the resulting
metric is $C\de$-intrinsic and coincides with $d$ within balls of radius~$r$.
The new distances are measured along ``discrete shortest paths'' in $X$,
see \eqref{e:minchain} and Lemma \ref{l:minchain}
in section \ref{subsec:almost intrinsic}.

This procedure may split $X$ into several ``components'' with modified
distances between them being infinite.
In the original metric space these components are subsets separated
by distances greater that~$r$. They correspond to connected components of the
approximating manifold $M$.

For $\de$-intrinsic metrics, an approximation $f$ as above is
$(1+C\de r^{-1},C\de)$-quasi-isometry and vice versa.
This follows from Lemma \ref{l:QI-for-balls} 
and Lemma \ref{l:local-gh-implies-qi}, see section~\ref{sec:metric spaces}.
\end{remark}
}

Furthermore, Theorem \ref{t:manifold} gives a characterisation result for metric spaces that
GH approximate smooth manifolds with certain geometric bounds. 
The precise formulation is the following.
 
Let $\mathcal M(n,K,i_0,D)$ denote the class of $n$-dimensional compact
Riemannian manifolds $M$ satisfying $|\Sec_M|\leq K$, $\inj_M\ge i_0$, 
and $\diam(M)\leq D$. 
Denote by $\mathcal M_\ep(n,K,i_0,D)$  the class of metric spaces $X$
such that $d_{GH}(X,M)<\ep$  for some $M\in \mathcal M(n,K,i_0,D)$.
Also, let $\mathcal X(n,\delta,r,D)$  denote the class of metric spaces $X$
that are $\de$-intrinsic and
$\delta$-close to $\R^n$ at scale $r$, and satisfy $\diam(X)\leq D$.
Theorem \ref{t:manifold} has the following corollary that concerns neighbourhoods of smooth manifolds
and the class of metric spaces that satisfy a weak $\de$-flatness condition in the scale
of injectivity radius and a strong $\de$-flatness condition in a small scale $r$. 

\begin{corollary}\label{cor:compact class}
For every $n\in\N$ there exist ${\sigma_1}={\sigma_1}(n)>0$ 
and $\Cfour=\Cfour(n),\Cfive=\Cfive(n)>0$ such that the following holds.
Let $K,i_0,D>0$ and assume that $i_0<\sqrt{{\sigma_1}/K}$.
Let  $\de_0=Ki_0^3$, 
$0<\de<\de_0$, and $r=(\de/K)^{\frac 1 3}$.
Let $\mathcal X$ be the class of metric spaces
defined by
$$
 \mathcal X := \mathcal X(n,\de,r,D)\cap \mathcal X(n,\de_0,i_0,D) .
$$
Then
\be\label{e:class inclusion}
 \mathcal M_{\ep_1}(n,K/2,2i_0,D-\de) \subset \mathcal X 
 \subset \mathcal M_{\ep_2}(n,\Cfour K,i_0/{2},D)
\ee
where $\ep_1=\de/6$ and  $\ep_2=\Cfive DK^{1/3}\de^{2/3}$.
\end{corollary}


{For a metric space $X$,
the first inclusion in \eqref{e:class inclusion} means that
the condition
$X\in \mathcal X$ is necessary for $X$ to approximate
a manifold  from $\mathcal M(n,K/2,2i_0,D-\de)$ with accuracy $\ep_1$. 
Likewise, the second inclusion in \eqref{e:class inclusion}
says that $X\in \mathcal X$ is a sufficient condition for $X$
to approximate a manifold from $\mathcal M(n,\Cfour K,i_0/{2},D)$  with  accuracy~$\ep_2$.
}

The optimal values of $\ep_1$ and $\ep_2$ in Corollary \ref{cor:compact class}
remain an open question.
The proof of Corollary \ref{cor:compact class} is given
{in Section \ref{sec:proof-corollaries}.}
It is based on Theorem~\ref{t:manifold} and Proposition \ref{p:injrad} below.

{

\medskip

Another application of Theorem \ref{t:manifold} is the following
characterization of Alexandrov spaces with two-sided curvature bounds.

\begin{corollary}\label{cor:alexandrov}
For a complete geodesic metric space $X$ and $n\in\N$,
the following two conditions are equivalent:
\begin{enumerate}
\item There exists $K>0$ such that for all $x\in X$ and $r>0$,
\beq\label{gh dist BB}
  d_{GH}(B^X_r(x),B^n_r) \le Kr^3 .
\eeq
 \item
$X$ is an $n$-dimensional manifold, its metric has two-sided bounded 
curvature in the sense of Alexandrov,
and its injectivity radius is bounded away from~0.
\end{enumerate}
Furthermore, if $(1)$ holds then $X$ has Alexandrov curvature bounds
between $-\cone K$ and $\cone K$ and injectivity radius at least $1/(\ctwo\sqrt K)$,
where $\cone$ and $\ctwo$ are positive constants depending only on~$n$.
\end{corollary}

The proof of Corollary \ref{cor:alexandrov} is given in Section \ref{sec:proof-corollaries}.
We refer to \cite{BBI}, \cite{BGP}, \cite{BrHa} or \cite{BerNik}
for the definition and basic properties of Alexandrov
curvature bounds. 
Here we only mention the fact that finite-dimensional boundaryless Alexandrov spaces
with two-sided curvature bounds are Riemannian manifolds 
with $C^{1,\alpha}$ metrics (\cite{Nik83}, see also \cite[Theorem 14.1]{BerNik}).

}

\subsubsection*{On the proof of Theorem \ref{t:manifold}}
In the proof of Theorem \ref{t:manifold}, $M$ is constructed as a
submanifold of a separable Hilbert space~$E$,
which is either $\R^N$ with a large $N$
(in case when $X$ is bounded) or $\ell^2$ endowed with the the standard $\|\,\cdotp \|_{\ell^2}$ norm.
However the Riemannian metric on $M$
is different from the one inherited from $E$.

We note that an algorithm based on Theorem \ref{t:manifold}, that 
also summarises some of the main objects used in the proof, is given in  
Section \ref{sec:constructive}, see also Figure \ref{fig: manifold}
in Section~\ref{sec:manifold-proof}. 

Here is the idea of the proof of Theorem \ref{t:manifold}.
Since the $r$-balls in $X$ are $GH$ close to the Euclidean ball $B_r^n$,
they admit nice maps ($2\de$-isometries) to $B_r^n$.
These maps can be used as a kind of coordinate charts for~$X$,
allowing us to argue about $X$ as if it were a manifold.
In particular, we can mimic the proof of Whitney embedding theorem 
(on the classical Whitney {\ctext embedding}, see \cite{W2,W3}).
If $X$ were a manifold, this would give us
a diffeomorphic submanifold of a higher-dimensional Euclidean space $E$.
In our case we get a set $\Sigma\subset E$ which is
a Hausdorff approximation of a submanifold $M\subset E$.
In order to prove this, we use Theorem \ref{t:surface} 
(see subsection \ref{sec:intro-submanifold} below)
which characterizes sets approximable by (nice) submanifolds.
We emphasize that the resulting submanifold $M\subset E$ is the image
of a Whitney embedding but not of a Nash isometric embedding \cite{Nash1,Nash2}.
As the last step of the construction (see section \ref{subsec:riemannian metric}),
we construct a Riemannian metric $g$ on $M$ so that a natural map from $X$ to $(M,g)$ is almost
isometric at scale $r$. 
The construction is explicit and can be performed in an algorithmic manner,
see section \ref{sec:constructive}.
Then, with the assumption that $X$ is $\de$-intrinsic,
it is not hard to show that $X$ and $(M,g)$ are quasi-isometric with small
quasi-isometry constants.

\begin{convention}
Here and later we fix the notation $n$ for the dimension of
a (sub)mani\-fold in question.
Throughout the paper we denote by $c,\,C$, $C_1,C_2,$ etc.,
various constants depending only on $n$
and, when dealing with derivative estimates, on the order of the derivative involved. 
The same letter $C$ can be used to denote different constants,
even within one formula. {The  constants with a number are used for two reasons:
First, to make it possible to compute the values of these constants when needed. Second, to make the presentation easier, so that the reader can
see what earlier estimates are involved in each formula. However, to keep the presentation simpler, {some} constants that are not needed
in later estimates or are not main in focus of the interest of this paper are not numbered.}
To indicate dependence on other parameters,  {when we introduce a new constant,} we use
notation like $C(M,k)$ or {$C_j{(M,k)}$} for numbers depending on manifold $M$ and number $k$. {The locations where the constants appear first time are listed in Table 2.}
{Constants that does not depend on any parameters, including the dimension $n$ of the manifold, are called universal constants. Most of the constants $C_j$ depend on the intrinsic dimension $n$ and we do not usually indicate it (except in the introduction where main results are stated),
that is, we have $C_j=C_j(n)$,  $C(M,k)= C(M,k,n)$, etc. 
We note that several constants $C_j$ have an exponential dependency in $n$. One of the reasons for this is that in a manifold having a negative sectional curvature, e.g.\ in the hyperbolic space,
 a ball of radius $r$ contains $e^{cr/\delta}$  points that are $\delta$-separated. We emphasize that
 $n$ is the intrinsic dimension of the manifold, that is relatively small in several applications, and the constants does not depend on the dimension of an ambient space where a considered manifold may be {\ctext embedded} in.}

\end{convention}

\subsection{Manifold reconstruction and inverse problems}

Theorem~\ref{t:manifold} and  Corollary \ref{cor:compact class}
give quantitative estimates on how
 one can use discrete
metric spaces as models of Riemannian manifolds, for example
for the purposes of numerical analysis.
With this approach, a data set representing a Riemannian manifold is just a matrix 
of distances between points of some $\de$-net.
Naturally, the distances can be measured with some error.
In fact, only `small scale' distances need to be known,
see Corollary \ref{cor 2:manifold} below.

The statement of Theorem~\ref{t:manifold} provides a verifiable
criterion to tell whether a given data set approximates
any Riemannian manifold (with certain bounds
for curvature and injectivity radius).
See section \ref{subsec:algorithm-gh} for an explicit algorithm.

The proof of Theorem \ref{t:manifold} is constructive.
It provides an algorithm, although a rather complicated one, 
to {\it construct} a Riemannian manifold approximated
by a given discrete metric space~$X$.
See section \ref{sec:constructive} for an outline of the algorithm.

Next we formulate results that describe properties of
the manifold $M$ constructed from data $X$ that approximates 
some smooth manifold $\tilde M$ and discuss how this result
is used in inverse problems.

\subsubsection{Reconstructions with data that approximate a smooth manifold.}

When dealing with inverse problems,
it is assumed that the data set $X$ comes from
some unknown Riemannian manifold $\tilde M$, and moreover some {\it a priori} 
bounds on the geometry of this manifold are given.
Applying Theorem \ref{t:manifold} to this data set
yields another manifold $M$ which is
$(1+C\de r^{-1},C\de)$-quasi-isometric to~$\tilde M$.
One naturally asks what information about the original manifold $\tilde M$
can be recovered, {in particular, if the topological and differentiable type of the manifold $\tilde M$ can be determined using the set $X$}.
An answer is given by the following proposition.

\begin{proposition}[{cf.\ Theorem 8.19 in \cite{Gr}}]
\label{p:diffeo}
There exist 
${\kappa_0}>0$ 
and $\cthree>0$ such that the following holds.
Let $M$ and $\tilde M$ be complete Riemannian $n$-manifolds 
with $|\Sec_M|\le K$ and $|\Sec_{\tilde M}|\le K$, 
where $K>0$.

Let $0<{\kappa}<{\kappa_0}$ and assume that $M$ and $\tilde M$
are $(1+{\kappa},{\kappa} r)$-quasi-isometric, where
$
  r < \min\{({\kappa}/K)^{1/2}, \inj_M, \inj_{\tilde M} \} .
$

Then $M$ and $\tilde M$ are diffeomorphic.
Moreover there exists
a bi-Lipschitz diffeomorphism $\Psi$ between $M$ and $\tilde M$
with bi-Lipschitz constant bounded by $1+\cthree{\kappa}$,
\beq\label{c3}
(1+\cthree{\kappa})^{-1}d_{\tilde M}(\Psi(x),\Psi(y))\le d_{M}(x,y)\le (1+\cthree{\kappa})\,d_{\tilde M}(\Psi(x),\Psi(y)).\hspace{-2cm}
\eeq
\end{proposition}

We do not prove Proposition \ref{p:diffeo} because it is
essentially the same as Theorem 8.19 in \cite{Gr}
except that the approximation is quasi-isometric 
rather than GH.
To prove Proposition \ref{p:diffeo} one can apply the same 
arguments as in \cite[8.19]{Gr} using 
coordinate neighborhoods of size~$r$.
The estimates are not given explicitly in \cite{Gr}
but they follow from the argument.
These results can be regarded as quantitative versions
of Cheeger's Finiteness Theorem \cite{Ch},
see \cite[Ch.~10]{Pe} and \cite{Peters} for different proofs.

\begin{remark}
Using results of \cite{Anderson} one can show that
{{for any $\alpha<1$,}}
 $M$ and $\tilde M$
in Proposition \ref{p:diffeo} are close to each other in $C^{1,\alpha}$ topology.
However we do not know explicit estimates in this case.
\end{remark}

\subsubsection{An improved  estimate for the injectivity radius}
The injectivity radius estimate provided by Theorem~\ref{t:manifold}
is not good enough in the context of manifold reconstruction.
Indeed, in order to obtain a good approximation
one has to begin with a small $r$. 
(Recall that for Theorem \ref{t:manifold} to work,
$\de$ should be of order $Kr^3$ where $K$ is the curvature bound.)
However Theorem~\ref{t:manifold} guarantees only a lower
bound of order $r$ for $\inj_M$, so a priori
one could end up with an approximating manifold $M$
with a very small injectivity radius.
In order to rectify this we need the following result.

\begin{proposition}\label{p:injrad} 
There exists {a universal constant} $\cfour>0$
such that the following holds.
Let $K>0$ and let $M, \tilde M$ be complete $n$-dimensional
Riemannian manifolds with $|\Sec_M|\le K$ and $|\Sec_{\tilde M}|\le K$. 
\smallskip

1.  Let $x\in M$, $\tilde x\in\tilde M$, and
$
 0 < \rho \le \min\{\inj_{\tilde M}(\tilde x) , \tfrac\pi{\sqrt K} \} .
$
Then
\be\label{e:injrad-gh}
 \inj_M(x) \ge \rho - \cfour\cdot d_{GH}(B_\rho^M(x),B_\rho^{\tilde M}(\tilde x)) .
\ee

{
2. Suppose that $M$ and $\tilde M$ are $(1+\ep,\de)$-quasi-isometric
where $\ep,\de\ge 0$.
Then 
\beq\label{final injectivity radius estimate}
 \inj_M \ge (1-\cfour \ep) \min \{\inj_{\tilde M}, \tfrac\pi{\sqrt K} \} - \cfour\de.
\eeq
}
\end{proposition}
{This result is important for the inverse problems
of an approximate recovery of an unknown manifold $\widetilde M$.
It is often the case that we {\it a priori}
know bounds for the sectional curvature, injectivity radius, etc
of $\widetilde M$. On the other hand,  
any other manifold $M$ described by Theorem \ref{t:manifold}
 is
$(1+C\de r^{-1},C\de)$-quasi-isometric to $\widetilde M$.
Thus, the second part of Proposition \ref{p:injrad} gives a better
estimate for $\inj_M$ than Theorem \ref{t:manifold}.

The proof of Proposition \ref{p:injrad} is given in section~\ref{sec:injrad}.}

%
%
\subsubsection{An approximation result with only one parameter}

We summarize the manifold reconstruction features of 
Theorem \ref{t:manifold} in the following corollary where all approximations, errors in data, as well as
the errors in the reconstruction are given in terms of a single parameter $\hat\de$. 
Essentially, the corollary tells that a manifold $N$ can be approximately reconstructed 
from a $\hat \delta$-net $X$ of $N$
and the information about local distances between points of $X$ containing small errors. 
This type of results are useful e.g.\ in inverse problems discussed below.

 \begin{corollary}
\label{cor 2:manifold}
Let  $K>0$, $n\in\Z_+$ and {$(N,g)$}  be a compact $n$-dimensional manifold 
with sectional curvature bounded by $|\Sec_N|\leq  K$.
There exist $\delta_0=\delta_0(n,K)$ 
{and $\Csix=\Csix(n),\Cseven=\Cseven(n)>0$}
such that if $0<\hat \delta<\delta_0$  then the following holds:

Let $r=(\hat \delta/K)^{1/3}$ and suppose that the injectivity radius $\inj_N$ of $N$ satisfies $\inj_N>2r$. 
Also, let $X=\{x_j:\ j=1,2,\dots,J\}\subset N$
be a $\hat \delta$-net of $N$
 and 
 $\tilde d \co
 X\times X\to \R_+\cup\{0\}$
be {an approximate local distance function} that satisfies for all $x,y\in X$
\be
\label{delta-condition 0 B}
 |\tilde d(x,y)-d_N(x,y)|\leq \hat \delta,\quad \hbox{if } d_N(x,y)< r,
\ee
and
$$
 \tilde d(x,y)>r-\hat\de,\quad \hbox{if } d_N(x,y)\ge r.
$$
Given the set $X$ and the function $\tilde d$, one can 
effectively construct a  {compact, smooth $n$-dimensional Riemannian manifold $(M,g_M)$,  with distance function $d_M$. This manifold  approximates the manifold $(N,g)$ in the following way:}

\begin{enumerate}
\item 
There is a diffeomorphism $F:M\to N$ satisfying 
\beq
\label{Lip-condition}
\frac 1L\leq \frac{d_N(F(x),F(y))}{d_M(x,y)}\leq L,\quad \hbox{for all }x,y\in M,
\eeq
where 
$L=1+\Cseven K^{1/3}\hat \delta\,{}^{2/3}$.


\item 
{The} sectional curvature $\Sec_M$ of $M$ satisfies $|\Sec_M|\le \Csix K$.

\item The injectivity radius $\inj_M$ of $M$ satisfies
$$
 \inj_M\ge \min\{(\Csix K)^{-1/2}, (1-\Cseven K^{1/3}\hat \delta\,{}^{2/3})\inj_N\} .
$$
\end{enumerate}
\end{corollary}

The proof of Corollary \ref{cor 2:manifold}
is given in the end of Section \ref{sec:proof-corollaries}.

We call the function $\tilde d\co X\times X\to \R_+\cup\{0\}$, defined
on the $\hat\delta$-net $X$ and satisfying the assumptions of Corollary \ref{cor 2:manifold},
an approximate local distance function with accuracy $\hat \delta$. 
Many inverse problems can be reduced to a setting where one can 
 determine the distance function $d_N(x_j,x_k)$, with measurement errors $\epsilon_{j,k}$, in a discrete set $\{x_j\}_{j\in J}\subset N$.
 Thus, if the set $\{x_j\}_{j\in J}$ is $\hat \delta$-net in $N$, the errors $\epsilon_{j,k}$ satisfy
 conditions \eqref{delta-condition 0 B}, and
  $\hat \delta$ is small enough, then the diffeomorphism type of the manifold can be
 uniquely determined by Corollary \ref{cor 2:manifold}. Moreover, the bi-Lipschitz 
 condition (\ref {Lip-condition}) means that also the distance function can 
 be determined with small errors.
We emphasize that in \eqref{delta-condition 0 B} 
one needs to approximately know only the distances smaller than~$r=(\hat \delta/K)^{1/3}$.
The larger distances can be computed as in \eqref{e:minchain}.

\subsubsection{Manifold reconstructions in imaging and inverse problems.} 
Recently, geometric models have became an area of  focus of research in inverse
problems. As an example of such problems, one may consider an object with
a variable speed of wave propagation. The travel time of a wave
between two points defines a natural non-Euclidean distance between the
points. This is called the travel time metric and it corresponds
to the distance function of a Riemannian metric.
In many topical inverse problem the task  is to determine the Riemannian metric
inside an object from external measurements,
 see e.g.\ \cite{LU,LeU,PSU,PU,SU,SUV1,SUV2,UV}.
 These problems are the idealizations
of practical imaging tasks encountered in medical imaging or in Earth sciences.
Also, the relation of discrete and continuous models for these problems
is an active topic of research, see e.g.\ \cite{Beretta,Borcea1,Borcea2,KatKL}.
In these results  discrete models have been reconstructed from various types of measurement
data. However, a rigorously analyzed technique to construct a smooth manifold from these discrete models
to complete the construction has been missing until now.

In practice the {{measurement data always contain measurement errors and 
the amount of these data is limited.}}
This is why the problem of the approximate reconstruction of a Riemannian manifold
and the metric on it from discrete or noisy data is essential for several
 geometric inverse problems. Earlier, various regularization
techniques have been developed to solve noisy inverse problems in the PDE-setting, see 
e.g.\ \cite{Engl,MS},
but most of such methods depend on the used coordinates and, therefore, are not invariant.
{\ctext One of the purposes of this paper is to provide invariant tools for solving practical
imaging problems.}

An example of problems with limited  data is an inverse problem for the heat kernel, 
where the information about the
unknown manifold $(M, g)$ is given
in the form of discrete samples $(h_M(x_j, y_k, t_i))_{j,k\in J,i\in I}$ of the heat kernel $h_M(x, y, t)$,
satisfying
\ba
& &(\partial_t-\Delta_g)h_M(x, y,t)=0,\quad \hbox{on }(x,t)\in M\times\R_+,\\
& &h_M(x, y, 0)=\delta_y(x),
\ea
where the Laplace operator $\Delta_g$  operates in the $x$ variable, see e.g.\ \cite{KKL}.
Here $y_j=x_j$, where $\{x_j:\ j\in J\}$  is a finite
 $\ep$-net in an open set $\Omega \subset M$,
while $\{t_i:\ i\in I\}$ is in  $\ep$-net of the time interval $(t_0, t_1)$. It is also natural to
assume that one is given   
 measurements $h_M^{(m)}(x_j, y_k, t_i)$  of the heat kernel
 with errors satisfying $|h_M^{(m)}(x_j, y_k, t_i)-h_M(x_j, y_k, t_i)|<\ep$.    
Several inverse problems for {the} wave equation lead to a similar problem
for the wave kernel $G_M(x,y,t)$  satisfying  
\ba
& &(\partial_t^2-\Delta_g)G_M(x, y,t)=\delta_0(t)\delta_y(x),\quad \hbox{on }(x,t)\in M\times \R,\\
& &G_M(x, y, t)=0,\quad \hbox{for }t<0,
\ea 
see e.g.\ \cite{Iversen, KKL,Oksanen}.
In the case of complete data (corresponding to the case when $\ep$  vanishes), the inverse problem for heat kernel and wave kernel are 
equivalent to  the inverse interior spectral problem, see \cite{Mand}. In this problem
 one considers the eigenvalues
$\lambda_k$  of  $-\Delta_g$, counted by their multiplicity, and the corresponding
$L^2(M)$-orthonormal eigenfunctions,
$\varphi_k(x)$ that satisfy $$
-\Delta_g \varphi_k(x)=\lambda_k\varphi_k(x),\quad x\in M.$$


{
In the inverse interior spectral problem one
assumes that  we are given  
approximations  $\tilde \lambda_k,$ $k=0,1, 2,\dots, N-1$,
to the first $N$ smallest 
eigenvalues of $-\Delta_g$,
and  values $\varphi_k^{\prime}(x_j)$, at points $x_j$, of
approximations to the eigenfunctions $\varphi_k$.
Here $x_j$ form
an $\ep$-net $\{x_j:\ j\in J\}\subset \Omega$,
where $\Omega\subset M$  is open, and
$|\tilde \lambda_k -\lambda_k| \le \ep$ and
$|\varphi_k^{\prime}(x_j)-\varphi_k(x_j)|<\ep$.

  It is shown in \cite{BosiKL}
that 
these data determine 
a metric space $(X, d_X)$ which is a $\delta-$approximation (in the Gromov-Hausdorff distance) to
the unknown manifold $M$, where $\delta= \delta(\ep,N; \Omega)$
 tends to $0$
as $\ep \to 0$ and $N\to \infty$. It may be noted that the earlier works  
  \cite{AKKLT, KatKL} dealt with similar approximations (under other
geometric conditions)
for the case of 
manifolds with boundary and the Laplace operators with some classical boundary conditions.

Returning to the case when $M$ has no boundary,
Theorem \ref{t:manifold} completes the solution
of the above inverse problems by constructing a smooth manifold that approximates $M$.}

%
%
%
%

\subsection{Interpolation of manifolds in Hilbert spaces}
\label{sec:intro-submanifold}

As already mentioned, in the proof of Theorem \ref{t:manifold}
we need to approximate a set in 
a Hilbert space by an $n$-dimensional
submanifold (with bounded geometry).
At small scale, the set in question should be
close to affine subspaces in the following sense.

\begin{definition}\label{d:delta-flat}
Let $E$ be a Hilbert space, $X\subset E$, $n\in\N$ and $r,\de>0$.
We say that $X$ is \textit{$\delta$-close to $n$-flats at scale $r$} if
for any $x\in X$, there exists an $n$-dimensional affine space $A_x\subset E$ through
$x$ such that
\be \label{affine1}
d_H(X \cap B^E_r({x}),\, A_x \cap B^E_r(x)) \leq  \delta.
\ee
\end{definition}

To formulate our result for the sets in Hilbert spaces, we recall some definitions.
By a \textit{closed submanifold} of a Hilbert space $E$ we mean a finite-dimensional
smooth submanifold which is a closed subset of $E$.

Let $M\subset E$ be a closed submanifold.
{The \textit{reach} of $M$, denoted by $\Reach(M)$, 
is the supremum of all $r>0$
such that for every $x\in {\mathcal U}_r(M)$ there exists a unique nearest point in $M$.}
We denote this nearest point by $P_M(x)$ and
refer to the map $P_M\co {\mathcal U}_r(M)\to M$ as the \textit{normal projection}.

{For $x\in M$ we denote by $T_xM$ the tangent space of $M$ at~$x$.
The tangent space is regarded as an affine subspace of $E$
containing $x$.
We denote by $\vec T_xM$ the linear subspace of $E$ parallel to $T_xM$.
}

\begin{theorem}
\label{t:surface}

For every $n,k\in\N$
there exist  {positive  constants ${\sigma_2}$, $\Ceight$, $\Cnine$ depending only on $n$,
and a positive constant $\Cten({n,k})>0$}
such that the following holds.
Let $E$ be a separable Hilbert space,
$X\subset E$, $r>0$ and
\begin{equation}
\label{delta-condition 0}
0<\de<{\sigma_2}r.
\end{equation}
Suppose that $X$
is $\delta$-close to $n$-flats at scale $r$ (see Def.\ \ref{d:delta-flat}).
Then there exists a closed $n$-dimensional smooth
submanifold $M\subset E$ such that:

\begin{enumerate}

\item $d_H(X,M)\le 5\de$.

\item  The second fundamental form of $M$ at every point is bounded by $\Ceight \de r^{-2}$.

\item  
{$\Reach(M)\ge r/3$.}

\item  The normal projection $P_M\co {\mathcal U}_{r/3}(M)\to M$
{is smooth and satisfies for all $x\in {\mathcal U}_{r/3}(M)$
\be
\label{e:thm1dxmPM}
 \|d^k_x P_M\| < \Cten(n,k)
\de r^{-k}, \qquad k\ge 2 ,
\ee
and 
\be\label{e:thm1dx1mPM}
 \|d_x P_M - P_{\vec T_yM}\| < \Cten(n,1)
 \de r^{-1}
\ee
where $y=P_M(x)$ and $P_{\vec T_yM}$ is the orthogonal projector to $\vec T_yM$.
}

\item The tangent spaces of $M$ approximate subspaces $A_x$ from  Def.\ \ref{d:delta-flat}
in the following sense. If $x\in X$ and $y=P_M(x)$, then the angle between $A_x$
and the tangent space $T_yM$ satisfies
\be
\label{e:thm1angle}
\angle(A_x,T_yM) < \Cnine \de r^{-1} .
\ee

\end{enumerate}
\end{theorem}

\begin{notation}
{In \eqref{e:thm1dx1mPM}, \eqref{e:thm1dxmPM} and throughout the paper, 
$d_x$ and $d^k_x$ denote the first and $k$th differentials of a smooth map
at a point~$x$.}
The norm of the $k$th differential is derived from the
{inner product} 
norm on $E$ in the standard way.
As usual, we define the $C^k$-norm of a map $f$ defined on an open set $U\subset E$, by
$$
 \|f\|_{C^k(U)} = \sup_{x\in U}\max_{0\le m\le k} \|d^m_x f\| 
$$
{where $d_x^0f=f(x)$}.


{The angle $\angle(A_1,A_2)$ between $n$-dimensional linear subspaces $A_1,A_2\subset E$
is defined by
\be\label{angle 1}
 \angle(A_1,A_2):=\max_{u_1}\min_{u_2} 
\left\{
 \angle(u_1,u_2) \ | \
  u_1\in A_1,u_2  \in A_2,\ u_j\not=0\right\}
\ee
where $\angle(u_1,u_2) = \arccos \frac{\langle u_1,u_2\rangle}{|u_1||u_2|}$
and $\langle\,\cdotp,\cdotp\rangle$ is the inner product in~$E$.
The angle between affine subspaces is defined as the angle between their parallel translates containing the origin.}
Note that if $A_1$  and $A_2$ are linear subspaces and $P_{A_1}$ and $P_{A_2}$
are orthogonal projectors onto $A_1$ and $A_2$, respectively,
then 
\be\label{angle 2}
\|P_{A_1}-P_{A_2}\| {= \sin \angle(A_1,A_2)} .
\ee
\end{notation}

The proof of Theorem \ref{t:surface} is given in Section
\ref{sec:submanifold-construction}.
{
An algorithm based on Theorem~\ref{t:surface}, that also
summarises the main objects used in its proof, is given in  
Section \ref{sec:constructive}, see also Figure~\ref{fig: surface}.
}

{

 The above question of {submanifold} interpolation has attracted recently much interest in geometry and machine learning. For example,  
 an interpolation problem similar to Theorem \ref{t:surface} has been considered in the recent paper of Kleiner and Lott concerning Perelman's proof of the geometrization
conjecture, see \cite[Lemma~B.2]{Kleiner-Lott}.
{Compared} to these results, our method provides
explicit geometric bounds for $M$ in claims (2)--(4) of Theorem \ref{t:surface}, whereas 
the bounds in \cite[Lemma~B.2]{Kleiner-Lott}
arise from a contradiction argument and 
have the form of an unknown $\ep=\ep(\de,\dots)$
which just goes to 0 along with~$\de$.
In particular, Theorem \ref{t:surface} provides
explicit curvature bounds that are linear in~$\de$.
This is essential in the proof of Theorem~\ref{t:manifold} as well as in applications.
}

In  Remark \ref{rem: Theorem 2 optimality} below we show that
the bounds in claims (2) and (3) in Theorem \ref{t:surface} are optimal, up to  constant factors
depending on $n$.
Thus Theorem \ref{t:surface} 
gives necessary and sufficient conditions (up to multiplication
of the bounds by a constant factor) for 
a set $X\subset E$ to approximate a smooth submanifold
with given geometric bounds.

  %
%
%

In order to approximate a submanifold $M$ as in Theorem \ref{t:surface},
the set $X$ must contain as many points as a $C\de$-net in~$M$.
This is an unreasonably large number of points when $\de$ is small.
The following corollary allows one to reconstruct $M$ from
a smaller approximating set. 
It involves two parameters $\ep$ and $\de$ where $\ep$ is a `density'
of a net and $\de$ is a `measurement error'.
Note that $\de$ may be much smaller than $\ep$.
A similar generalization is possible for Theorem \ref{t:manifold} but
we omit these details.

\begin{corollary}
For every $n\in\N$
there exists ${\sigma_2}={\sigma_2}(n)>0$ 
{and $\cfive=\cfive(n)>0$}
such that the following holds.
Let $E$ be a Hilbert space, $X\subset E$,
$0<\ep<r/10$ and $0<\de<{\sigma_2}r$.
Suppose that for every $x\in X$ there exists
an $n$-dimensional affine subspace $A_x\subset E$
such that the set $X\cap B_r(x)$ is within Hausdorff distance $\de$
from an $\ep$-net of the affine $n$-ball $A_x\cap B_r(x)$.

Then there exists a closed $n$-dimensional submanifold $M\subset E$
satisfying properties 2--4 of Theorem \ref{t:surface} {and
an $\ep$-net  $Y$ of~$M$
such that}
\beq\label{c5}
d_H(X,Y)\le \cfive\de.
\eeq

\end{corollary}

\begin{proof}[Proof sketch] 
Consider the set
$
 X' = \bigcup_{x\in X} (A_x \cap B_r(x)) \subset E .
$
A suitably modified version of Lemma \ref{l:small-angle}
implies that $\angle(A_x,A_y)<C\de r^{-1}$ for all $x,y\in X$
such that $|x-y|<r$.
It then follows that $X'$ is $C\de$-close to $n$-flats
at scale $r-C\de$.
Now the corollary follows from Theorem \ref{t:surface}
applied to $X'$.
\end{proof}

\subsection{Submanifold interpolation and Machine Learning}

{\ctext The construction of a manifold  that  approximates, in some suitable sense, the given data points is a classical problem of Machine Learning.}
{We emphasise that we consider reconstruction of manifolds which are either considered as (differentiable) Riemannian manifolds or embedded submanifolds of an Euclidean space but not   
immersed submanifolds of an Euclidean space (i.e., a submanifold that intersects itself)
that are outside the context of this paper.} 

Next we give
a short review on existing methods and discuss how Theorem \ref{t:surface} is applied for
problems of Manifold Learning.

\subsubsection{Literature on Submanifold interpolation}
The question of fitting a manifold to data has been of interest to data analysts and statisticians  of late. There are several results dealing exclusively with sample complexity such as \cite{Aamari,Wasserman, NarMit}. We will restrict our attention to results that provide an algorithm for describing a manifold to fit the data  together with upper bounds on the sample complexity. 

A  work in this direction, \cite{Wasserman3}, building over \cite{Ozertem11} provides an upper bound on the Hausdorff distance between the output manifold and the true manifold equal to $O((\frac{\log N}{N})^{\frac{2}{D+8}}) + \tilde{O}(\sigma^2\log (\sigma^{-1}))$. In order to obtain a Hausdorff distance of $c\ep$, one needs more than $\ep^{-D/2}$ samples, where $D$ is the ambient dimension. The  results of \cite{putative}  guarantee (for sufficiently small $\sigma$) a Hausdorff distance of $$Cd^{7} (\sigma \sqrt{D}) = O(\sigma)$$ with less than $$\frac{CV}{\omega_d( \sigma \sqrt{D})^{d}} = O(\sigma^{-d})$$ samples, where $d$ is the dimension of the submanifold, $V$ is in upper bound in the $d-$dimensional volume, and $\sigma$ is the standard deviation of the noise projected in one dimension.
The question of fitting a manifold $\MM_o$ to data with control both on the reach $\tau_o$ and mean squared distance of the data to the manifold was considered in \cite{FMN}. The paper \cite{FMN} did not assume a generative model for the data, and had to use an exhaustive search over the space of candidate manifolds whose time complexity was doubly exponential in the intrinsic dimension $d$ of $\MM_o$. In \cite{putative} the construction of $\MM_o$ has a sample complexity that is singly exponential in $d$, made possible by the generative model.  While \cite{FMN} did not specify the bound on $\tau_o$, beyond stating that the multiplicative degradation $\frac{\tau}{\tau_o}$ in the reach depends on the intrinsic dimension alone. In \cite{putative} this degradation is pinned down to within $(0, C d^7]$, where $C$ is an absolute constant and $d$ is the dimension of $\MM$.

{\ctext
There are also  methods which map high dimensional data points to low dimensional
piecewise linear manifolds. 
Cheng, Dey and Ramos present an algorithm \cite{cheng} to reconstruct a smooth $k$ dimensional manifold $\mathcal M$ embedded in a Euclidean space
from a sufficiently dense point sample on the manifold. The algorithm outputs a simplicial manifold that is
homeomorphic to $\mathcal M$ and close to $\mathcal M$ in Hausdorff distance (see also the related work using witness complexes in \cite{witness}). 
In recent work, Aamari and Levrard \cite{Aamari} derive optimal rates for the estimation of tangent spaces the second fundamental form, and the submanifold $\mathcal M$
given a sample drawn from a submanifold $\mathcal M$ of Euclidean space. Unlike this paper or \cite{Wasserman3, FMN, putative} they do not however provide a method to produce a single consistent manifold from finitely many samples.
In other  recent work \cite{boissonnat}, Boissonnat
 et al.  
present an algorithm for producing Delaunay triangulations of manifolds. Given a set of sample points and an atlas on a compact manifold, a manifold Delaunay complex is produced for a perturbed point set
provided the transition functions are bi-Lipschitz with a constant close to 1, and the
original sample points meet a local density requirement. The output complex is endowed with a piecewise-flat metric which is a close approximation of
the original Riemannian metric. This is similar to our present work, except that our metric is $C^\infty$, and not just piecewise linear.
}

\subsubsection{Literature on Manifold Learning}\label{sec:lit}

The following methods aim to transform data lying near a $d$-dimensional manifold
in an $N$ dimensional space
into a set of points in a low dimensional space close to a $d$-dimensional
manifold. During transformation all of them try to preserve some geometric
properties, such as appropriately measured distances between points of the original data set. Usually the Euclidean distance to the `nearest' neighbours of a point 
is preserved. In addition some of the methods preserve, for points farther
away, some notion of geodesic distance capturing the curvature 
of the manifold.

Perhaps the most basic of such methods is `Principal Component Analysis' (PCA), \cite{PCA1, PCA2}
where 
one projects the data points onto the span of the $d$  eigenvectors corresponding
to the top $d$ eigenvalues of the ($N\times N$) covariance matrix of the data points.

An important variation is the `Kernel PCA' \cite{kernel} where one defines a feature map
$\varphi(\cdot)$ mapping the data points into a Hilbert space called the feature space. A `kernel matrix' $K$ is built
whose $(i,j)^{th} $ entry is the dot product $\langle \varphi(x_i), \varphi(x_j)\rangle$
between the data points $x_i,x_j.$ From the top $d$ eigenvectors of this matrix, the corresponding
eigenvectors of the covariance matrix  of the image of the data points in the feature space can be computed. The data points are projected onto
the span of these
eigenvectors of this  covariance matrix  in the feature space.

In the case of `Multi Dimensional Scaling' (MDS) \cite{mds}, only pairwise distances between
points are attempted to be preserved. One minimizes a certain `stress function' which captures the total error in pairwise distances between the data points
and between their lower  dimensional counterparts. For instance, a raw stress
function could be $\Sigma (\|x_i-x_j\|-\|y_i-y_j\|)^2,$ where $x_i$ are the
original data points, $y_i,$ the transformed ones, and $\|x_i-x_j\|,$
the distance between $x_i, x_j.$

`Isomap' \cite{TSL} attempts to improve on MDS by trying to capture geodesic distances
between points while projecting. For each data point a `neighbourhood graph' 
is constructed using its 
$k$ neighbours ($k$ could be varied based on various criteria), the edges 
carrying the length between points. Now shortest distance between points is computed in the resulting global graph containing all the neighbourhood graphs 
using a standard graph theoretic algorithm such as Dijkstra's.
Let $D = [d_{ij}]$ be the $n\times n$ matrix of graph distances. Let $S = [d_{ij}^2]$ be the $n \times n$ matrix of squared graph distances. Form the matrix
$A =\frac{1}{2}HSH,$ where $H = I - n^{-1}\mathbf{1}\mathbf{1}^T$. The matrix $A$ is of rank $t < n$, where $t$ is the dimension of the manifold. Let $A^Y =\frac{1}{2}HS^YH$, where $[S^Y]_{ij} = \|y_i - y_j\|^2.$  Here the $y_i$ are arbitrary $t$-dimensional vectors. The embedding vectors $\hat{y}_i$ are chosen to minimize $\|A - A^Y\|$. The optimal solution is given by the eigenvectors $v_1, \dots, v_t$ corresponding to the $t$ largest eigenvalues of $A$. The vertices of the graph $G$ are embedded by the $t \times n$ matrix $$\hat{Y}=(\hat{y}_1, \dots, \hat{y}_n) = (\sqrt{\la_1}v_1, \dots, \sqrt{\la_t}v_t)^T.$$

`Maximum Variance Unfolding' (MVU) \cite{MaximumVariance} also constructs the neighbourhood graph as in the case of Isomap but tries to maximize distance between projected points keeping
distance between nearest points unchanged after projection. It uses semidefinite programming
for this purpose.

In `Diffusion Maps' \cite{diffusion}, a complete graph on the data points is built and each
edge is assigned a weight based on a gaussian:
$w_{ij}\equiv \exp({\frac{\|x_i-x_j\|^2}{\sigma ^2}}).$  
Normalization is performed on this matrix so that the entries in each row add up to $1.$
This matrix is then used as the transition matrix $P$ of a Markov chain.
$P^t$ is therefore the transition probability between data points in $t$ steps.
The $d$ nontrivial eigenvalues $\lambda _i$ and their eigenvectors $v_i$ of $P^t$ are computed and
the data is now represented by the matrix $[\lambda _1v_1, \cdots , \lambda _dv_d], $ with the row $i$ corresponding to data point $x_i.$


The following are essentially local methods of manifold learning in the sense that they attempt to preserve local properties of the manifold around a 
data point.

`Local Linear Embedding' (LLE) \cite{RS} preserves solely local properties of the data.
Let $N_i$ be the neighborhood of $x_i$, consisting of $k$ points. Find optimal weights $\hat{w}_{ij}$ by solving
 $ \hat{W} := \arg\min_W \sum_{i=1}^n \|x_i - \sum_{j=1}^n w_{ij}x_j\|^2,$
 subject to the constraints (i) $\forall i, \sum_j w_{ij} = 1$, (ii) $\forall i,j,  w_{ij} \geq 0,$
 (iii) $w_{ij} = 0 $ if $j \not\in N_i.$
 Once the weight matrix $\hat{W}$ is found a spectral embedding is constructed using it. More precisely,
 a matrix $\hat{Y}$ is is a $t \times n$ matrix constructed satisfying
 $ \hat{Y} = \arg\min_Y Tr( {YMY^T}),$
under the constraints $Y\mathbf{1} = 0$ and {$YY^T = nI_t$},  {where $M = (I_n - \hat{W})^T(I_n - \hat{W})$.} $\hat{Y}$ is used to get a $t$-dimensional embedding of the initial data.

In the case of the `Laplacian Eigenmap'  \cite{BN}, \cite{heat} again, a nearest neighbor graph is formed. 
The details are as follows.
Let $n_i$ denote the neighborhood of $i$. Let $W = (w_{ij})$ be a symmetric $(n \times n)$ weighted adjacency matrix defined by (i)  $w_{ij} = 0$ if $j$ does not belong to the neighborhood of $i$;
(ii)  $w_{ij} = \exp(\|x_i - x_j\|^2/{2\sigma^2}),$ if $x_j$ belongs to the neighborhood of $x_i$.
Here $\sigma$ is a scale parameter. Let $G$ be the corresponding weighted graph.
Let $D = (d_{ij})$ be a diagonal matrix whose $i^{th}$ entry is given by $(W \mathbf{1})_i$. The matrix
$L = D - W$ is called the  Laplacian of $G$. We seek a solution in the set of $t \times n$ matrices
$ \hat{Y} = \arg \min_{Y:YDY^T = I_t}\hbox{Tr}(YLY^T).$ The rows of $\hat{Y}$ are  given by solutions of the equation $Lv = \la D v$.

Hessian LLE (HLLE) (also called Hessian Eigenmaps) \cite{donoho} and `Local Tangent Space Alignment' (LTSA) \cite{ZZ} 
attempt to improve on LLE by also taking into consideration
the curvature of the higher dimensional manifold while preserving the local
pairwise distances. We describe LTSA below.

LTSA  attempts to compute coordinates of the low dimensional data points and align the tangent spaces in the resulting embedding.  
It starts with computing bases for the approximate tangent spaces at the datapoints $x_i$ by applying PCA on the neighboring data points.
The  coordinates of the low dimensional data points are computed 
by carrying out a further minimization $min_{Y_i,L_i}\Sigma _i \|Y_iJ_k -L_i\Theta _i\|^2$. Here $Y_i$ has as its columns, the lower dimensional vectors,
$J_k$ is a `centering'  matrix, $\Theta _i$ has as its columns 
the projections of the  $k$ neighbors onto the $d$ eigenvectors obtained from the PCA and $L_i$ maps these coordinates to those
of the lower dimensional representation of the data points.
The minimization is again carried out through suitable spectral methods.

The alignment of local coordinate mappings also underlies some other methods
such as `Local Linear Coordinates' (LLC) \cite{llc} and `Manifold Charting' \cite{mc}.

%

Each of the algorithms is based on strong domain based intuition and in general
performs well in practice {at least} for the domain for which it was originally intended. PCA is still competitive as a general method.

Some of the algorithms are known to perform correctly under the hypothesis that data lie on a manifold of  a specific kind. In Isomap and LLE, the manifold has to be an isometric embedding of a convex subset of Euclidean space. 
In the limit as number of data 
points tends to infinity, 
when the data approximate a manifold, then one can recover the geometry of this manifold by computing an approximation of the Laplace-Beltrami operator.
 Laplacian Eigenmaps and Diffusion maps rest on this idea.
LTSA works for parameterized manifolds 
and detailed error analysis is available for it.

\subsubsection{Theorems \ref{t:manifold} and \ref{t:surface}  and the problems of machine {reconstruction}}
\label{sec:learning}


The Theorem \ref{t:manifold}
addresses the fundamental question, when  
a given metric space 
$(X,d_X)$, corresponding to data points and their `abstract' mutual distances,
{\ctext approximates}  a Riemannian manifold with
a bounded sectional curvature and injectivity radius. In the context of Theorem \ref{t:manifold},
the distances are measured in intrinsic sense in $M$ and~$X$.

Theorem~\ref{t:surface} deals with approximating a subset of a Hilbert space $E$ satisfying certain 
local constraints by a manifold having bounded second fundamental form and {reach}. 
In the context of Theorem \ref{t:surface},
the distances are measured in extrinsic sense in $E$.
Such approximations have extensively been considered
in machine learning or, more precisely, manifold learning and non-linear dimensionality reduction, 
where the goal is to approximate the set of data lying in a high-dimensional space like $E$ 
by a submanifold in $E$ of a low enough dimension in order to visualize these data,
see e.g. references of Section~\ref{sec:lit}. 

The results of this paper provide for the observed data an abstract low-dim\-en\-sio\-nal representation of the intrinsic manifold
structure that the data may possess. In particular, the topology of the manifold structure is determined,
assuming that the sampling density has been sufficient.
As described in Section~\ref{sec:submanifold-construction}, the proof of
Theorem \ref{t:surface} is of a constructive nature and provides an algorithm
to perform such visualisation. Note that this algorithm starts with tangent-type planes
which makes it distantly similar to the LTSA method in machine learning, see e.g. \cite{MCT, ZZ}.
In paper \cite{FMN},  the authors provide a  method of visualization of a given data
using a probabilistic setting. In comparison, Theorem \ref{t:surface} helps us visualize data in a deterministic setting.

The results of this paper are also related to  dimensionality reduction considered extensively in machine learning, see e.g. \cite{BN, BN2,BN1}. Using the constructions of Section \ref{AWE}, we can associate with given data not only the metric structure but also
point measures. Combining this with the constructions of \cite{BIK}, one could analyse the approximate determination of the eigenvalues
and eigenfunctions of a manifold that approximates the  data set.

\section{Approximation of metric spaces}
\label{sec:metric spaces}

In this section we collect preliminaries about GH 
and quasi-isometric approximation of metric spaces.
In subsections \ref{subsec:algorithm-gh} and \ref{subsec:algorithm-finddisc}
we present algorithms that can be used to verify
the assumptions of Theorems \ref{t:manifold} and~\ref{t:surface}.

\subsection{Gromov-Hausdorff approximations}
\label{subsec:gh}

Let $X$ be a metric space.
Recall that the Hausdorff distance between sets $A,B\subset X$
is defined by
\be\label{dist H}
 d_H(A,B) = \inf\{r>0: A\subset {\mathcal U}_r(B)\text{ and } B\subset {\mathcal U}_r(A)\}
\ee
where ${\mathcal U}_r$ denotes the $r$-neighborhood of a set.

The Gromov-Hausdorff (GH) distance $d_{GH}(X,Y)$ between
metric spaces $X$ and $Y$ is 
{the infimum of all $\ep>0$ such that}
there exists a metric space $Z$ and subsets $X',Y'\subset Z$ 
isometric to $X$ and $Y$, resp., such that
$d_H(X',Y')<\ep$. 
One can always assume that $Z$ is the disjoint union
of $X$ and $Y$ with a metric extending those of $X$ and $Y$.
The pointed GH distance between pointed metric spaces $(X,x_0)$ and $(Y,y_0)$
is defined in the same way with an additional requirement 
that $d_Z(x_0,y_0)<\ep$.
{See e.g.\ \cite[\S1.2 in Ch.~10]{Pe} or \cite{BBI} for details.}

\begin{example}[Distorted net]\label{ex:distorted net}
Recall that a subset $S$ of a metric space $X$ is called an $\ep$-net
if ${\mathcal U}_\ep(S)=X$. Let $S$ be an $\ep$-net in $X$ and
imagine that we have measured the distances between points of $S$
with an absolute error~$\ep$. That is, we have a 
distance function $d'$ on $S\times S$ such that 
$|d'(x,y)-d(x,y)|<\ep$ for all $x,y\in S$.
Then the GH distance between $X$ and $(S,d')$ is bounded by~$2\ep$.
This follows from the fact that the inclusion $S\hookrightarrow X$
is an $\ep$-isometry from $(S,d')$ to $(X,d)$, see below.

Strictly speaking, the `measurement errors'
in this example may break the triangle inequality
so that $(S,d')$ is no longer a metric space. 
This can be fixed by adding $3\ep$ to all $d'$-distances.
\end{example}


\vfill

\medskip
{
\noindent
\begin{tabular}{|p{1cm}|p{8.5cm}|p{2cm}|}
\hline
& {\bf Table 1: Index of the key notations} &\\
\hline
Name & Description  & Defined in \\
\hline
$d_X$ &  $d_X(x,y)=\hbox{dist}(x,y)$ in metric space $X$ &Sec.\ \ref{section: Introduction}\\
$d_H$ & Hausdorff distance & (\ref{dist H}) \\
$d_{GH}$ & Gromov-Hausdorff distance &Subsec.\ \ref{subsec:gh} \\
$\sigma_1$ & Maximal ration of $\delta$ and $r$ in Theorem \ref{t:manifold} & Thm.\ \ref{t:manifold} \\
$K$ & Bound for sectional curvature, $|\Sec_M|\leq K$ &Cor.\ \ref{cor 2:manifold}\\
$\inj_M$ & injectivity radius of manifold  $M$ &Sec.\ \ref{section: Introduction}\\
$d^k$ & $k$-th derivative &Thm.\ \ref{t:surface}\\
$E$ & Euclidean space $\R^N$  or the Hilbert space $\ell^2$ &Sec.\ \ref{section: Introduction} \\
$\Reach$ & $\Reach(M)$ is the reach of submanifold $M\subset E$ &Sec.\ \ref{sec:intro-submanifold}\\
${\sigma_0}(n)$ & Maximal value for parameter $\delta$ when $r=1$ & (\ref{delta condition}) \\
$X_0$ & subset of $X$ or a maximal $r/100$ separated subset of $X$ &(\ref{X0 net})\\
$B_r^n$ & a ball of radius $r$ in ${\mathbb R}^n$& Sec.\ \ref{section: Introduction}\\ 
$B_r(x)$ & a ball either in $E$, or $X$ depending on the context& Sec.\ \ref{section: Introduction}\\  
$D_i^r$ & $D_i^r=B_r^n(p_i)\subset\R^n$ is a ball of radius $r$ centered at $p_i$ &
Subsec.\ \ref{subsec:app charts} \\
$D_i$ & $D_i=B_1^n(p_i)\subset\R^n$ is a ball of radius 1 centered at $p_i$ &Subsec.\ \ref{subsec:app charts} \\
${\mathcal U}_\de(A)$ & a $\de$-neighbourhood of the set $A$&Subsec.\ \ref{subsec Geometrization}\\
$A_x$ & Affine plane of dimension $n$ containing $x$ &  Prop. \ref{p:surface}\\
$P_i $ & Orthogonal projector $P_i=P_{A_{q_i}}$ onto affine plane ${A_{q_i}}$ & (\ref{PAx}) \\
$ \mu_i(x)$ & Smooth cut-off functions $\mu_i\co E\to[0,1]$ & (\ref{mu function})-(\ref{mu function 2})\\
$\varphi_i(x)$ &  Function $\varphi_i:E\to E$,  $\varphi_i(x)= \mu_i(x) P_i(x) + (1-\mu_i(x)) x$  & (\ref{e:def of phi_i})\\
$f_i(x)$ & $ f_i\co E\to E$, $f_i= \varphi_i\circ\varphi_{i-1}\circ\ldots\circ\varphi_1$  & (\ref{e:def of fi})\\
$f(x)$ & $f(x)=\lim_{i\to \infty} f_i(x)$ or $f(x)=f_{\max i}(x)$ &(\ref{e:def of f})\\
$M$ & Reconstructed submanifold $M=f({\mathcal U}_{1/5}(X_0))\subset E$ & (\ref{e:Mdefinition})\\
$A_{ij}$ &
Affine transition maps $A_{ij}\co\R^n\to\R^n$  & \eqref{e:transition2b} \\
$\Omega$  & $\Omega= \bigcup_i D_i\subset \R^n$ is a union of coordinate charts& Subsec.\ \ref{AWE}\\
$\Omega_0$  & $\Omega_0= \bigcup_i D_i^{1/10}\subset \R^n$ is a union of coordinate charts& Subsec.\ \ref{AWE}\\
${\phi}(x)$  & A smooth map
${\phi}\co\R^n\to \mathbb S^n$
& (\ref{smooth map phi})\\
$F_i(x)$ & Embedding $F_i:\Omega\to \R^{n+1}$,
$ F_i(x) =
  {\phi}(A_{ji}(x)-p_i)$
   & (\ref{e:Fi definition})\\
$F(x)$ & An embedding 
$
F\co\Omega\to \R^m $, $F(x)=(F_i(x))_{i=1}^N$ &(\ref{e:Fi definition})\\
$\Sigma_i$& $\Sigma_i=F(D_i^{1/10})$ are local coordinate patches & (\ref{Sigma})\\
$\Sigma_i^0$ & $\Sigma_i^0=F(D_i^{1/50})$ are small local coordinate patches & (\ref{Sigma})\\
$P_M(x)$ & A projection of $x$ to the nearest point of $M$ to $x$ &Sec.\ \ref{section: Introduction}\\
${{\bf F}_{(i)}}$ & Map $\bf F$ for the ball $B_1(q_i)$ given by GHDist  & (\ref{GHDist map F} )\\
$V_i$ & $ V_i = P_M(\Sigma_i)$ coordinate neighborhoods on $M$  & (\ref{Vi set})\\
$\psi_i(x)$ & $\psi_i = P_M\circ F$  maps local chart  $D_i^{1/10}$  to $V_i\subset M$ &(\ref{e:psi-F-mod1})\\
$u_i(x)$ & $u_i=\tilde u_i/(\sum_i\tilde u_i)\in C^\infty_0(V_i)$ is a partition of unity on $M$ & (\ref{tilde ui})\\
$g$ & $g=\sum_i u_i g_i$ the metric constructed on $M$ &
(\ref{e: gi metric})-(\ref{e: g metric})\\
\hline
\end{tabular}
\medskip
}
\vfill

\newpage
{
{\bf Table 2:} Lemmas or formulas in which or after which $C_k$ is introduced.
\medskip

 \begin{tabular}{|p{1cm}|p{1.89cm}|p{1cm}|p{1.89cm}|p{1cm}|p{1.89cm}|}
\hline 
$C_k$  & Location  & $C_k$  & Location & $C_k$  & Location\\
\hline
$\Cone$ & \eqref {e:t1-gh-estimate}&
$\Ctwo$ & \eqref {e:t1-gh-estimate}&
$\Cfour $ & \eqref{e:class inclusion}\\
$\Cfive $ & \eqref{e:class inclusion}&
$\cone $ & \eqref{gh dist BB}&
$\ctwo $ & \eqref{gh dist BB}\\
$\cthree$ & \eqref {c3}&
$\cfour$ & \eqref {e:injrad-gh}&
$\Csix$ & \eqref{Lip-condition}\\
$\Cseven$ & \eqref{Lip-condition}&
$\Ceight $ & \eqref{delta-condition 0}&
$\Cnine $ & \eqref{e:thm1angle}\\
$\Cten$ & \eqref{delta-condition 0}&
$\cfive $ & \eqref{c5}&
$\csix $ & \eqref{c6}\\
$\Cthirteen$ & \eqref{e:rotation0}&
$\Celeven$& Lemma \ref{l:gh algorithm}&
$\Ctwelve$& Lemma \ref{Lem C12}\\
$\Csixteen $ & \eqref{e C 16}&
$\Cnineteen $ & \eqref{l:fdisplacement}&
$\Cseventeen $ & \eqref{e:fm-almost-projection}\\
$\Ctwenty $ & \eqref{e C20}&
$\Ctwentytwo$ & \eqref{e:transition}&
$\Ctwentythree $ & \eqref{e:3way}\\
$\Ctwentyfour$ & \eqref{e C24}&
$\Ctwentyfive$ & \eqref{e:Fsmooth}&
$\Ctwentysix$ & \eqref{e:Fbilip}\\
$\Ctwentyseven$ & \eqref{e C27}&
$\Ctwentyeight$ & \eqref{e:Sigma01}&
$\Ctwentynine$ & \eqref{e:Sigmai-near-flat}\\
$\Cthirtythree $ & \eqref{e:Sigma-near-flat0}&
$\ctwelve$  &Lemma \ref{l:submanifold}&
$\Cthirtysix$ & \eqref{e C36}\\
$\Cthirtyseven $ & \eqref{e:charts2}&
$\Cthirtyeight$  &Lemma \ref{l:charts}&
$\Cthirtynine$  &Lemma \ref{l:charts}\\
$\Cforty $ & \eqref{e C40}&
$\Cfortyone $ & \eqref{e: C41}&
$\Cfortyfour$ & \eqref{e:PMbounded}\\
$\Cfortyfive$ & \eqref{e:psibounded}&
$\Cfortysix$ & \eqref{e:sameprojection}&
$\Cfortyseven$ & \eqref{e C47}\\
$\Cfortyeight$ & \eqref{e:transdelta}&
$\Cfortynine$ & \eqref{e:tilde psibounded}&
$\Cfifty $ & \eqref{e: inverse tilde psi} 
\\
$\Cfiftyone $ & \eqref{e:ubounded}&
$\Cfiftytwo$ & \eqref{e:g-estimate}&
$\Cfiftythree $ & \eqref{e:g-estimate}\\
$\Cfiftyfour$ & \eqref{e:another-qi}&
$\Cfiftyfive$ & \eqref{e:another-qi2}&
$\Cfiftysix $ & \eqref{e C56}\\
$\Cfiftyseven$ & \eqref{e:distancedelta}&
$\cfourteen$ & \eqref{e:injMi}&
$\Clast$ & \eqref{e:sup  PM}\\
$\Cbracetthree $ & \eqref{e C61 interpolated}&
$\Cbracetfour $ & \eqref{added a1}&
$\cbracetone $ & \eqref{C2 estimate}\\
$\Cbracetfive $ & \eqref{e C65}&
$\cbracettwo$ & \eqref{e C66}&
$\cbracetthree$ & \eqref{e C66}
\\
\hline
\end{tabular}}
\bigskip

Let $X,Y$ be metric spaces,
{$f\co X\to Y$ a (not necessarily continuous) map, and $\ep>0$.
The \textit{distortion} of $f$, denoted by $\dis f$, is defined by
$$
 \dis f = \sup_{x,y\in X} |d_Y(f(x),f(y))-d_X(x,y)| ,
$$
and $f$ is called an \textit{$\ep$-isometry} if $\dis f<\ep$ and
$f(X)$ is an $\ep$-net in $Y$.}


If $d_{GH}(X,Y)<\ep$ then there exists a $2\ep$-isometry from $X$ to $Y$,
and conversely, if there is an $\ep$-isometry from $X$ to $Y$
then $d_{GH}(X,Y)<2\ep$.
{
Moreover,
\beq  \label{2.08.1a}
 d_{GH}(X,f(X))\le\frac12\dis f.
 \eeq
 Also, if $f(X)$ is $\ep-$net in $Y$, then
 \beq \label{2.08.1b}
d_{GH}(X,Y)\le\frac12\dis f+\ep.
\eeq
%
See \cite[\S7.3.3]{BBI} for proofs of these facts.}
They also hold for the pointed GH distance between pointed metric spaces
$(X,x_0)$ and $(Y,y_0)$, provided that $f(x_0)=y_0$.
Throughout the paper we use these properties without explicit reference.

{If $f$ is a $(\la,\ep)$-quasi-isometry (see Definition~\ref{d:QI}),
then $\dis f\le (\la-1)\diam(X)+\ep$. This together with (\ref{2.08.1a})-(\ref{2.08.1b})
implies \eqref{e:GH-via-QI}.}
{The next lemma is a variant of \eqref{e:GH-via-QI} for metric balls.

\begin{lemma}\label{l:QI-for-balls}
Let $f\co X\to {M}$ be a $(\la,\ep)$-quasi-isometry
and suppose {that $({M,d_M)}$ is a Riemannian manifold.} 
Then every $r$-ball in ${M}$ is within GH distance $2(\la-1) r +  {5}\ep$
from some $r$-ball in~$X$. More precisely, 
\be\label{e:GH-via-QI-for-balls}
 d_{GH}(B_r^X(x),B_r^{M}(y)) < 2(\la-1) r + {5}\ep
\ee
for all $x\in X$ and $y\in {M}$ such that $d_{M}(f(x),y)<\ep$.
\end{lemma}

\begin{proof}
Let $x$ and $y$ be as in the formulation. Then,
\be\label{e:ball-neighborhood}
 B^M_{r_1+r_2}(y) =\mathcal U_{r_2} (B^M_{r_1}(y))
\ee
for all $r_1,r_2>0$.

Fix $r>0$ and denote ${X_1}=B^X_r(x)$ and ${{M}_1}={B^{M}_r(y)}$.
Since $\diam({X_1})\le 2r$, 
the distortion of $f|_{{X_1}}$ is bounded by $2(\la-1)r+\ep$.
Hence by (\ref{2.08.1b}),
\be\label{e:QI-for-balls1}
 d_{GH}({X_1},f({X_1}))\le (\la-1)r + \ep/2
\ee
where $f({X_1})$ is regarded as a pointed metric space with distinguished point $f(x)$.
Now we estimate the Hausdorff distance $d_H(f({X_1}),{{M}_1})$ in ${M}$.
By \eqref{e:QI} and \eqref{e:ball-neighborhood}, 
$$
 f({X_1})\subset {B^{M}_{\la r+\ep}(f(x))\subset B^{M}_{\la r+2\ep}(y)}\subset\mathcal U_{\ep_1}({{M}_1}), \qquad \ep_1=(\la-1)r+2\ep.
$$
To prove that ${{M}_1}$ is contained in a suitable neighborhood of $f({X_1})$,
let $r_1=\la^{-1}r-3\ep$ and consider $z\in {B^{M}_{r_1}(y)}$.
Since $f(X)$ is an $\ep$-net in ${M}$, there is $x'\in X$ such {that $d_{M}(z,f(x'))<\ep$
and hence $d_{M}(f(x),f(x'))<r_1+2\ep$.}
This and \eqref{e:QI} imply that
$$
 d_X(x,x') < \la (d_{M}(f(x),f(x')) + \ep) < \la ( r_1 + 3\ep) = r,
$$
hence $x'\in {X_1}$.
Thus $ {B^{M}_{r_1}(y)} \subset \mathcal U_{\ep}(f({X_1})) $.
This and \eqref{e:ball-neighborhood} imply that
$$
{{M}_1}\subset\mathcal U_{\ep_2}(f({X_1})), \qquad
 \ep_2 = r-r_1 + \ep= (1-\la^{-1})r+4\ep .
$$
Thus $d_H(f({X_1}),{{M}_1}) \le \max(\ep,\ep_1,\ep_2) < (\la-1)r+4\ep$.
Since the Hausdorff distance is an upper bound for the GH distance, 
this and \eqref{e:QI-for-balls1} imply \eqref{e:GH-via-QI-for-balls}.
\end{proof}
}

\subsection{Almost intrinsic metrics}
\label{subsec:almost intrinsic}

Here we discuss properties of $\de$-intrinsic metrics
and related notions from
Definition \ref{d:almost-intrinsic}.
First observe that, if $x_1,x_2,\dots,x_N$
is a $\de$-straight sequence, then its 
`length' satisfies
\be\label{e:dist-chain}
 \sum_{i=1}^{N-1} d(x_i,x_{i+1}) \le d(x_1,x_N) + (N-2)\de .
\ee
This follows by induction from \eqref{e:delta-straight} 
and the triangle inequality.

The next lemma characterizes almost intrinsic metrics
as those that are GH close to Riemannian manifolds.
However manifolds provided by this lemma
may have extremely large curvatures and tiny injectivity radii.

\begin{lemma}\label{l:intrinsic metric}
Let $X$ be a metric space and $\de>0$.

1. If there exists a length space $Y$ such that $d_{GH}(X,Y)<\de$,
then $X$ is $6\de$-intrinsic.

2. Conversely, if $X$ is {compact} and $\de$-intrinsic, then there exists
a two-dimensional Riemannian manifold $M$ such that
\beq\label{c6}
d_{GH}(X,M)<\csix\de,
\eeq
 where $\csix$ is a universal constant.
\end{lemma}

\begin{proof}
1. By the definition of the GH distance, there exists a metric $d$
on the disjoint union $Z:=X\sqcup Y$ such that $d$ extends $d_X$ and $d_Y$
and $d_H(X,Y)<\de$ in $(Z,d)$.
Let $x,x'\in X$. Since  $d_H(X,Y)<\de$, there exist $y,y'\in Y$
such that $d(x,y)<\de$ and $d(x',y')<\de$. 
Connect $y$ to $y'$ by a minimizing geodesic and let
$y=y_1,y_2,\dots,y_N=y'$ be a sequence of points along this geodesic
such that $d(y_i,y_{i+1})<\de$ for all $i$.
For each $i=2,\dots,N-1$, choose $x_i\in X$ such that $d(x_i,y_i)<\de$.
Then $x,x_2,\dots,x_{N-1},x'$ is a $6\de$-straight $3\de$-chain
connecting $x$ and $x'$.
Since $x$ and $x'$ are arbitrary points of $X$, the claim follows.

2. Since we do not use this claim, we do not give a detailed proof of it.
Here is a sketch of the construction. 
First, arguing as in \cite[Proposition {7.5.5}]{BBI},
one can approximate $X$ by a metric graph.
If $X$ is $\de$-intrinsic, the graph can be made GH $\csix\de$-close to~$X$.
Consider a piecewise-smooth arcwise isometric embedding 
of the graph into $\R^3$ and let $M$ be a smoothed boundary
of a small neighborhood of the image.
Then $M$ is a two-dimensional Riemannian manifold which
can be made arbitrarily close to the graph and hence $\csix\de$-close to $X$.
\end{proof}

Now we describe a construction that makes a $C\de$-intrinsic metric
out of a metric which is $\de$-close to $\R^n$
at scale~$r$ (see Definition~\ref{d:closetoRn}).
More generally, let $X=(X,d)$ be a metric space 
in which every ball of radius $r$
is $\de$-intrinsic, where $r>\de>0$.
For $x,y\in X$, define the new distance $d'(x,y)$ by
\be
\label{e:minchain}
 d'(x,y) = \inf_{\{x_i\}} \biggl\{\sum_{i=1}^{N-1} d(x_i,x_{i+1}) :
  \ x_1=x, \ x_N=y
\biggr\}
\ee
where the infimum is taken over all 
finite sequences $x_1,\dots,x_N$ connecting $x$ to~$y$
and such that every pair of subsequent points $x_i,x_{i+1}$
is contained in a ball of radius~$r$ in $(X,d)$.

In order to avoid infinite $d'$-distances, we need to assume
that any two points can be connected by such a sequence.
If this is not the case, $X$ divides into components
separated from one another by distance at least~$r$.
For our purposes, such components are unrelated to one another
just like {disconnected} components of a manifold.

\begin{lemma}
\label{l:minchain}
Under the above assumptions, the function $d'$ given by \eqref{e:minchain}
is a {$8\de$-intrinsic} metric on $X$. 
Furthermore, $d$ and $d'$ coincide within any ball of radius~$r$.
\end{lemma}

\begin{proof}
The triangle inequality for $d$ implies that $d'$ is a metric,
$d'\ge d$, and $d'(x,y)=d(x,y)$ if $x$ and $y$ belong to an $r$-ball 
in $(X,d)$.
It remains to verify that $(X,d')$ is ${8}\de$-intrinsic.
Let $x,y\in X$ and let $x=x_1,\dots,x_N=y$ be a
sequence realizing the infimum in \eqref{e:minchain}
{with an error less than $\delta$.
Then
$$
 \sum d'(x_i,x_{i+1}) = \sum d(x_i,x_{i+1}) < d'(x,y)+\de ,
$$
hence the sequence $\{x_i\}$ is $\de$-straight with respect to~$d'$.
Recall that every pair $x_i,x_{i+1}$ belongs to an $r$-ball
and this ball is $\de$-intrinsic.
Hence there is a $\de$-straight $\de$-chain $S_i=\{z_j^{(i)}\}_{j=1}^{N_i}$
connecting $x_i$ to $x_{i+1}$ and contained in an $r$-ball.
Joining the sequences $S_i$ together yields a $\de$-chain $\{y_k\}_{k=1}^{N'}$,
$N'=\sum N_i$,
connecting $x$ to~$y$.

It suffices to prove that the sequence $\{y_k\}$ is $8\de$-straight with respect to~$d'$.
Note that the chains $S_i$ are $\de$-straight
with respect to both $d$ and $d'$ since the two metrics coincide in any $r$-ball.
Let $a_i=d'(x,x_i)$ for $i=1,\dots,N$.
Then for $i\le j$ we have
\be\label{e:minchain1}
 a_j-a_i \le d'(x_i,x_j) < a_j-a_i + \de
\ee
by the triangle inequality and the $\de$-straightness of $\{x_i\}$.
For $k\in\{1,\dots,N'\}$ define 
\be\label{e:minchaineq}
b_k = a_i+d'(x_i,y_k)
\ee
where $i=i(k)$ is the index such that $y_k$ belongs to $S_i$.
Note that, for $y_k\in S_i$,
\be\label{e:minchain2}
  d'(y_k,x_{i+1}) 
 < d'(x_i,x_{i+1}) - d'(x_i,y_k) + \de
 < a_{i+1} - b_k + 2\de
\ee
due to the $\de$-straightness of $S_i$ and \eqref{e:minchain1}.
We claim that
\be\label{e:minchain3}
  {\ctext b_m-b_k} - 2\de < d'(y_k,y_m) <   {\ctext b_m-b_k}+3\de
\ee
for all $k,m\in\{1,\dots,N'\}$ such that $k\le m$.
If both $y_k$ and $y_m$ are from one sub-chain $S_i$,
then \eqref{e:minchain3} follows from the $\de$-straightness of $S_i$.
Assume that $y_k\in S_i$ and $y_m\in S_j$ where $i<j$.
Then the triangle inequality
$$
 d'(y_k,y_m) \le d'(y_k,x_{i+1})+d'(x_{i+1},x_j) + d'(x_j,y_m)
$$
and relations $d'(y_k,x_{i+1})<a_{i+1} - b_k + 2\de$ (cf.~\eqref{e:minchain2}),
$d'(x_{i+1},x_j)< a_j-a_{i+1} + \de$ (cf.~\eqref{e:minchain1}),
and $d'(x_j,y_m)=b_m-a_j$ (cf.~\eqref{e:minchaineq}) imply the upper bound in \eqref{e:minchain3}.
Similarly, the lower bound in \eqref{e:minchain3} follows from the
triangle inequality
$$
 d'(y_k,y_m) \ge d'(x_i,x_{j+1}) - d'(x_i,y_k) - d'(y_m,x_{j+1})
$$
and relations $d'(x_i,x_{j+1}) \ge a_{j+1}-a_i$ (cf.~\eqref{e:minchain1}),
$d'(x_i,y_k) = b_k-a_i$ (cf.~\eqref{e:minchaineq}),
and $d'(y_m,x_{j+1}) < a_{j+1}-b_m+2\de$ (cf.~\eqref{e:minchain2}).
This finishes the proof of \eqref{e:minchain3}.

For $k,m,n\in\{1,\dots,N'\}$ such that $k\le m\le n$, \eqref{e:minchain3} implies that
$$
  -7\de < d'(y_k,y_m)+d'(y_m,y_n)-d'(y_k,y_n) < 8\de .
$$
Thus $\{y_k\}$ is a $8\de$-straight sequence and the lemma follows.
}
\end{proof}

The next lemma shows that if a map is almost
isometric at small scale then it is a quasi-isometry
with small constants.
It is used in the proof of Theorem~\ref{t:manifold}.

\begin{lemma}
\label{l:local-gh-implies-qi}
Let $r>{15}\de>0$.
Let $X$ and $Y$ be $\de$-intrinsic metric spaces
and $f\co X\to Y$ a map such that
$f(X)$ is a $\de$-net in $Y$ and
\be
\label{e:ghqi0}
 |d_Y(f(x),f(y)) - d_X(x,y)| < \de
\ee
for all $x,y\in X$ such that
$$
 \min\{d_X(x,y), d_Y(f(x),f(y)) \} < r .
$$
Then $f$ is a $(1+{10}r^{-1}\de,{3\de})$-quasi-isometry. 
\end{lemma}

\begin{proof} 
{Let $p,q\in X$ and $D=d_X(p,q)$.}
We have to verify that
\be\label{e:ghqi1}
 (1+{10}r^{-1}\de)^{-1} D-{3\de} < d_Y({f(p),f(q)}) < (1+{10}r^{-1}\de) D+{3\de} .
\ee
{Since $X$ is $\de$-intrinsic,
$p$ and $q$ can be connected by a $\de$-straight $\de$-chain,
see Definition \ref{d:almost-intrinsic}.
This chain contains a subsequence $p=x_1,x_2,\dots,x_N=q$
such that
$ r-\de < d_X(x_i,x_{i+1}) < r $
for all $i=1,\dots,N-2$ and $d_X(x_{N-1},q)<r$.}
Since the subsequence is also $\de$-straight,
by \eqref{e:dist-chain} we have
\be\label{e:ghqi2}
 \sum d_X(x_i,x_{i+1}) < D + (N-2)\de .
\ee
Since $d_X(x_i,x_{i+1})>r-\de$ for each $i\le N-2$,
the left-hand side {of \eqref{e:ghqi2}} is bounded below by
$(N-2)(r-\de)$.
Hence
\be\label{e:ghqi3}
 N\le (r-2\de)^{-1}D+2 
\ee
By \eqref{e:ghqi0} we have
$d_Y(f(x_i),f(x_{i+1})) < d_X(x_i,x_{i+1})+\de$
for all~$i$. 
Therefore
$$
  \sum d_Y(f(x_i),f(x_{i+1})) 
 < \sum d_X(x_i,x_{i+1}) + (N-1)\de
 < D+(2N-3)\de .
$$
by \eqref{e:ghqi2}.
By \eqref{e:ghqi3},
$$
D+(2N-3)\de
 < D + (2(r-2\de)^{-1} D+1)\de
 = (1+2(r-2\de)^{-1}\de)D+\de .
$$
{
Thus
\be\label{e:ghqi4}
 d_Y(f(p),f(q)) < (1+2(r-2\de)^{-1}\de)D+\de
\ee
Since $r-2\de>r/2$, the second inequality in \eqref{e:ghqi1} follows.
}

{To prove the first inequality in \eqref{e:ghqi1}, interchange the roles of $X$ and $Y$ and}
apply the same argument
to an `almost inverse' map $g\co Y\to X$ constructed as follows:
for each $y\in Y$, let $g(y)$ be an arbitrary point from the set $f^{-1}(B_\de(y))$.
{
This map satisfies the assumptions of the lemma with $3\de$ in place of $\de$ 
and $r-2\de$ in place of~$r$.
We may assume that $g(f(p))=p$ and $g(f(q))=q$, then \eqref{e:ghqi4} for $g$ takes the form
$$
 D <  (1+6(r-6\de)^{-1}\de) D' + 3\de
 < (1+10r^{-1}\de) D' + 3\de , \qquad D'=d_Y(f(p),f(q)) .
$$
This implies the first inequality in \eqref{e:ghqi1}
and the lemma follows.
}
\end{proof}

{
\subsection{GH approximations of the disc}
Here we prove a technical Lemma \ref{l:rotation}
about $\de$-isometries to subsets of $\R^n$. {For a matrix $A\in \R^{n\times n}$,
the norm $\|A\|$ is the operator norm of the map $A:\R^n\to \R^n$, unless stated otherwise.}
First we need the following estimate.
}

{
\begin{lemma}\label{l:gram}  
Let $\ep>0$ and $v_1,\dots,v_n\in\R^n$ be such that
$\left| |v_i|^2-1\right| < \ep$ and $|\langle v_i,v_j \rangle| < \ep$
if $i\ne j$, for all $i,j\in\{1,\dots,n\}$.
Define a linear map $L\co\R^n\to\R^n$ by
$ L(v) = (\langle v,v_i\rangle)_{i=1}^n $.
Then there exists an orthogonal operator $U\co\R^n\to\R^n$ such that
$ \|L-U\| < n\ep$.
\end{lemma}

\begin{proof}
We regard $L$ as an $n\times n$ matrix whose $i$th row consists of coordinates of $v_i$. 
The inner products $\langle v_i,v_j\rangle$ are elements of the matrix $LL^t$.
By assumptions of the lemma, all elements of the matrix $LL^t-I$ are bounded by~$\ep$.
Therefore the operator norm $\|LL^t-I\|$ is bounded by $n\ep$.
Decompose $L$ as $L=U_1DU_2$ where $U_1$ and $U_2$ are orthogonal matrices
and $D$ is a diagonal matrix with nonnegative entries.
Then $LL^t=U_1D^2U^{-1}_1$ and
$$
  \|L-U_1U_2\| = \|D-I\| \le \|D^2-I\| = \|U_1D^2U_1^{-1}-I\| = \|LL^t-I\| < n\ep .
$$
Thus the operator $U=U_1U_2$ satisfies the desired inequality.
\end{proof}
}

{
\begin{lemma}\label{l:rotation}
There is {a universal} constant $\Cthirteen>0$ such that the following holds.
Let $X$ be a metric space, $x_0\in X$,
and $f,g\co X\to\R^n$ maps 
with $f(x_0)=g(x_0)=0$.
Let $R\ge r\ge \de > 0$ and assume that
$f$ and $g$ {\ctext are}  $\de$-isometries to sets
$Y_1\subset\R^n$ and $Y_2\subset\R^n$, resp., 
such that $B_r^n\subset Y_i\subset B_R^n$ for $i=1,2$.

Then there exists an orthogonal operator $U\co\R^n\to\R^n$
such that 
\be\label{e:rotation0}
 |f(x)-U(g(x))| < \Cthirteen n Rr^{-1}\de
\ee
for all $x\in X$.
\end{lemma}

\begin{proof}
The statement of the lemma is scale invariant,
i.e., one can multiply the parameters $R,r,\de$, the maps $f,g$,
and the distances in $X$ by the same scale factor.
Thus we may assume that $r=1$.
Since $f$ and $g$ are $\de$-isometries, we have
$$
 {\bigl||f(x)-f(y)| - |g(x)-g(y)|\bigr| < 2\de}
$$
for all $x,y\in X$.
In particular $ \bigl||f(x)| - |g(x)|\bigr| < 2\de$ since $f(x_0)=g(x_0)=0$.
Hence
$$
 \bigl||f(x)|^2 - |g(x)|^2\bigr| \le 2\de (|f(x)| + |g(x)|)
$$
and
\begin{align*}
 {\bigl||f(x)-f(y)|^2 - |g(x)-g(y)|^2\bigr| }
 &\le 2\de (|f(x)-f(y)| + |g(x)-g(y)|) \\
&\le 2\de (|f(x)|+|f(y)| + |g(x)|+|g(y)|)
\end{align*}
for all $x,y\in X$.
These inequalities and the polarization identity
\be\label{e:Rn-polarization}
 \langle u,v\rangle = \tfrac 12 \bigl(|u|^2+|v|^2-|u-v|^2\bigr), \qquad u,v\in\R^n,
\ee
imply that
\be\label{e:rotation1}
 | \langle g(x),g(y)\rangle - \langle f(x),f(y)\rangle | 
 \le 2\de(|f(x)|+|f(y)|+|g(x)|+|g(y)|) 
\ee
for all $x,y\in X$.

Since $B_1^n\subset Y_1$,
there exist $x_1,\dots,x_n\in X$ such that 
$|f(x_i)-e_i|<\de$ for all~$i$, where $(e_i)_{i=1}^n$ is the standard basis of~$\R^n$.
Let $v_i=g(x_i)$, $i=1,\dots,n$.
Then, by \eqref{e:rotation1} applied to $x=x_i$ and $y=x_j$,
for all $i,j\in\{1,\dots,n\}$ we have
$$
 | \langle v_i,v_j\rangle - \langle f(x_i),f(x_j)\rangle | 
 \le 2\de (4+8\de) = (8+16\de)\de ,
$$
since $|f(x_i)|<1+\de$ and $|g(x_i)|<1+3\de$.
Therefore
$$
| \langle v_i,v_j\rangle - \langle e_i,e_j\rangle | 
 \le (10+17\de)\de =: \de_1 ,
$$
since $|\langle f(x_i),f(x_j)-\langle e_i,e_j\rangle|<2\de+\de^2$.
Thus the vectors $v_i$ satisfy the assumptions of Lemma \ref{l:gram} with $\ep=\de_1$.
As in Lemma \ref{l:gram}, define $L(v)=(\langle v ,v_i\rangle)_{i=1}^n$ for all $v\in\R^n$
and let $U$ be an orthogonal operator such that $\|L-U\|<n\de_1$.

Since $f(X)$ and ${g}(Y)$ are contained in $B_R^n$,
the right-hand side of \eqref{e:rotation1} is bounded by $8R\de$.
Hence, by \eqref{e:rotation1}  applied to $y=x_i$,
\be\label{e:rotation2}
 | \langle g(x),v_i)\rangle - \langle f(x),f(x_i)\rangle |  < 8R\de
\ee
for all $x\in X$ and $i\in\{1,\dots,n\}$.
We also have
$$
 | \langle f(x),f(x_i)\rangle - \langle f(x),e_i\rangle | \le |f(x)|\cdot |f(x_i)-e_i| \le R\de .
$$
This and \eqref{e:rotation2} imply that
\be\label{e:rotation3}
 | \langle g(x),v_i)\rangle - \langle f(x),e_i\rangle |  < 9R\de .
\ee
The term $\langle g(x),v_i)\rangle$ is the $i$th coordinate of the vector $L(g(x))$
(recall the definition of $L$ above),
and $\langle f(x),e_i\rangle$ is the $i$th coordinate of $f(x)$.
Hence \eqref{e:rotation3} implies that
$$
 | L(g(x)) - f(x) | < 9\sqrt n R\de .
$$
Since $\|L-U\|<n\de_1$, we also have 
$$
 | L(g(x)) - U(g(x)) | \le n\de_1 |g(x)| \le nR\de_1 .
$$
Therefore
$$
 | f(x) -  U(g(x)) | < (9\sqrt n \de + n\de_1) R \le 36nR\de
$$
since $\de_1\le 27\de$.
Thus \eqref{e:rotation0} holds with $\Cthirteen=36$.
\end{proof}
}

\subsection{Verifying GH closeness to the disc}
\label{subsec:algorithm-gh}


Here we present an algorithm that can be used to verify the main assumption
of Theorem \ref{t:manifold}. Namely, given a discrete metric space $X$, $n\in \N$
and $r>0$, one can approximately (i.e., up to a factor $C=C(n)$) 
find the smallest $\de$ such
that $X$ is $\de$-close to $\R^n$ at scale $r$ (see Definition \ref{d:closetoRn}).
Due to rescaling it suffices to handle the case $r=1$. 

Thus the problem boils down to the following: given a point $x_0\in X$, 
find approximately the (pointed) GH distance
between the metric ball $B^X_1(x_0)\subset X$ 
of radius 1 centered at $x_0$
and the Euclidean unit ball $B_1^n\subset\R^n$.
 {In the case when $X$ is finite, the following algorithm solves this problem.} 

\medskip

\underline{Algorithm GHDist:}
Assume that we are given $n$, the point $x_0\in X$ and the ball ${X_1}=B^X_1(x_0)\subset X$.
We regard ${X_1}$ as a metric space with metric $d=d_X|_{{X_1}\times {X_1}}$.
We implement the following steps: 
\begin{enumerate}
\item Let $x_1\in {X_1}$ be a point that minimizes $ |1 - d(x_0, x)|$ over all $x\in {X_1}$.
\item Given $x_1, x_2,\dots x_m$ for $m \leq n$,  we define the coordinate function
\be \label{20.7.2}
f_m(x)=\tfrac 12 \bigl(d(x,x_0)^2-d(x,x_m)^2+d(x_0,x_m)^2\bigr)
\ee

\item Given $x_1, x_2,\dots x_m$ and coordinate functions $f_1(x),f_2(x),\dots,f_m(x)$ for $m \leq n-1$, 
choose $x_{m+1}$ that is the solution of the minimization problem
 $$
 \min_{x\in {X_1}} K_m(x),\quad K_m(x)=\max(|1-{d(x_0, x)^2}|, |f_1(x)|,\dots,|f_{m}(x)| ).
 $$
 \item {
When $x_1, x_2,\dots ,x_n$ and coordinate functions $f_1(x),f_2(x),\dots,f_n(x)$
are determined, compute the map ${{\bf F}}\co {X_1}\to B_1^n$ defined by
\be\label{GHDist map F} 
{{\bf F}}(x)=P(f_1(x),\dots,f_n(x))
\ee
where $P$ is the map from $\R^n$ to $B_1^n$ defined as follows: $P(v)=v$ if $|v|\le 1$,
otherwise $P(v)=v/|v|$. 

 \item {Let  $\ell_1=\# X_1$ be the number of elements in $X_1$ and compute the values
 \begin{eqnarray}\nonumber
& & \de_1=\sup_{x,x'\in {X_1}} \bigg|d(x',x)-|{{\bf F}}(x')-{{\bf F}}(x)|\bigg|,\quad
\de_2=\sup_{y\in Y(\ell_1)} \inf _{x\in  {X_1}} |{{\bf F}}(x)-y| + \ell_1^{-1/n},\\
 & &\qquad \de_a=\max(\de_1,\de_2). \label{delta alpha estimate0}
 \end{eqnarray}
 where  $Y(\ell_1)=(h\Z^n)\cap B^n_1$ is the set of points in the unit ball whose coordinates are integer 
multiplies of $h=\ell_1^{-1/n}/\sqrt{n}$.}

Finally, the algorithm outputs the value of $\de_a$
and the map ${{\bf F}}$.}

\end{enumerate}
\begin{lemma}
\label{l:gh algorithm}
{
There is {a universal} constant $\Celeven>0$ such that the following holds.
Let ${X_1}$, $x_0$ be as in the above algorithm, $\de>0$,
and suppose that $d_{GH}({X_1},B_1^n)<\de$
where ${X_1}$ and $B_1^n$ are regarded as pointed metric spaces with distinguished points $x_0$ and 0, resp.}
Then 
\begin{enumerate}
 \item The output value $\de_a$ of the algorithm satisfies $\de_a<\Celeven n\de$.
 \item The {output} map {${{\bf F}}\co {X_1}\to B_1^n$
is a $\delta_a$-isometry with ${{\bf F}}(x_0)=0$}.
\end{enumerate}
\end{lemma}

\begin{proof} {First we make some preliminary considerations.
 Let $\de_1$  be as in (\ref{delta alpha estimate0}), and define 
 \begin{eqnarray}
& & \label{delta alpha estimate0 B}
\de_2'=\sup_{y\in B^n_1} \inf _{x\in  {X_1}} |{{\bf F}}(x)-y| ,\\
 & &\de_a'=\max(\de_1,\de_2'). \nonumber \end{eqnarray}
Here, $\de_a'$ is considered as a better approximation of the Gromov-Hausdorff distance $d_{GH}({X_1},B_1^n)$ than
$\de_a$, but it is computationally more difficult to obtain.

Next we show that
\beq\label{e: delta2' vs delta2}
\delta_2' \leq \delta_2 \leq 2 \delta_2'.
\eeq
To show this, denote 
 $\# X_1=\ell_1$.  By (\ref{delta alpha estimate0 B}), the unit ball $B^n_1$ 
 can be covered with  closed $\delta_2'$-balls which center points are in ${\bf F}(X_1)$. 
Considering 
their volumes we obtain an estimate $\ell_1(\delta_2')^n \geq 1$, or, equivalently, 
$\delta_2' \geq \ell_1^{-1/n}$. 
%
As $Y(\ell_1)$ is a $(\ell_1^{-1/n})$-net in the 
unit ball, we see that the supremums in the definitions of $\delta_2$ and 
$\delta_2'$ differ by no more than $\ell_1^{-1/n}$. This yields 
 the first inequality in \eqref{e: delta2' vs delta2}.
The second inequality in \eqref{e: delta2' vs delta2} follows from the fact that both the 
{\ctext supremum} in (\ref{delta alpha estimate0})
and $\ell_1^{-1/n}$ are no greater than $\delta_2'$. Thus \eqref{e: delta2' vs delta2} is valid
and we have
\beq\label{e: deltaa' vs deltaa}
\delta_a' \leq \delta_a \leq  2 \delta_a'.
\eeq
}

Now we are ready to prove the claims of the lemma.
{
By construction of ${{\bf F}}$ we have ${{\bf F}}(x_0)=0$
and the definition of $\de_a'$ implies that ${{\bf F}}$
is a $\de_a'$-isometry from ${X_1}$ to $B^n_1$.
This proves the second claim of the lemma.
It remains to prove the first one.






Consider the points $x_1,\dots,x_n$ constructed by the algorithm
and the corresponding functions $f_1(x),\dots,f_n(x)$,
see \eqref{20.7.2}.
Note that $f_i(x_i)=d(x_i,x_0)^2$.
Fix a $2\de$-isometry $h\co {X_1}\to B_1^n$ with $h(x_0)=0$ and
define functions $h_i\co {X_1}\to\R$, $i=1,\dots,n$, by
\ba
h_{i}(x):=\langle h(x),h(x_i)\rangle=\tfrac 12(| h(x)|^2 + |h(x_i)|^2 - |h(x)-h(x_i)|^2) .
\ea
Since $h$ is a $3\delta$-isometry, $h(x_0)=0$, $d(x,x_0)\le 1$ and $|h(x)|\le 1$ for all $x\in {X_1}$,
we have
$
 \left| d(x,x_0)^2-|h(x)|^2\right| <4\de
$
and
$
 \left| d(x,y)^2-|h(x)-h(y)|^2\right| <8\de
$
for all $x,y\in {X_1}$. Therefore
\be\label{h close to f}
|h_i(x)-f_i(x)|\leq  \tfrac 12 (4\de+4\de+8\de) = 8\de
\ee
for all $x\in {X_1}$, $i=1,\dots,n$.

Now we estimate $K_m(x_{m+1})$ for $m\in\{0,1,\dots,n-1\}$
assuming that $K_0$ is defined by $K_0(x)=|1-d(x,x_0)^2|$.
Since $m<n$, there exists $y_{m+1}\in \pd B_1^n$ orthogonal to all vectors
$h(x_1),\dots,h(x_m)$.
Since $h$ is a $2\de$-isometry, there exists $x_{m+1}'\in {X_1}$ such 
that $|h(x_{m+1}')-y_{m+1}|<2\delta$. 
This implies that ${\ctext d(x_0,x_{m+1}')}>|h(x_{m+1}')|-2\de>1-4\delta$ and therefore
\be \label{20.7.3}
|1 - d(x_0,x_{m+1}')^2| < 8\de.
\ee
Moreover,  for all $ i=1,2,   \dots, m$, we have
$$
 |h_i(x'_{m+1})| = |\langle h(x_i),h(x'_{m+1})\rangle| 
 = |\langle h(x_i),h(x'_{m+1})-y_{m+1}\rangle|  < 2\delta
$$
since $y_{m+1}$ is orthogonal to $h(x_i)$
and $|h(x'_{m+1})-y_{m+1}\rangle|<2\de$.
Hence, by \eqref{h close to f}, $|f_i(x'_{m+1})|<10\de$.
This and \eqref{20.7.3} imply that $K_m(x_{m+1}')<10\delta$.
Hence the minimizer $x_{m+1}$ of ${K_m}$ also satisfies ${K_m}(x_{m+1})<10\delta$.
Equivalently, $|f_i(x_{m+1})| < 10\de$ for $i=1,\dots,m$
and $f_{m+1}(x_{m+1})=d(x_0,x_{m+1})^2 > 1-10\de$.
Since $m+1$ is an arbitrary element of $\{1,\dots,n\}$,
we have shown that, for all $i,j\in\{1,\dots,n\}$,
$|f_i(x_j)| < 10\de$ if $i<j$ and $|f_i(x_i)-1| < 10\de$.

These inequalities and \eqref{h close to f} imply that
$$
 |\langle h(x_i),h(x_j)\rangle| = |h_i(x_j)| < |f_i(x_j)| + 8\de < 18\de \qquad\text{if $i<j$},
$$
and
$$
 \bigl||h(x_i)|^2-1\bigr| = |h_i(x_i)-1| < |f_i(x_i)-1| + 8\de < 18\de .
$$
Thus the vectors $v_i=h(x_i)$ satisfy the assumptions
of Lemma \ref{l:gram} for ${\ep=18\de}$.
Let $L\co\R^n\to\R^n$ be as in Lemma \ref{l:gram}, namely
$L(v)=(L_i(v))_{i=1}^n$ where $L_i(v)=\langle v,h(x_i)\rangle$. 
Then Lemma \ref{l:gram} provides an orthogonal operator $U\co\R^n\to\R^n$
such that $\|L-U\|\le 18n\de$.

For every $x\in {X_1}$ and $i\in\{1,\dots,n\}$ we have $L_i(h(x))=\langle h(x),h(x_i)\rangle=h_i(x)$.
This and \eqref{h close to f} imply that $|f_i(x)-L_i(h(x))| < 8\de$.
Thus for $f(x)=(f_i(x))_{i=1}^n$ we have
$
 |f(x) - L(h(x))| < 8n\de .
$
Since $\|L-U\|\le 18n\de$ and $|h(x)|\le 1$,
it follows that
$
 | f(x) - U(h(x)) | < 26n\de
$
for all $x\in {X_1}$.
Since ${{\bf F}}(x)=P(f(x))$ where $P\co\R^n\to B_1^n$ is a retraction
that does not increase distances, ${{\bf F}}(x)$ satisfies the same inequality:
\be\label{e:gram5}
| {{\bf F}}(x) - U(h(x)) | < 26n\de
\ee
for all $x\in {X_1}$.
Since $h$ is a $2\de$-isometry to $B_1^n$, so is $U\circ h$.
This and \eqref{e:gram5} imply that 
${{\bf F}}$ is a $54n\delta$-isometry from ${X_1}$ to~$B_1^n$.
Thus the first claim of the lemma holds with $\Celeven=54$.}
\end{proof}

The above lemma and \eqref{e: deltaa' vs deltaa} imply that the (pointed) Gromov-Hausdorff distance 
between ${X_1}$ and $B_1^n$ satisfies
\beq\label{delta alpha estimate}
{(2\Celeven n)^{-1}}\de_a\leq d_{GH}({X_1},B_1^n)\leq 2\de_a .
\eeq
Thus the algorithm {\it GHDist} gives the Gromov-Hausdorff distance of ${X_1}$ 
and $B_1^n$ up to a constant factor $2\Celeven$ depending only on dimension $n$.


\subsection{Learning the subspaces that approximate the data locally}
\label{subsec:algorithm-finddisc}

Let $X$ be a finite set of points in $E= \mathbb{R}^N$ and $X \cap B_1(x) := \{x, \tilde{x}_1, \dots, \tilde{x}_s\}$ be a set of points within a Hausdorff distance $\delta$ of some (unknown) unit $n$-dimensional disc $D_1(x)$ centered at $x$.  Here $B_1(x)$ is the set of points in ${\mathbb R}^N$ whose distance from $x$ is less or equal to $1$. We give below a simple algorithm that finds a unit $n$-disc centered at $x$ within a Hausdorff distance $C{n}\delta$ of ${X_1}:= X \cap B_1(x)$, where $C$ is {a universal} constant.

The basic idea is to choose a near orthonormal basis from ${X_1}$ where $x$ is taken to be the origin and let the span of this basis intersected with $B_1(x)$ be the desired disc.

\underline{Algorithm FindDisc:}
\begin{enumerate}
\item Let $x_1$ be a point that minimizes $ |1 - |x- x'||$ over all $x' \in {X_1}$.
\item Given $x_1, \dots x_m$ for $m \leq n-1$, choose $x_{m+1}$ such that $$\max(|1-|x- x'||, |\langle x_1/|x_1|, x'\rangle|, \dots, |\langle x_m/|x_m|, x'\rangle|)$$ is minimized among all $x' \in {X_1}$ for $x'= x_{m+1}$.
    \end{enumerate}
    {The output of the algorithm is the sequence $(x_1,x_2,\dots,x_n)$.}
    Let $\tilde{A}_x$  be the affine $n$-dimensional subspace containing $x, x_1, \dots, x_n$, and the unit $n$-disc $\tilde{D}_1(x)$ be $\tilde{A}_x \cap B_1(x)$. Recall that for two subsets $A, B$ of $\R^N$, $d_H(A, B)$ represents the Hausdorff distance between the sets. The same letter $C$ can be used to denote different constants,
even within one formula.

\begin{lemma}\label{Lem C12}
Suppose there exists an $n$-dimensional affine subspace $A_x$ containing $x$ such that $D_1(x) = A_x \cap B_1(x)$ satisfies $d_H({X_1}, D_1(x)) \leq \delta$. 
{{Suppose $0 < \delta < \frac{1}{20n}$.}} 
Then $d_H({X_1}, \tilde{D}_1(x)) \leq  \Ctwelve {n}\delta$.
\end{lemma}
\begin{proof}
Without loss of generality, let $x$ be the origin. Let $d(x, y)$ be used to denote $|x-y|$.
We will first show that for all $m \leq n-1$,
$$\max\left(|1-d(x, x_{m+1})|, \left | \left\langle \frac{x_1}{|x_1|}, x_{m+1}\right\rangle \right |, \dots, \left | \left\langle \frac{x_m}{|x_m|}, x_{m+1}\right\rangle \right | \right) < \delta.$$

To this end, {consider 
the function  $L_{m+1}:D_1(x)\to \R$, given by
\beq 
 {L_{m+1}(y)=}\max\left(|1-d(x, y)|, \left | \left\langle \frac{(x_1)}{|x_1|}, y\right\rangle \right |, \dots, \left | \left\langle \frac{(x_m)}{|x_m|}, y\right\rangle \right | \right),\label{eq:Hari1}
\eeq
and 
let $z_{m+1}\in D_1(x)$ be the point where $L_{m+1}$ obtains its minimum in $D_1(x)$.
The minimal value $L_{m+1}(z_{m+1})$
is $0$,} because the dimension of $D_1(x)$ is $n$ and there are only $m \leq n-1$ linear equality constraints. Also, the radius of $D_1(x)$ is $1$, so $|1 - d(x, z_{m+1})|$ has a value of $0$ where a minimum of (\ref{eq:Hari1}) occurs at $y = z_{m+1}$.
Since the Hausdorff distance between $D_1(x)$ and ${X_1}$ is less than $ \delta$ there exists a point $y_{m+1} \in {X_1}$ whose distance from $z_{m+1}$ is less than $\delta$. For this point $y_{m+1}$, we have $\delta$ greater than
\beq \label{eq:prec}
\max\left(|1-d(x, y_{m+1})|, \left | \left\langle \frac{x_1}{|x_1|}, y_{m+1}\right\rangle \right |, \dots, \left | \left\langle \frac{x_m}{|x_m|}, y_{m+1}\right\rangle \right | \right).\eeq Since
$$\max\left(|1-d(x, x_{m+1})|, \left | \left\langle \frac{x_1}{|x_1|}, x_{m+1}\right\rangle \right |, \dots, \left | \left\langle \frac{x_m}{|x_m|}, x_{m+1}\right\rangle \right | \right)
$$ is no more than the corresponding quantity in (\ref{eq:prec}), we see that for each $m+1 \leq n$,
$$\max\left(|1-d(x, x_{m+1})|, \left | \left\langle \frac{x_1}{|x_1|}, x_{m+1}\right\rangle \right |, \dots, \left | \left\langle \frac{x_m}{|x_m|}, x_{m+1}\right\rangle \right | \right) < \delta.$$
Let $\tilde V$ be an $N \times n$ matrix whose $i^{th}$ column is the column $x_i$. {We recall that the operator norm 
 of a matrix $Z$ is denoted by $\|Z\|$.}  For any distinct $i,j$ we have  $|\langle x_i , x_j \rangle|<\delta$, and for any $i$, $|\langle x_i, x_i\rangle - 1|<2\delta$, because $0 < 1-\de < |x_i| < 1$. For a matrix $X$, let $\|X\|_F$ denote its Frobenius norm. Therefore, {$$\|{\tilde V}^t {\tilde V} - I\| \leq \|{\tilde V}^t {\tilde V} - I\|_F \leq  \sqrt{(n^2 -n + 4n)\delta^2} < 2n\delta.$$}
Therefore, the singular values of $\tilde V$ lie in the interval 
$$
{I_C
=[1 - 4 n\de, 1 + 4 n\de].}$$
For each $i \leq n$, let $x'_i$ be the nearest point on $D_1(x)$ to the point $x_i$. Since the Hausdorff distance of ${X_1}$ to $D_1(x)$ is less than $\delta$, this implies that $|x'_i - x_i| < \delta$ for all $i \leq n$.
Let $\hat{V}$ be an $N \times n$ matrix whose $i^{th}$ column is $ x'_i$. Since for any distinct $i,j$, $|\langle x'_i, x'_j\rangle|<3\delta +\de^2 < 4\delta$, and for any $i$, $|\langle x'_i , x'_i \rangle - 1|<4\delta$,  
This means that the singular values of $\hat{V}$ lie in the interval $I_C$.

We shall now proceed to obtain an upper bound of $Cn\delta$ on the Hausdorff distance between ${X_1}$ and $\tilde{D}_1(x)$. Recall that the unit $n$-disc $\tilde{D}_1(x)$ is $\tilde{A}_x \cap B_1(x)$.  By the triangle inequality, since the Hausdorff distance of ${X_1}$ to $D_1(x)$ is less than $\delta$, it suffices to show that the Hausdorff distance between $D_1(x)$ and $\tilde{D}_1(x)$ is less than $ {6}n\delta.$

Let $x'$ denote a point on $D_1(x)$.  We will show that there exists a point $z' \in \tilde{D}_1(x)$ such that 
$|x' - z'| <  {4} n \delta.$ 

Let $\hat{V}\alpha = x'$.  {Assuming
that $\delta<1/(16{n}\delta)$ and using the bound on the singular values of $\hat{V}$,  
we have $|\alpha| \le 1+4 {n}\delta.$
Let $y' = \tilde{V}\alpha$. Then, by the bound on the singular values of $\tilde{V}$, we have $|y'| \leq( 1+4 {n}\delta)^2\leq  1+10 {n}\delta$.}
{Let $z' = \min(1-\de,|y'|)\,  |y'|^{-1}y' $.} By the preceding two lines,
$  z'$ belongs to $\tilde{D}_1(x).$ We next obtain an upper bound on $|x' -  z'|$
\begin{eqnarray} |x' - z'|  \leq  |x' - y'|\label{eq:1stterm} +|y' - z'| 
.\end{eqnarray}
We examine the first  term in the right side of (\ref{eq:1stterm})
\begin{eqnarray*}  |x' -  y'|  =  |\hat{V}\alpha - \tilde{V}\alpha|
                                                               \leq  \sup_i |x_i - x'_i|(\sum_i |\alpha_i|)
                                                               \leq  \delta {n}(1+10 {n}\delta).
\end{eqnarray*}
We next bound the {\ctext second term in} the right side of (\ref{eq:1stterm}). We have
\begin{eqnarray*} |y' - z'| \leq  {\delta |y'| 
                                                                            \leq 2\delta}. 
                                                                            \end{eqnarray*}
                                                                            Together, these calculations show that
$$|x' - z'| <   {4} {n}\delta.$$ 
A similar argument shows that if $ z''$ belongs to $\tilde{D}_1(x)$ then there is a point $p' \in D_1(x)$ such that 
$|p'-z''| < {6} {n}\delta$; the details follow.
 {Again, assume
that $\delta<1/(16{n}\delta)$ and 
let} $\hat{V}\beta = z''$. From the bound on the singular values of $\hat{V}$, $|\beta|<( 1+4 {n}\delta).$
Let $q' :=  \tilde{V}\beta$. Let $p' :=
{z' = \min(1-\de,|q'|)\,  |q'|^{-1}q'}$.
 {Then,
\begin{eqnarray*} |p' - z''| & \leq & |q' - z''| + |p' - q'|\\
                                   & \leq & |\tilde{V}\beta - V\beta| + |p' - \tilde{V}\beta|\\
                                   & \leq &  \sup_i |x_i - x'_i|(\sum_i |\beta_i|) + 5\delta n\\
                                   & \leq & \delta{n}(1+10{n}\delta) + 5 \delta{n}\\
                                   & \leq &  6 \delta{n}.
\end{eqnarray*}
This} proves that the Hausdorff distance between ${X_1}$ and $\tilde{D}_1(x)$ is bounded above by 
{$ \Ctwelve {n}\delta=6n\delta$.}
\end{proof}

{
\begin{remark}
\label{rem:complexity 1}
Let us consider the computational complexity of the above algorithms {\it GHDist} and
  {\it FindDisc}
 in terms the number of elementary operations one has to perform. 
{Here, we} count the computation of an algebraic function of
the distance $d_X(x_i,x_j),$ of two elements of $x_i,x_j\in X$ {and a computations of a piecewise analytic function  $t\mapsto f(t)$ of a real variable,} as one operation and the computation of
the inner product of two $m$-dimensional vectors in $\R^m$ as $m$ operations.
{We note that computing an inverse of an $n\times n$ matrix requires  a number of elementary operations that depends only on the intrinsic dimension $n$, and thus it requires in our notation convention  $C=C(n)$ elementary operations.}

Let us consider the algorithm {\it GHDist}. 
When the set $X_1$  has $\ell_1 =\# X_1$  elements,
the steps 1-3 of the algorithm {\it GHDist} need $C\ell_1$ steps, where $C$ is a generic constant depending on the
dimension $n$. The step 4 needs $C\ell_1$ steps. In  the step 5 the set
$Y(\ell_1)$ contains at most $C\ell_1$ points, and hence the step 5
needs at most $C\ell_1^2$ steps. Thus the computational complexity 
of {\it GHDist} is $C\ell_1^2$.

We assume that $E=\R^m$  and $X$  satisfies assumptions of Theorem \ref{t:surface}, so that 
 the set $X$ is $\delta$-close to $n$-flats in scale $r$.  
When the set $X$  has $\ell =\# X$  elements,
the algorithm  {\it FindDisc} minimises $n$ times functions that are the maximum of at most $n$
functions on involving inner products of $m$-dimensional vectors (i.e.\ points of $X$). Thus  the computational complexity of   {\it FindDisc} is $Cm\ell$.
\end{remark}
}

\section{Proof of Theorem \ref{t:surface}}
\label{sec:submanifold-construction}

The statement of Theorem \ref{t:surface} is scale invariant:
it does not change if one multiplies $r$ and $\de$
by $\lambda>0$ and applies a $\lambda$-homothety to
all subsets of~$E$.
Hence it suffices to prove the theorem only for $r=1$. {\ctext  We recall that we use the notation
$\hat \de_0=\sigma_2 r$.
Thus to prove   Theorem \ref{t:surface}  with $r=1$, it is enough to prove the following proposition
(where ${\sigma_2}$ is renamed to ${\ctext \hat \de_0}$):}

\begin{proposition}
\label{p:surface}
There exist {positive constants 
$\delta_0<1$, $\Ceight$, $\Cnine$ depending only on $n$, and $\Cten(k)>0$}
such that the following holds.
Let $E$ be a separable Hilbert space,
$X\subset E$ and $0<\de<{\ctext \hat \de_0} $.
Suppose that for every $x\in X$ there is
an $n$-dimensional affine subspace $A_x\subset E$
through $x$ such that
\be
\label{e:close-to-affine-ball}
d_H(X\cap B_1(x), A_x\cap B_1(x)) < \de .
\ee
Then there is a closed $n$-dimensional smooth
submanifold $M\subset E$ such that 

1. $d_H(X,M)\le 5\de$.

2. The second fundamental form of $M$ at every point is bounded by ${\Ceight}\de$.

3. 
{$\Reach(M)\ge 1/3$.}

4. The normal projection $P_M\co {\mathcal U}_{1/3}(M)\to M$
{
is smooth and for all $x\in {\mathcal U}_{1/3}(M)$
\be\label{e:dkPM}
\|d^k_x P_M\| < {\Cten(k)}\de, \qquad k\ge 2,
\ee
and
\be\label{e:d1PM}
\|d_x P_M-P_{\vec T_yM}\| < \Cten(1)\de, \qquad y=P_M(x).
\ee
}

5. Let $x\in X$ and $y=P_M(x)$. Then
\be
\label{e:Pangle}
\angle(A_x,T_yM) < \Cnine \de .
\ee

\end{proposition}

The proof of Proposition \ref{p:surface} occupies the rest of this section.
Let $X$ and $\{A_x\}_{x\in X}$ be as in the proposition.
Let 
\be\label{PAx}
P_{A_x}\co E\to A_x
\ee
 be the orthogonal
projection to $A_x$.
By $\vec A_x$ we denote the linear subspace parallel to $A_x$.
For $x\in X$ and $\rho>0$, 
we define $B_\rho^X(x)=X\cap B_\rho(x)$ and 
$D_\rho(x)=A_x\cap B_\rho(x)$.
In this notation, \eqref{e:close-to-affine-ball} takes the form
\be
\label{e:affine-delta}
 d_H(B_1^X(x),D_1(x)) < \de, \qquad x\in X .
\ee

In the sequel we assume that $\de$ is sufficiently small
so that the inequalities arising throughout the proof are valid, that is,
{we have 
\beq\label{delta condition}
 \delta <{\sigma_0}(n)\hbox{ \ctext and } r=1,
\eeq
where ${\sigma_0}(n)>0$ depends only on $n$. The number
${\sigma_0}(n)$ can be explicitly estimated by numbers $C_k$  
 appearing in 
the proof. 
}



\begin{lemma}
\label{l:small-angle}
Let $p,q\in X$ be such that $|p-q|<1$. Then
$\dist(q,A_p)<\de$ and
\item $\angle(A_p,A_q)<5\de$.
\end{lemma}

\begin{proof}
Since $q\in B_1^X(p)$, we have
$$
\dist(q,A_p)\le\dist(q,D_1(p))\le d_H(B_1^X(p),D_1(p))<\de
$$
by \eqref{e:affine-delta}.
It remains to prove the second claim of the lemma.

Let $z=P_{A_p}(\frac{p+q}2)$.
Then $|z-p|<\tfrac12$ and $|z-q|<\tfrac12+\de$
by the triangle inequality.
Define $B=A_p\cap B_{1/2-2\de}(z)$.
We claim that $\dist(y,A_q)<2\de$ for every $y\in B$.
Indeed, let $y\in B$. 
Then $|y-q|<1-\de$ and $|y-p|<1-2\de$.
The latter implies that $y\in D_1(p)$, hence by \eqref{e:affine-delta}
there exists $x\in X$ such that $|x-y|<\de$.
By the triangle inequality we have $x\in B_1^X(q)$, 
hence \eqref{e:affine-delta} implies that $\dist(x,A_q)<\de$.
Therefore
$
\dist(y,A_q)\le|y-x|+\dist(x,A_q)<2\de
$
as claimed.

Define a function $h\co \vec A_p\to\R_+$ by
$
 h(v) = \dist(z+v,A_q)^2 .
$
As shown above, $h(v)\le 4\de^2$
for all $v\in\vec A_p$ such that $|v|\le \frac12-2\de$.
The function $h$ is polynomial of degree~2, i.e.,
$
 h(v) = Q(v) + L(v) + h_0
$
where $Q$ is a (nonnegative) quadratic form, $L$ is a linear function
and $h_0=h(0)$.
Furthermore,
$$
 Q(v) = \sin^2\angle (v,\vec A_q) \cdot |v|^2
$$
for all $v\in\vec A_p$.
Let $\alpha=\angle(A_p,A_q)$,
and let $v_0\in\vec A_p$ be such that
$\angle (v_0,\vec A_q)=\alpha$ and $|v_0|=\frac12-2\de$.
Then
$$
 Q(v_0) = \frac{h(v_0)+h(-v_0)}2-h(0) \le 4\de^2
$$
since $h(\pm v_0)\le 4\de^2$ and $h(0)\ge 0$.
Thus
$
 \sin^2(\alpha) \cdot |v_0|^2 \le 4\de^2 ,
$
or, equivalently,
$$
 \sin\alpha \le 2\de (\tfrac12-2\de)^{-1} = 4\de(1-4\de)^{-1} .
$$
If $\de$ is sufficiently small, this implies the desired inequality $\alpha<5\de$.
\end{proof}

 Let $X_0$ be a maximal (with respect to inclusion)
$\frac1{100}$-separated subset of~$X$, {that is, a maximal subset $X_0\subset X$
satisfying 
\be\label{X0 net}
d_X(x,x')\ge \frac1{100}\quad\hbox{for all }x,x'\in X_0,\ x\not =x'.
\ee
Note} that $X_0$ is a $\frac1{100}$-net in $X$
and $X_0$ is at most countable.
Let $X_0=\{q_i\}_{i=1}^{|X_0|}$.
For brevity, we introduce notation $A_i=A_{q_i}$ and $P_i=P_{A_{q_i}}$.

Throughout the argument below we assume that $|X_0|=\infty$,
i.e.\ $X_0$ is a countably infinite set.
In the case when $X_0$ is finite the proof is the same
except that ranges of some indices should be restricted.

Assuming that $\delta<\frac1{300}$,
there is a number $N=N(n)$ such that every set
of the form $X_0\cap B_1(q_i)$ contains at most $N$ points.
This follows from the fact that this set is $\frac1{100}$-separated
and contained in the $\delta$-neighborhood of
a unit $n$-dimensional ball $D_1(q_i)$.

{Let us fix} a smooth function $\mu\co\R_+\to[0,1]$ such that
$\mu(t)=1$ for all $t\in[0,\frac13]$ and $\mu(t)=0$ for all $t\ge \frac12$.
{Below, we will use the function $\mu(t)=\alpha_{1/3,1/2}(t)$,  where
\be\label{mu function}
\alpha_{a,b}(t)=(\exp((t-a)^{-1})/ (\exp((t-a)^{-1})+
 \exp((b-t)^{-1}))\quad\hbox{for }t\in [a,b].
 \ee}
For each $i\ge 1$ define a function $\mu_i\co E\to[0,1]$ by
\be\label{mu function 2}
 \mu_i(x) = \mu (|x-q_i|) .
\ee
Clearly $\mu_i$ is smooth and
\beq\label{e C 16}
\max_{j\leq k}\|d_x^j\mu_i\|_{L^\infty(\R)}\leq \Csixteen(k)
\eeq
for every $k\ge 1$. {Here,
$\Csixteen(k)$ can be chosen uniformly over $n$  as by Lemma~\ref{thm:nest} in Appendix A,  the supremum of the  the $k$-th order derivative  $d_x^k\mu_i$, considered as a multilinear form, is attained at a derivative corresponding to vectors that are all equal, which in turn implies that for all $n\geq 2$, $\|d_x^k\mu_i\|_{L^\infty(\mathbb{R}^n)}$ is independent of $n$.}
Let $\varphi_i\co E\to E$ be a map given by
\be\label{e:def of phi_i}
 \varphi_i(x) = \mu_i(x) P_i(x) + (1-\mu_i(x)) x .
\ee
Now define a map $f_i\co E\to E$ by
\be\label{e:def of fi}
 f_i = \varphi_i\circ\varphi_{i-1}\circ\ldots\circ\varphi_1
\ee
for all $i\ge 1$, and let $f_0=id_E$.


For $x\in E$ and $i\ge 1$ we have
$f_i(x)=f_{i-1}(x)$ if $|f_{i-1}(x)-q_i|\ge\frac12$.
This follows from the relation $f_i=\varphi_i\circ f_{i-1}$
and the fact that $\varphi_i$ is the identity outside
the ball $B_{1/2}(q_i)$.

Let $U={\mathcal U}_{1/4}(X_0)\subset E$.
We are going to show that for every $x\in U$
the sequence $\{f_i(x)\}$ stabilizes and hence a map
$f=\lim_{i\to\infty} f_i$ is well-defined on~$U$.

Define $B_m=B_{1/4}(q_m)$ for $m=1,2,\dots$.
Note that $U=\bigcup_m B_m$.

\begin{figure}

\psfrag{1}{$p^-$}
\psfrag{2}{$p^+$}
\psfrag{4}{$J$}
\psfrag{3}{\hspace{-0.1cm}U}
\psfrag{5}{$V$}
\psfrag{6}{}

\figcommented{
}

\epscommented{
\includegraphics[width=10.5cm]{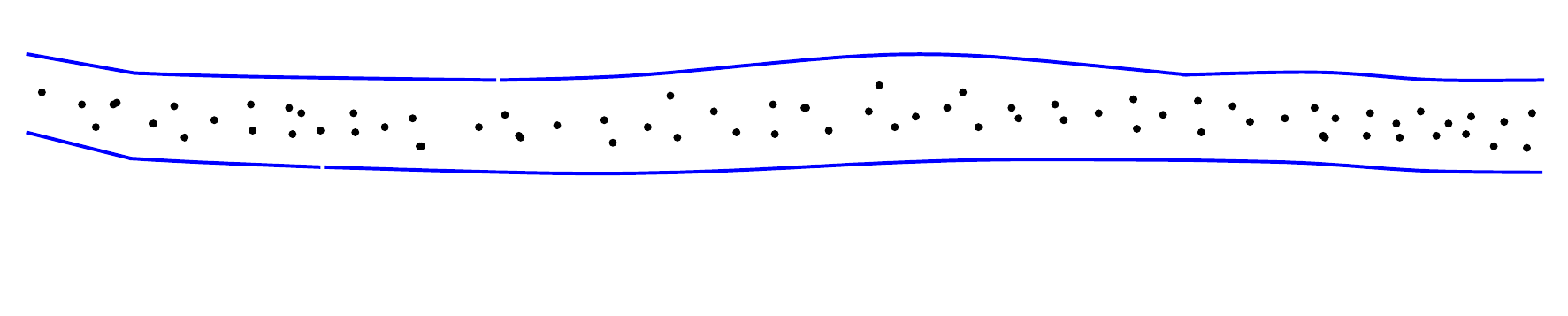}

\includegraphics[width=10.5cm]{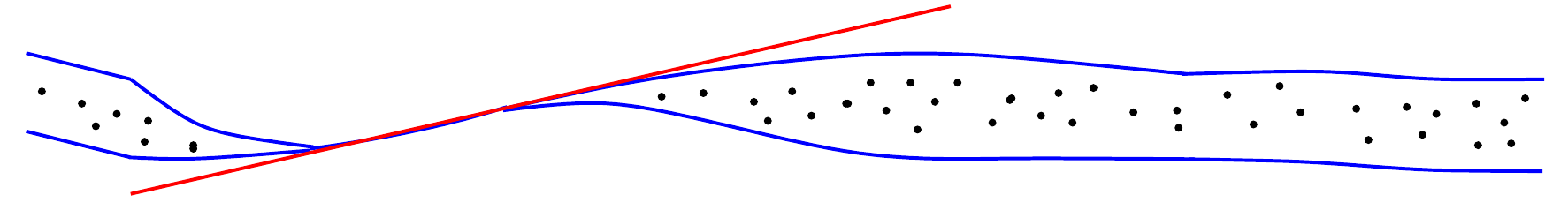}

\includegraphics[width=10.5cm]{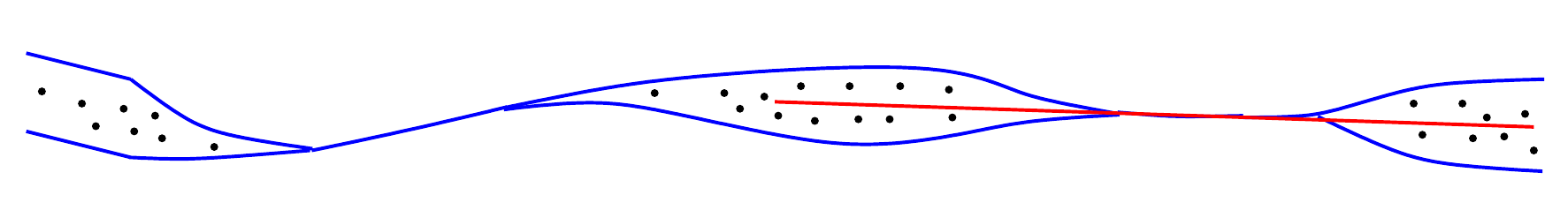}

\includegraphics[width=10.5cm]{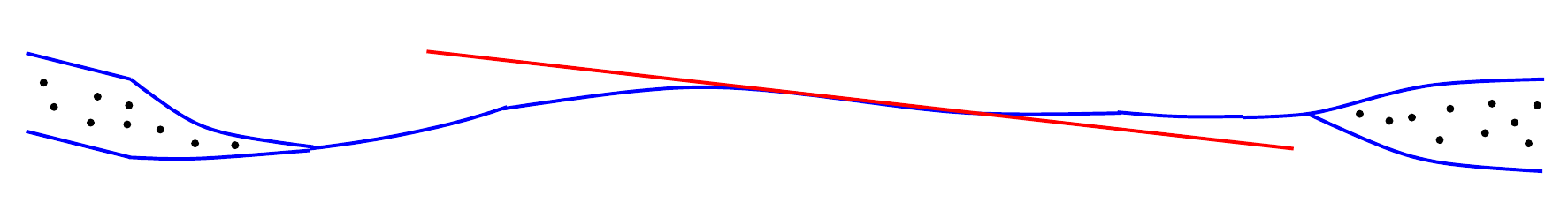}
}

\psfrag{1}{$y_0$}
\psfrag{2}{$\Sigma$}
\psfrag{3}{$\Sigma_1$}
\psfrag{4}{$y^\prime$}
\psfrag{5}{$\mu_{a}$}
\psfrag{4}{$W_{\nohat g}$}
\psfrag{5}{$\mu_{a}$}

\caption{A schematic visualisation of the interpolation algorithm Algorithm `SubmanifoldInterpolation' based on {\ctext Theorem \ref{t:surface}}.
In the figure, {\ctext on the top}, the black data points $X\subset E=\R^m$ have a $\delta$-neighbourhood
$U={\mathcal U}_\delta(X)$. The boundary of $U$ is marked by blue. In the figures below,
we determine, near points $x_i\in X$, $i=1,2,3$ the approximating $n$-dimensional planes $A_i$,
marked by red lines. Then we
map the set $U$ by applying to it iteratively
 functions $\varphi_i:E\to E$, defined in 
(\ref{e:def of phi_i}). The maps $\varphi_i$  are convex combinations of the projector $P_{A_i}$, onto $A_i$,
 and  the identity map. {\ctext The three figures on  the bottom} show the sets $\varphi_1(U)$,
 $\varphi_2(\varphi_1(U))$ and  $\varphi_3(\varphi_2(\varphi_1(U)))$, respectively. The limit
of these sets converge to the $n$-dimensional {submanifold} $M\subset E$.
}
\label{fig: surface}
\end{figure}

\begin{lemma}
\label{l:capture}
If $x\in B_m$ then $|f_i(x)-q_m| < \tfrac13$ for all {$i\ge 1$.}
\end{lemma}

\begin{proof}
Suppose the contrary
and let
$$
i_0=\min\{i:|f_i(x)-q_m|\ge\tfrac13 \} .
$$
{Let $i\le i_0$ be such that $|q_i-q_m|<1$. Such $i$ does exist since
otherwise $|q_i-q_m|>1$ for all $i\leq i_0$ implying that
$f_i(x)=x \in B_m$. In particular, for $i=i_0$, 
$|f_{i_0}(x) -q_m|  <1/3$ which is a contradiction.
Next let $z=f_{i-1}(x)$.
} 
Since $i-1<i_0$, we have $|z-q_m|<\frac13$.
Lemma \ref{l:small-angle} applied to $p=q_i$ and $q=q_m$
implies that $|P_i(z)-P_m(z)| < 6\de$.
Since $P_m$ is the orthogonal projection to a subspace
containing $q_m$, we have
$|P_m(z)-q_m| \le |z-q_m|$,
therefore
$$
 |P_i(z)-q_m| \le |P_m(z)-q_m|+|P_i(z)-P_m(z)| \le |z-q_m| + 6\de
$$
and hence the point
$$
f_i(x)=\varphi_i(z)=\mu_i(z)P_i(z)+(1-\mu_i(z))z
$$
satisfies
$$
 |f_i(x)-q_m| \le \mu_i(z)|P_i(z)-q_m| + (1-\mu_i(z))|z-q_m|
 \le |z-q_m| + 6\de .
$$
Thus
\be
\label{e:capture1}
 |f_i(x)-q_m| \le |f_{i-1}(x)-q_m| + 6\de
\ee
for all $i\le i_0$ such that $|q_i-q_m|<1$.
For indices $i\le i_0$ such that $|q_i-q_m|\ge 1$,
we have
$$
|f_{i-1}(x)-q_i|\ge 1 - |f_{i-1}(x)-q_m| > 1 - \tfrac13 > \tfrac12
$$
and hence
$f_i(x)=f_{i-1}(x)$.
Since there are at most $N=N(n)$ indices $i\le i_0$ such that
$|q_i-q_m|<1$, by \eqref{e:capture1} it follows that
$$
 |f_{i_0}(x)-q_m| \le |x-q_m| + 6N\de < |x-q_m| + \tfrac1{20} < \tfrac13,
$$
provided that $\de < 1/(120N)$.
This contradicts the choice of $i_0$.
\end{proof}

Lemma \ref{l:capture} implies that
there exists only finitely many indices $i$
such that $f_i|_{B_m}\ne f_{i-1}|_{B_m}$.
Indeed, if $f_i(x)\ne f_{i-1}(x)$ for some $x\in B_m$,
then $|q_i-q_m|<1$ because $|f_{i-1}(x)-q_m|<\frac13$
by Lemma \ref{l:capture} and $|f_{i-1}(x)-q_i|<\frac12$
(since $\varphi_i$ is the identity outside $B_{1/2}(q_i)$).
Thus the sequence $\{f_i|_{B_m}\}_{i=1}^\infty$ stabilized
and hence the map 
\be\label{e:def of f}
f(x)=\lim_{i\to \infty} f_i(x)
\ee 
is well-defined and smooth on~$B_m$.
Since $m$ is arbitrary, $f$ is well-defined and smooth
on $U=\bigcup_m B_m$. 

\begin{remark}
{We note that in the case when $X$ and thus $X_0\subset X$ are finite sets and
when $N$ is the number of the elements in $X_0$, we define instead of (\ref{e:def of f})
\be\label{e:def of f finite X0}
f(x)=f_N(x).
\ee}
\end{remark}

\subsubsection{Estimates for interpolation maps $f_i$ and $f$.}\label{subsubsection f and f_i}
Next, we consider functions $f$ and $f_i$ defined in (\ref{e:def of phi_i}), (\ref{e:def of fi}), and (\ref{e:def of f}).

\begin{lemma}
\label{l:almost-projection} {Let $k\ge 0$. There is $\Cseventeen(k)= {(C(k))^{N(n)}}
>0$
 such that}
\be
\label{e:fm-almost-projection}
 \|f_i-P_m\|_{C^k(B_m)} \le \Cseventeen(k)\de \qquad\text{for all $i\ge m$},
\ee
and  therefore
\be
\label{e:f-almost-projection}
 \|f-P_m\|_{C^k(B_m)} \le \Cseventeen(k)\de .
\ee
{Below we denote $\Cseventeen=\Cseventeen(0)$,
$\Cseventeen'=\Cseventeen(1)$, and $\Cseventeen''=\Cseventeen(2)$.}
\end{lemma}

\begin{proof}
Let $I_m=\{i:|q_i-q_m|<1\}$ and
let $j_1<\dots<j_{N_m}$ be all elements of $I_m$.
Recall that $N_m=|I_m|\le N=N(n)$.
As shown above, Lemma \ref{l:capture} implies that
$\varphi_i$ is the identity on $f_{i-1}(B_m)$
for $i\notin I_m$.
Therefore for every $i$ we have
\be
\label{e:fi-short}
 f_i|_{B_m} = \varphi_{j_{l(i)}}\circ\varphi_{j_{l(i)-1}}\circ\ldots\circ\varphi_{j_1}|_{B_m}
\ee
where $l(i)=\max \{k: j_k\le i \}$.

We compare $\varphi_i$ and $f_i$ with maps
$\hat\varphi_i$ and $\hat f_i$ defined by
\be\label{hat phi}
 \hat\varphi_i(x) = \mu_i(x) P_m(x) + (1-\mu_i(x)) x ,
\ee
and
\be\label{hat f}
 \hat f_i = \hat\varphi_{j_{l(i)}}\circ\hat\varphi_{j_{l(i)-1}}\circ\ldots\circ\hat\varphi_{j_1}
\ee
By induction one easily sees that
\be
\label{e:lambdami}
 \hat f_i(x) = \la_i(x) P_m(x) + (1-\la_i(x))x
\ee
for some $\la_i(x)\in[0,1]$, $\la_1(x)\le\la_2(x)\le\dots$.
Therefore $\hat f_i(B_m)\subset B_m$ for all~$i$.
Similarly to the case of $f_i$ this implies that
\be
\label{e:hatfi-short}
 \hat f_i|_{B_m} = \hat\varphi_{j_{l(i)}}\circ\hat\varphi_{j_{l(i)-1}}\circ\ldots\circ\hat\varphi_{j_1}|_{B_m}
\ee

{Let 
$$
\Phi_{i}^{i'}=\varphi_{j_{l(i)}}\circ\varphi_{j_{l(i)-1}}\circ\ldots \circ \varphi_{j_{i'+1}},\quad 
\hat \Phi_{i}^{i'}=
\hat\varphi_{j_{i'}}\circ \ldots \circ \hat\varphi_{j_{1}}
$$
and 
$$ f_i^{i'} :=
\Phi_{i}^{i'}\circ \hat \Phi_{i}^{i'}=
 \varphi_{j_{l(i)}}\circ\varphi_{j_{l(i)-1}}\circ\ldots \circ\varphi_{j_{i'+1}}\circ\hat\varphi_{j_{i'}}\circ \ldots \circ \hat\varphi_{j_{1}}|_{B_m}$$
By Lemma \ref{l:small-angle} and \eqref{angle 2}, for every $i\in I_m$ we have
$$
 \|P_i(x)-P_m(x)\|\le  {11}\de, \qquad \|d_xP_i-d_xP_m\|\le {10}\de
$$
for all $x\in B_1(q_m)$ and therefore, {as $P_i$ and $P_m$ are affine maps,}
\beq\label{e difference}
& &\|\varphi_i  \|_{C^k(B_1(q_m))}\leq \Csixteen(k)k,\\
\nonumber
& &\|\hat\varphi_i  \|_{C^k(B_1(q_m))}\leq \Csixteen(k)k,\\
\nonumber
& &
 \|\hat\varphi_i - \varphi_i \|_{C^k(B_1(q_m))} = \|\mu_i\cdot(P_m-P_i)\|_{C^k(B_1(q_m))} \le {11 \Csixteen(k) k^2}\de ,
\eeq
{\ctext where the factor $k^2$ appears due to the Leibniz rule for derivatives of the product and the fact that the 2nd and the higher order derivatives of affine maps vanish.}
This estimate, \eqref{e:fi-short}, \eqref{e:hatfi-short} and the fact that $l(i)\le|I_m|\le N(n)$
imply that
\begin{eqnarray}
\label{e:fi-hatfi}
\nonumber  \|f_i-\hat f_i\|_{C^k(B_m)}  \le  \sum_{i' = 1}^{l(i)} A_i^{i'},\quad  A_i^{i'}=\|f_i^{i'} - f_i^{i'-1} \|_{C^k(B_m)}\end{eqnarray}
As
$$
A_i^{i'}=\|\Phi_{i}^{i'+1}\circ   \varphi_{j_{i'}}\circ \hat \Phi_{i}^{i'}-\Phi_{i}^{i'+1}\circ   \hat \varphi_{j_{i'}}\circ \hat \Phi_{i}^{i'}\|_{C^k(B_m)}
$$
we see by using 
Lemma \ref{l:iterated chainrule}(2) in Appendix A with $f=f_{i'}=\Phi_{i}^{i'+1}$, $h=h_{i'}= \varphi_{j_{i'}}\circ \hat \Phi_{i}^{i'}$  and $g=g_{i'}= \hat\varphi_{j_{i'}}\circ \hat \Phi_{i}^{i'}$,
we see for $k\geq 1$ that 
\begin{eqnarray}\label{e: technical 0}\\ \nonumber 
A_i^{i'}  \leq
(k+1)2^{k(k-1)}\|  f_{i'}  \|_{C^{k+1}(B_m)}(1+\| g_{i'} \|_{C^{k}(B_m)}+\|  h_{i'}   \|_{C^{k}(B_m)})^{k} \| g_{i'} - h _{i'} \|_{C^{k}(B_m)}.\hspace{-1cm}
\end{eqnarray}
Here, by Lemma \ref{l:iterated chainrule}(1) in Appendix A 
and (\ref{e difference}) we have
\begin{eqnarray}\label{e: technical 1}
 & &\hspace{-1cm}\| g_{i'} \|_{C^{k}(B_m)}+\|  h_{i'}   \|_{C^{k}(B_m)} \leq 2^{(k+1)^2N(n)}\Csixteen(k) \,(1+ k\Csixteen(k))^{kN(n)}\hspace{-2cm}\\ \nonumber
& & \hspace{-1cm}\| f_{i'} \|_{C^{k+1}(B_m)} \leq 2^{(k+2)^2N(n)}\Csixteen(k+1) \,(1+ (k+1) \Csixteen(k+1))^{(k+1)N(n)}\hspace{-2cm}
\end{eqnarray}
and 
\begin{eqnarray} \nonumber
   \| g_{i'} - h _{i'} \|_{C^{k}(B_m)}&=& \| ( \hat\varphi_{j_{i'}}-\varphi_{j_{i'}})\circ \hat \Phi_{i}^{i'} \|_{C^{k}(B_m)}
\\   \label{e: technical 2}
& &  \leq  
 2^{k(k+1)N(n)}\,\cdotp
 {11 \Csixteen(k) k^2}\de\,\cdotp
  \,(1+ k\Csixteen(k))^{kN(n)}.\end{eqnarray}
By substituting formulas (\ref{e: technical 1}) and \eqref{e: technical 2} in to formula \eqref{e: technical 0}, and using that
${l(i)} \leq N(n)$,
we see that
\begin{eqnarray*}
 & &\hspace{-1cm}\|f_i-\hat f_i\|_{C^k(B_m)}  \le \sum_{i' = 1}^{l(i)} A_i^{i'}
 \leq \Cseventeen(k)\de
\end{eqnarray*}
for all $i$ and $k\ge 0$, where $\Cseventeen(k)$ can be written as an explicit formula involving $k$, $\Csixteen(k)$, $ \Csixteen(k+1)$, and $N(n)$.}
Observe that $\hat\varphi_m|_{B_m}=P_m|_{B_m}$
since $\mu_m=1$ on $B_m$.
{This fact together with $\hat f_m= \hat \phi_m \circ \hat f_{m-1}$ and
 \eqref{e:lambdami}
imply that $\hat f_m|_{B_m}=P_m|_{B_m}$.
%
Thus}
$\hat f_i|_{B_m}=P_m|_{B_m}$
for all $i\ge m$.
Therefore for $i\ge m$ the estimate \eqref{e:fi-hatfi}
turns into \eqref{e:fm-almost-projection}
and the claim of the lemma follows.
\end{proof}

\begin{lemma}
\label{l:ndim-image}
$f_m(B_m)\subset D_{1/3}(q_m)$.
\end{lemma}

\begin{proof}
Let $x\in B_m$ and $y=f_{m-1}(x)$, then $f_m(x)=\varphi_m(y)$.
By Lemma \ref{l:capture}, $|y-q_m|<\frac13$.
Therefore $\mu_m(y)=1$ and hence $\varphi_m(y)=P_m(y)$.
Thus $f_m(x)=P_m(y)\in D_{1/3}(q_m)$.
\end{proof}

By definition, $f=g\circ f_m$ for some smooth map $g\co E\to E$.
Therefore $f(B_m)$ is contained in an image of {\ctext the} $n$-dimensional
disc $D_{1/3}(q_m)$ under a smooth map~$g$.

\begin{lemma}
\label{l:fU-near-X}
$f(B_m)\subset {\mathcal U}_{4\de}(D_{1/3}(q_m))$ for every $m$,
and $f(U)\subset {\mathcal U}_{5\de}(X)$.
\end{lemma}

\begin{proof}
Let $x\in B_m$.
By Lemma \ref{l:capture} we have $f_i(x)\in B_{1/3}(q_m)$
for all~$i$. Let us show that $f_i(x)\in {\mathcal U}_{4\de}(A_m)$
for all $i\ge m$. This is true for $i=m$
since $f_m(x)\in D_{1/3}(q_m)\subset A_m$ by Lemma \ref{l:ndim-image}.
Arguing by induction, let $i>m$ and assume that $y=f_{i-1}(x)\in {\mathcal U}_{4\de}(A_m)$.
If $|y-q_i|\ge\frac12$, then ${f_i(x)}=y\in {\mathcal U}_{4\de}(A_m)$,
so we assume that $|y-q_i|<\frac12$.
Note that
$$
|q_i-q_m|\le |q_m-y|+|y-q_i|<\tfrac13+\tfrac12<1 .
$$
By definition, the point $f_i(x)=\varphi_i(y)$ belongs to the line segment $[yz]$
where $z=P_i(y)$.
Since $z\in A_i$ and $|q_i-z| \le |q_i-y|<\frac12$,
we have
$$
 \dist(z,A_m) \le \dist(q_i,A_m)+ \tfrac12\sin\angle(A_i,A_m)<\de+\tfrac52\de<4\de
$$
where the second inequality follows from Lemma \ref{l:small-angle}.
Thus $z\in {\mathcal U}_{4\de}(A_m)$.
Since $f_i(x)\in[yz]$, both $y$ and $z$ belong to ${\mathcal U}_{4\de}(A_m)$
and ${\mathcal U}_{4\de}(A_m)$ is a convex set,
$f_i(x)\in {\mathcal U}_{4\de}(A_m)$ as claimed.

Thus $f_i(x)\in {\mathcal U}_{4\de}(A_m)\cap B_{1/3}(q_m)$ for all $x\in B_m$ and all $i\ge m$.
This implies the first claim of the lemma.
To prove the second one, recall that
$D_1(q_m)\subset {\mathcal U}_\de(X)$ by \eqref{e:affine-delta}.
Hence $f(B_m)\subset {\mathcal U}_{4\de}(D_{1/3}(q_m))\subset {\mathcal U}_{5\de}(X)$.
Since $m$ is arbitrary, the second assertion of the lemma follows.
\end{proof}

\subsubsection{Construction and properties of the submanifold $M$} 
Now define 
\be\label{e:Mdefinition}
M=f({\mathcal U}_{1/5}(X_0)) .
\ee
We are going to show that $M$ is a desired submanifold.

\begin{lemma}
\label{l:disc-cover}
For every $y\in M$ there exists $q_m\in X_0$ such that
$|y-q_m|<\frac1{100}+5\de$ and
$$
 M\cap B_{1/100}(y) \subset f(D_{1/10}(q_m)) .
$$
In particular, $M=\bigcup_m f(D_{1/10}(q_m))$.
\end{lemma}

\begin{proof}
By Lemma \ref{l:fU-near-X}, $y\in {\mathcal U}_{5\de}(X)$.
Since $X_0$ is a $\frac1{100}$-net in $X$,
there is point $q_m\in X_0$ such that 
$|y-q_m|<\tfrac1{100}+5\de$.
Let us show that this point satisfies the requirements of the lemma.
%
Let $W=M\cap B_{1/100}(y)$ and $D=D_{1/10}(q_m)$.
We are to show that $W\subset f(D)$.
Fix a point $z\in W$. Observe that
$$
 |z-q_m|\le |z-y|+|y-q_m|<\tfrac1{100}+\tfrac1{100}+{5}\de =\tfrac1{50}+{5}\de .
$$
Since $z\in M$, we have $z=f(x)$ for some $x\in {\mathcal U}_{1/5}(X_0)$.
Let $p\in X_0$ be such that $|x-p|<\frac15$.
Then $|z-P_{A_p}(x)|<\Cseventeen\de$
by Lemma \ref{l:almost-projection}.
On the other hand,
$$
 |x-P_{A_p}(x)|\le |x-p| <\tfrac15 .
$$
Therefore, {assuming that $\delta$  is smaller than some constant depending only on $n$ (see (\ref{delta condition})), we have} 
$$
|x-q_m|\le|x-P_{A_p}(x)|+|z-P_{A_p}(x)|+|z-q_m|<\tfrac15+\Cseventeen\de+\tfrac1{50}+5\de <\tfrac14,
$$
thus $x\in B_m$.

By Lemma \ref{l:almost-projection} it follows that $|z-P_m(x)|=|f(x)-P_m(x)|<\Cseventeen\de$
and $|f_m(x)-P_m(x)|<\Cseventeen\de$.
Therefore $|f_m(x)-z| < 2\Cseventeen\de$
and hence
$$
 |f_m(x)-q_m|\le |f_m(x)-z|+|z-q_m|<\tfrac1{50}+(2\Cseventeen+5)\de .
$$
By Lemma \ref{l:ndim-image} we have $f_m(x)\in A_m$,
hence $f_m(x)\in D_{1/50+{(2\Cseventeen+5)}\de}(q_m)$.  

Now consider the map $f_m|_D$.
By Lemma \ref{l:ndim-image}, its image $f_m(D)$
is contained in $A_m$. By Lemma \ref{l:almost-projection},
$f_m|_D$ is $\Cseventeen\de$-close to the projection $P_m|_D$,
which equals $id_D$ since $D\subset A_m$.
Thus $f_m|_D$ is $\Cseventeen\de$-close to the identity
and maps $D$ to a subset of the $n$-dimensional subspace $A_m$.
By topological reasons, see \cite[Thm. 1.2.6]{degreetheory},
this implies that $f_m(D)$ contains
an $n$-ball $D_{1/10-\Cseventeen\de}(q_m)$, {see (\ref{delta condition})}. 
Since $f_m(x)\in D_{1/50+(2\Cseventeen+5)\de}(q_m)\subset D_{1/10-\Cseventeen\de}(q_m)$,
it follows that there exists a point $x'\in D$ such that
$f_m(x')=f_m(x)$.
Since $f$ factors through $f_m$, this implies that $f(x')=f(x)=z$.
Thus $z\in f(D)$. Since $z$ is an arbitrary point of $W$,
the lemma follows.
\end{proof}


Now we prove that $M$ is a submanifold.

\begin{lemma}
\label{l:Mismanifold}
$M$ is a closed $n$-dimensional smooth submanifold of $E$.
{
Every $y\in M$ has a neighborhood in $M$
that admits a parametrization by a smooth map
$\varphi\co V\to E$, $V\subset\R^n$,
which is $\Cseventeen(k)\de$-close to an affine isometric embedding
in the $C^k$-topology for any $k\ge 0$,
where $\Cseventeen(k)$ is the constant provided by Lemma \ref{l:almost-projection}.
}
\end{lemma}

\begin{proof}
Pick $y\in M$ and
let $q_m\in X_0$ be as in Lemma \ref{l:disc-cover}.
{As in the proof of  Lemma \ref{l:disc-cover}, we use the notation}
$D=D_{1/10}(q_m)$.
By Lemma \ref{l:almost-projection},
$f|_D$ is $\Cseventeen(k)\de$-close to the {inclusion $D\hookrightarrow E$} in the $C^k$-topology.
{
Assuming that $\de<\Cseventeen(1)^{-1}$,
it follows that $f|_D$ is a smooth embedding and hence $f(D)$ is a smooth submanifold of $E$.}
By Lemma~\ref{l:disc-cover},
$$
 f(D)\cap B_{1/100}(y)=M\cap B_{1/100}(y) .
$$ 
{Thus $M\cap B_{1/100}(y)$ is
a submanifold for every $y\in M$, hence so is~$M$.
}

To see that $M$ is closed, recall that $|y-q_m|<\frac1{100}+5\de$.
Since $f|_D$ is $\Cseventeen\de$-close to identity,
this implies that the $f$-image of the boundary of $D$ is separated
away from $y$ by distance at least $\frac1{10}-\frac1{100}-{\ctext 5\de}-\Cseventeen\de>\frac1{100}$.
Therefore $M\cap B_{1/100}(y)$ is contained in a compact subset of
the submanifold $f(D)$. Since this holds within a uniform radius $\frac1{100}$
from any $y\in M$, it follows that $M$ is a closed set in $E$.

To construct the desired local parametrization $\varphi$, just compose $f|_D$
with an affine isometry between $D$ and an appropriate ball {$V\subset\R^n$}.
\end{proof}

{
The bounds on derivatives of $\varphi$ from Lemma \ref{l:Mismanifold} imply that
the second fundamental form of $M$ is bounded by
\beq\label{use of Ceight}
  (1-\Cseventeen'\de)^{-2}\Cseventeen''\de < 2\Cseventeen''\de  
\eeq
provided that $\de<(4\Cseventeen')^{-1}$.
See Lemma \ref{l:almost-projection} for the notation $\Cseventeen'$ and $\Cseventeen''$.
{The inequality (\ref{use of Ceight}) proves the second assertion of Proposition \ref{p:surface} with
$\Ceight=2\Cseventeen''$.}

The first assertion of Proposition \ref{p:surface} is the following lemma.
}


\begin{lemma}\label{l:dHMX}
$d_H(M,X)\le 5\de$. 
\end{lemma}

\begin{proof}
By Lemma \ref{l:fU-near-X} we have
$M\subset {\mathcal U}_{5\de}(X)$.
It remains to prove the inclusion $X\subset {\mathcal U}_{5\de}(M)$.
Fix $x\in X$ and let $q_m\in X_0$ be such that $|q_m-x|\le\frac1{100}$.
Consider the map $P_m\circ f|_{D_{1/5}(q_m)}$ from $D_{1/5}(q_m)\subset A_m$
to $A_m$.
By Lemma \ref{l:almost-projection}, this map is $\Cseventeen\de$-close to the identity.
Therefore its image contains the $n$-disc $D_{1/5-\Cseventeen\de}(q_m)$.
This disc contains the point $P_m(x)$ because
$$
|P_m(x)-q_m|\le |x-q_m|\le \tfrac1{100}<\tfrac15-\Cseventeen\de .
$$
Hence $P_m(x)\in P_m(f(D_{1/5}(q_m)))$.
This means that there exists $y\in D_{1/5}(q_m)$ such that
$P_m(f(y))=P_m(x)$.
By Lemma~\ref{l:fU-near-X}, we have $\dist(f(y),A_m)<4\de$
and therefore $$|f(y)-P_m(x)|=|f(y)-P_m(f(y))|<4\de.$$
By {Lemma \ref{l:small-angle}}, 
we have $\dist(x,A_m)\le\de$
and therefore $|x-P_m(x)|\le\de$. Hence
$$
 |f(y)-x| \le |f(y)-P_m(x)|+|x-P_m(x)| < 4\de+\de = 5\de .
$$
Observe that $f(y)\in M$ since $y\in D_{1/5}(q_m)\subset {\mathcal U}_{1/5}(X_0)$.
This and the above inequality imply that $x\in {\mathcal U}_{5\de}(M)$.
Since $x$ is an arbitrary point of $X$,
we have shown that $X\subset {\mathcal U}_{5\de}(M)$.
The lemma follows.
\end{proof}

\begin{remark}
\label{rem:small-tube}
We observe that
\be
\label{e:small-tube}
 M = f({\mathcal U}_\de(X))
\ee
(compare with \eqref{e:Mdefinition}).
Indeed, we have $ M \subset \bigcup_m f(D_{1/10}(q_m)) $
by Lemma \ref{l:disc-cover}
and
$
D_{1/10}(q_m) \subset {\mathcal U}_\de(X)
$
by \eqref{e:affine-delta}.

One can think of \eqref{e:small-tube}, \eqref{e:Mdefinition}
and the last claim of Lemma \ref{l:disc-cover}
as various reconstruction procedures for $M$.
\end{remark}

\begin{lemma}
\label{l:fdisplacement} 
$|f(y)-y|< \Cnineteen\de$ for every $y\in {\mathcal U}_\de(X)$.
\end{lemma}

\begin{proof}
Since $y\in {\mathcal U}_\de(X)$, there is $x\in X$ such that $|x-y|<\de$.
Pick $q_m\in X_0$ such that $|x-q_m|<\tfrac1{100}$.
Then $y\in B_m$ and hence $|f(y)-P_m(y)|<\Cseventeen\de$
by Lemma \ref{l:almost-projection}.
By {Lemma \ref{l:small-angle}}, 
we have $\dist(x,A_m)<\de$
and hence 
$$
 |y-P_m(y)|=\dist(y,A_m)<2\de.
$$
Therefore
$
 |f(y)-y| \le |f(y)-P_m(y)| + |y-P_m(y)| <(\Cseventeen+2)\de=\Cnineteen \delta
$.
\end{proof}

{
For $x,y\in M$, we denote by $d_M(x,y)$ the intrinsic arc-length distance between $x$ and $y$ in~$M$.
If $x$ and $y$ are from different connected components of $M$, then $d_M(x,y)=\infty$.
Since $M$ is closed in $E$, each component of $M$ is a complete Riemannian manifold.

\begin{lemma}\label{l:intrinsic-vs-ambient-distance}
Let $x,y\in M$ be such that $|x-y|<\tfrac45$. Then $d_M(x,y) < 1$.
\end{lemma}
}

\begin{proof}
Let $x,y\in M$ be as above.
Then by \eqref{e:small-tube} there are points $x',y'\in {\mathcal U}_\de(X)$ such that
$f(x')=x$ and $f(y')=y$.
By Lemma \ref{l:fdisplacement} we have
$|x-x'| < \Cnineteen\de$ and $|y-y'|< \Cnineteen\de$,
hence $|x'-y'|<\frac45+2 \Cnineteen\de$ by the triangle inequality.
Let $x'',y''\in X$ be such that $|x'-x''|<\de$
and $|y'-y''|<\de$.

{

Then, when $\delta$  is smaller than a bound depending on $n$, see (\ref{delta condition}), 
$$
|x''-y''|\le |x'-y'|+2\de<\tfrac45+2 \Cnineteen\de +2 \delta<1.
$$
Hence $y''\in B_1^X(x'')$.
This and \eqref{e:affine-delta} imply that
$y''\in {\mathcal U}_\de (D_1(x''))$.
Therefore both $x'$ and $y'$ and hence the line segment $[x',y']$
are contained in the $2\de$-neighborhood of
the affine $n$-disc $D_1(x'')$.
Since $B_1^X(x'') \subset X$ it follows from
 \eqref{e:affine-delta} that
{$D_1(x'')\subset\mathcal U_\delta(B_1^X(x''))\subset \mathcal U_\delta(X)$.
Hence the $2\de$-neighborhood of $D_1(x'')$}
is contained in ${\mathcal U}_{3\de}(X)$.
Thus $[x',y']$ is contained in ${\mathcal U}_{3\de}(X)$ and hence in the domain of~$f$.

Consider the $f$-image of the line segment $[x',y']$.
It is a smooth path in~$M$ connecting $x$ and $y$.
Lemma \ref{l:almost-projection} for $k=1$
implies that $f$ is locally Lipschitz with Lipschitz constant $1+{\Cseventeen'}\de$.
Therefore,  
$$
 \length(f([x',y'])) \le (1+{\Cseventeen'}\de) |x'-y'| < (1+{\Cseventeen'}\de)(\tfrac45+2\Cnineteen\de) < 1,
$$
see (\ref{delta condition}).
Hence $d_M(x,y)<1$.}


\end{proof}

Now we are in position to prove the third assertion of Proposition \ref{p:surface}.

{
\begin{lemma}\label{l:reachM}
$\Reach(M)\ge \frac13$.
Furthermore, for every $p\in \mathcal U_{1/3}(M)$
there exists a unique $x\in M$ such that $|p-x|<\frac13$ and $p-x\perp T_xM$.
\end{lemma}

\begin{proof}
Fix $p\in \mathcal U_{1/3}(M)$.
By Lemma \ref{l:intrinsic-vs-ambient-distance},
the set $B_{1/3}(p)\cap M$ is contained in a unit ball of $(M,d_M)$,
since the diameter of this set in $E$ is bounded by $\frac23<\frac45$.
Since $M$ is a complete Riemannian manifold, closed balls in $(M,d_M)$ are compact.
Hence there exists $x\in M$ nearest to~$p$.
It remains to prove that $x$ is a unique nearest point and that it
is also a unique point of $B_{1/3}(p)\cap M$ such that $p-x\perp T_xM$.

Let $y$ be another point from $B_{1/3}(p)\cap M$.
By Lemma \ref{l:intrinsic-vs-ambient-distance} we have $d_M(x,y)<1$.
Connect $x$ to $y$ by a unit-speed minimizing geodesic $\gamma\co[0,L]\to M$
and consider the function $f(t)=\frac12|p-\gamma(t)|^2$, $t\in[0,L]$.
Computing the second derivative of $f(t)$ yields
\be\label{e:reach1}
 f''(t) = {\ctext |\gamma'(t)|^2 - }\langle \gamma''(t),p-\gamma(t)\rangle
 = {\ctext 1  -} \langle \gamma''(t),p-\gamma(t)\rangle
\ee
where $\langle\cdot,\cdot\rangle$ is the inner product in $E$.


Let $\kappa$ denote our bound on the second fundamental form of $M$,
i.e.\ $\kappa=\Ceight\de$.
Since $\gamma$ is a geodesic, $|\gamma''(t)|\le\kappa$.
This and \eqref{e:reach1} imply that
$
 |f''(t)-1| \le \kappa |p-\gamma(t)|
$
for all $t$. Thus $0<f''(t)<2$ as long as $|p-\gamma(t)|<\kappa^{-1}$.
Since $p-x\perp T_x\Sigma$, we have $f'(0)=0$.
Therefore $0<f''(t)<2$, $0<f'(t)<2t$, and $f(0)<f(t)<f(0)+ t^2$
for all $t\in(0,L]$ such that $f(0)+ t^2<\frac12\kappa^{-2}$.

Assuming that $\kappa=\Ceight\de<\frac12$ (see \eqref{delta condition}) and using estimates
$|p-x|<1$ and $L=d_M(x,y)<1$, we see that the above inequalities
hold for all $t\in[0,L]$.
In particular $f(L)>f(0)$ and $f'(L)>0$,
hence $|p-y|>|p-x|$ and $p-y$ is not orthogonal to $T_yM$.
\end{proof}
}

{
Now we have the normal projection map $P_M\co\mathcal U_{1/3}(M)\to M$.
Let us prove the fifth assertion of Proposition \ref{p:surface}.
Let $x\in X$ and $y=P_M(x)$. Then $|x-y|<5\de$ by Lemma \ref{l:dHMX}.
By Lemma \ref{l:disc-cover} there exists $q_m\in X_0$ such that
$|y-q_m|<\frac1{100}+5\de$ and $y\in f(D)$ where $D=D_{1/10}(q_m)$.
By Lemma \ref{l:almost-projection}, $f|_D$ is $\Cseventeen'{\delta}$-close
to $P_m|_D=\text{id}_D$ in the $C^1$-topology.
Therefore
$$
 \angle ( A_{q_m}, T_yM ) < (1-\Cseventeen'\de)^{-1}\Cseventeen'\de < 2\Cseventeen'\de
$$
provided that $\de<(2\Cseventeen')^{-1}$.
By Lemma \ref{l:small-angle} we have $\angle(A_x,A_{q_m})<5\de$. Hence
$$
 \angle ( A_{x}
  T_yM ) < (2\Cseventeen'+5)\de .
$$
and \eqref{e:Pangle} follows with $\Cnine=2\Cseventeen'+5$.
}

\medbreak

{
It remains to prove the fourth assertion of Proposition \ref{p:surface}.
Consider the normal disc bundle
\be\label{normal disc bundle}
 \nu_{1/3}M := \{(x,v) : x\in M, v\in E, v\perp T_xM, |v|<1/3 \}
\ee
and the map $J\co\nu_{1/3}M\to E$ given by $J(x,v)=x+v$.
Lemma \ref{l:reachM} implies that $J$ is a bijection onto $\mathcal U_{1/3}(M)$.
The normal projection $P_M\co\mathcal U_{1/3}(M)\to M$
can be written as $P_M=\pi\circ J^{-1}$ where $\pi(x,v)=x$
for $(x,v)\in \nu_{1/3}M$. Thus is suffices to show that $J^{-1}$ is smooth
and estimate its derivatives.

Let $(x_0,v_0)\in\nu_{1/3}M$. 
By means of a parallel translation we may assume that $x_0$ is the origin of $E$.
By Lemma \ref{l:Mismanifold}, a neighborhood of $x_0$ in $M$
admits a local parametrization $\varphi\co V\to M$ which is
is $\Cseventeen(k)\de$-close in $C^k$-topology
to an affine isometric embedding.
We identify $V$ with a neighborhood of the origin in the tangent space $T_{x_0}M\subset E$
so that $\varphi$ is  $\Cseventeen(k)\de$-close to the identity in $C^k$-topology.
Let $B$ be the ball of radius $1/3$ in the orthogonal complement $(T_{x_0}M)^\perp$ of $T_{x_0}M$ in~$E$.
Parametrize a neighborhood of $(x_0,v_0)$ in $\nu_{1/3}M$ 
{by $(\varphi(x), \, v- P_{\varphi(x)}(v))$
and introduce
 a map
$\Phi\co V\times B\to E $ given by
\be\label{Phi capital}
 \Phi(x,v) = \varphi(x)+ v- P_{\varphi(x)}(v), \qquad x\in V,\ v\in B,
\ee
%
%
where} $P_{\varphi(x)}$ is the orthogonal projection from $E$ to $T_{\varphi(x)}M$.
The projection $P_{\varphi(x)}(v)$ can be written explicitly as an arithmetic
formula involving the first derivatives of $\varphi$ and their inner products
with each other and~$v$, hence the $k$th derivatives of ${\Phi}$ can be
written explicitly in terms of the derivatives of $\varphi$ up to the order $k+1$
and the inner product in~$E$. If $\varphi$ is the identity, then so is ${\Phi}$.
Since $\varphi$ is  $\Cseventeen(k+1)\de$-close to the identity in $C^{k+1}$-topology,
this implies an estimate
\beq\label{e C20}
 \|{\Phi}-\text{Id}\|_{C^k(V\times B)} < {\Ctwenty(k)}\de 
\eeq
where $\Ctwenty(k)$ is a constant that can be written explicitly
in terms of $\Cseventeen(k+1)$.
By the Inverse Function Theorem, this implies that
${\Phi}$ is a local diffeomorphism provided that $\de<C(k)^{-1}$.
The normal projection $P_M$ in a neighborhood of $(x_0,v_0)$
is given by $P_M=\varphi\circ\pi_1\circ{\Phi}^{-1}$ where
$\pi_1\co V\times B\to V$ is the first coordinate projection.
The $k$th differential of ${\Phi}^{-1}$ can be written as an explicit
dimension-independent formula in terms of differentials of $\varphi$
up the $k$th order. If $\varphi$ is the identity, then $P_M$ is
the orthogonal projection $P_{T_{x_0}M}$.
This implies an estimate
$$
 \|P_M-P_{T_{x_0}M}\|_{C^k} \le \Cten(k) \de
$$
for all $k\ge 0$, where $\Cten(k)$ can be written explicitly in terms of $\Cseventeen(k+1)$.
These bounds imply \eqref{e:dkPM} and \eqref{e:d1PM}.
}


This finishes the proof of Proposition \ref{p:surface}.
As explained in the beginning of this section,
of Theorem \ref{t:surface} follows via a rescaling argument.

%
%

\begin{remark}
\label{rem:f and PM}
{Assuming  in Theorem \ref{t:surface} that $\delta<r/100$, and by scaling metric by factor $r^{-1}$ in  Lemmas \ref{l:almost-projection}  and {\ref{l:dHMX},}}   
the above arguments about $P_M$
imply that
\be\label{e: new citation}
 \|f - P_M\|_{C^k({\mathcal U}_{r/10}(M))} < \Cseventeen(k)\de{r^{-k}} 
\ee
for all $k$. Thus, for computation purposes, the explicitly constructed
map $f$ is as good as the normal projection $P_M$.
\end{remark}

\begin{remark} \label{rem: Theorem 2 optimality} Let us show that 
the constants in Theorem \ref{t:surface} are optimal, up to  constant factors.
Let $M\subset E$  be a closed  $n$-dimensional submanifold whose
second fundamental form is bounded by $\kappa_{\de,r}=\frac 1{2}\de r^{-2}$, with $0<\de<r<1$, and
{$\Reach(M)>2r$}.
Let  $x\in M$.
Using formula  (\ref{eq: Hausdorff distance of balls}) we see that 
\be
\label{e:2r estimate}
d_{H}(B^M_{2r}(x),B_{2r}^{T_xM}(x))\le  \de .
\ee
Here $B^M_{2r}(x)$ is the intrinsic ball in $M$ of radius $2r$ centered at $x$.

Our assumptions on $M$
imply that the normal projection $P_M$ is well-defined and 2-Lipschitz
in the ball $B^E_r(x)$. 
Hence for any $z\in M \cap B^E_r(x)$
the projection $P_M([x,z])$ of the line segment $[x,z]$ is a curve of length
at most $2r$. 
Therefore $z=P_M(z)\in B^M_{2r}(x)$.
Thus $M \cap B_r^E(x) \subset B^M_{2r}(x)$.
Also note that 
$
B^M_{r}(x)\subset M \cap B_r^E(x)
$.
These relations, \eqref{e:2r estimate} and \eqref{eq: Hausdorff distance of balls}
imply that $d_{H}(M\cap B^E_r(x),B_{r}^{T_xM}(x))\le  \de$.
As $x$ above is an arbitrary point of $M$, we have that 
$M$ is $\delta$-close to $n$-flats at scale $r$. 
This shows that  in Theorem \ref{t:surface} the bounds in claims (2) and (3) 
on the second fundamental form 
and  {reach} 
are optimal, up to multiplying these bounds by constant factors depending
on $n$. 
\end{remark}

\section{Proof of Proposition \ref{p:injrad} and injectivity radius estimates}
\label{sec:injrad}

The main goal of this section is to prove Proposition \ref{p:injrad}.
We begin with recalling some facts about Riemannian manifolds 
of bounded curvature and proving the estimate \eqref{eq: GH distance of balls}

Let $M=(M,g)$ be a complete Riemannian manifold with $|\Sec_M|\le K$
where $K>0$. For $p\in M$, consider the exponential map
$\exp_p\co T_pM\to M$.
We restrict this map to the ball of radius $r<\frac{\pi}{\sqrt K}$ in $T_pM$
centered at the origin. As a consequence of Rauch Comparison Theorem,
$\exp_p$ is non-degenerate in this ball and we have the following
estimates on its local bi-Lipschitz constants:
for ${v}\in T_pM$ such that $|{v}|\le r<\frac\pi{\sqrt K}$
and every $\xi\in T_pM\setminus\{0\}$,
\be
\label{e:rauch}
\frac{\sin(\sqrt Kr)}{\sqrt Kr}
 \le \frac{|d_{v}\exp_p(\xi)|}{|\xi|} 
\le \frac{\sinh(\sqrt Kr)}{\sqrt Kr}
\ee
(see e.g.\ 
\cite[Thm.\ 27 in Ch.\ 6]{Pe} and \cite[Thm. IV.2.5 and Remark IV.2.6]{Sakai}).

If $r\le {\frac12\min\{\frac\pi{\sqrt K},\inj_M(p) \}}$
then the geodesic $r$-ball $B^M_r(p)$ is convex,
i.e., minimizing geodesics with endpoints in this ball do not leave it
(see e.g.\ \cite[Thm.\ 29 in Ch.\ 6]{Pe}).
This makes the local bi-Lipschitz estimate \eqref{e:rauch} global.
{
Hence
\be
\label{e:rauch-bilip}
\frac{\sin(\sqrt Kr)}{\sqrt Kr}
 \le \frac{d_M(\exp_p({u}),\exp_p({v}))}{|{u}-{v}|} 
\le \frac{\sinh(\sqrt Kr)}{\sqrt Kr}
\ee
and therefore
\beq\label{final comparision estimate}
 \bigl|d_M(\exp_p({u}),\exp_p({v}))-|{u}-{v}|\bigr| \le  \tfrac12 Kr^3
\eeq
for all ${u,v}\in T_pM$ such that $|{u}|,|{v}|\le r$.
Here \eqref{final comparision estimate} follows from \eqref{e:rauch-bilip}
and the estimates $|{u}-{v}|\le 2r$ and
\be\label{e:sinh-estimate}
 t-\tfrac16t^3 \le \sin (t)\le \sinh (t)\le t+\tfrac14t^3, \qquad t\in[0,\tfrac\pi2] .
\ee
Thus the distortion of $\exp_p$ within the $r$-ball
is bounded by $\tfrac12 Kr^3$ provided that $r\le \frac12 \min\{\frac\pi{\sqrt K},\inj_M(p) \}$.
This proves \eqref{eq: GH distance of balls}.
}



The upper bound in \eqref{e:rauch-bilip} does 
depend on the assumption that $r\le \frac12 \inj_M(p) $.
Moreover the distances within a ball of radius 
$\frac\pi{2\sqrt K}$ have a better upper bound stated in Lemma \ref{l:toponogov-lite} below.
This lemma is a variant of Toponogov's Comparison Theorem (see e.g.\ \cite[Thm.\ 79 in Ch.\ 11]{Pe})
for geodesics that are not necessarily minimizing but whose lengths are bounded in terms of curvature.

Let $M^2_{-K}$ denote the rescaled
hyperbolic plane of curvature $-K$.
For real numbers $a,b>0$ and $\alpha\in[0,\pi]$, denote by 
$\curlyvee_{-K}(a,b,\alpha)$ the length of the side $x_1x_2$
of a triangle $\triangle x_0x_1x_2$ in $M^2_{-K}$
whose sides $x_0x_1$ and $x_0x_2$
equal $a$ and $b$, resp., and the angle at $x_0$ equals $\alpha$.
Note that
\be\label{e:hyperb-vs-eucl}
 \curlyvee_{-K}(a,b,\alpha) \le \curlyvee_0(a,b,\alpha) + \tfrac12Kr^3
 \quad\text{if $a,b\le r\le\tfrac\pi{2\sqrt K}$} ,
\ee
where $\curlyvee_0$ is defined similarly using the Euclidean plane as~$M^2_0$.
This follows from \eqref{final comparision estimate} applied to $M^2_{-K}$
in place of~$M$.

\begin{lemma}\label{l:toponogov-lite}
Let $M=(M^n,g)$ be a complete Riemannian manifold, $|\Sec_M|\le K$ where $K>0$,
$p\in M$, and $0<r\le\frac{\pi}{2\sqrt K}$.
Then
\be\label{e:toponogov-lite}
 d_M(\exp_p(u),\exp_p(v)) \le \curlyvee_{-K}(|u|,|v|,\angle(u,v)) \le |u-v| + \tfrac12Kr^3
\ee
for all $u,v\in T_pM$ such that $|u|,|v|\le r$.
\end{lemma}

\begin{proof}
This lemma is a standard application of Rauch comparison.
We give a proof for the reader's convenience.

Let $\tilde M$ be the rescaled hyperbolic $n$-space of curvature $-K$
and $\tilde p\in\tilde M$.
Denote by $B$ and $\tilde B$ the closed $r$-balls
centered at $p$ and $\tilde p$ in $M$ and $\tilde M$, resp.
Define a map $f\co \tilde B\to B$ by
$
  f = \exp_p\circ I \circ \exp_{\tilde p}^{-1}|_{\tilde B}
$
where $\exp_p$ and $\exp_{\tilde p}$ are the Riemannian exponential maps
of $M$ and $\tilde M$, resp., and $I$ is a linear isometry from $T_{\tilde p}\tilde M$ to $T_pM$.
Since $r<\frac{\pi}{\sqrt K}$, $\exp_p$ is non-degenerate within the $r$-ball.
Therefore, by \cite[Thm. IV.2.5]{Sakai}, the map $f$ does not increase lengths
of smooth curves.

Let $u,v\in T_pM$ be such that $|u|,|v|\le r$
and let $\tilde\gamma$ be a geodesic segment in $\tilde M$ between
the points $\exp_{\tilde p}(I^{-1}(u))$ and $\exp_{\tilde p}(I^{-1}(v))$.
Then $\tilde\gamma$ is contained in~$\tilde B$ and 
\be\label{e:toponogov1}
\length(\tilde\gamma)=\curlyvee_{-K}(|u|,|v|,\angle(u,v)) .
\ee
Since $\gamma:=f\circ\gamma$ is a path connecting $\exp_p(u)$ and $\exp_p(v)$ in $M$,
we have
$$
 d_M(\exp_p(u),\exp_p(v)) \le \length(\gamma) \le \length(\tilde\gamma)
$$
where the second inequality follows from the above mentioned fact that $f$ does not increase lengths.
This and \eqref{e:toponogov1} imply the first inequality in \eqref{e:toponogov-lite}.
The second inequality in \eqref{e:toponogov-lite} follows from \eqref{e:hyperb-vs-eucl}.
\end{proof}

The next lemma is a quantified version of the fact that,
for Riemannian manifolds with two-sided curvature bounds,
collapsing in the Gromov--Hausdorff sense is equivalent
to injectivity radii going to zero
(see \cite{Fu88}, \cite{CFG} or \cite[Ch.~8]{Gr}).
The advantage of Lemma \ref{l:injrad-flat} over collapsing technique
is that it provides bounds independent of the dimension.

\begin{lemma}\label{l:injrad-flat}
Let $M$ be a complete Riemannian manifold
with $|\Sec_M|\le K$ where $K>0$.
Let $p\in M$ and $r>0$ be such that
$Kr^2\le 10^{-3}$ and
$$
 d_{GH}(B_r^M(p),B_r^n) < 10^{-3} r
$$
where $n=\dim M$.
Then $ \inj_M(p) \ge \tfrac9{10}r $.
\end{lemma}

\begin{proof}
Define $\ep=10^{-3}$.
The statement of the lemma is scale invariant so it suffices
to prove it for $r=1$. More precisely, we rescale $M$ by the factor $r^{-1}$.
The rescaled manifold, for which we reuse the notation $M$, satisfies
\be\label{e:flat1}
 |\Sec_M| \le Kr^2\le \ep 
\ee
and
\be\label{e:flat2}
 d_{GH}(B_1^M(p),B_1^n) < \ep  ,
\ee
The desired inequality now takes the form $\inj_M(p) \ge \tfrac9{10}$.
We rewrite it as 
\be\label{e:flat0}
 \inj_M(p) \ge 1 - 100\ep .
\ee
The rest of the proof works for any $\ep\le 10^{-3}$.

First we informally explain the idea of this long and technical proof.
{Since the curvature of $M$ is small, the ball $B_1^M(p)$ is GH close to the set
of vectors in the unit ball of $T_pM$ corresponding to minimizing
geodesics starting at~$p$.
This is shown in Step~1 below. 
(In fact, the proof deals with spheres rather than balls
but we speak about balls in this informal explanation).
If the injectivity radius $\inj_M(p)$ is small, there is a short geodesic loop from $p$ to itself.
Using triangle comparison one can see that minimizing geodesics of length close to~1
cannot have too small angles with this loop.
This part of the argument is contained in Step~5.
One concludes that, if $\inj_M(p)$ is small then
$B_1^M(p)$, and hence $B_1^n$, is GH close to a subset
of the Euclidean unit ball where a significant part (namely a ball of certain smaller radius)
is removed.
Clearly this is impossible in any fixed dimension but
it is not so obvious when $n\to\infty$ while the radius of the removed ball stays fixed.
This issue is handled in Steps 2--4 with Kirszbraun's and Ulam-Borsuk theorems
applied to suitable maps.
Now we proceed with the formal proof.
}

By \eqref{e:flat2}, 
there is a $2\ep$-isometry $f\co B_1^n\to B_1^M(p)$ such that $f(0)=p$.
We denote by $B$ the unit ball in $T_pM$
and construct a map $h\co B_1^n\to B$ as follows:
for every $x\in B_1^n$, choose $h(x)\in B$ such that
$\exp_p(h(x))=f(x)$ and $|h(x)|=d_M(p,f(x))$.
(Note that the choice of $h(x)$ is not necessarily unique).
We proceed in a number of steps.

\smallbreak
{\it Step 1.}
We show that $h$ has a small distortion
on the unit sphere $S^{n-1}=\pd B_1^n$,
see \eqref{e:hx-hy-below} and \eqref{e:hx-hy-above} below.

For every $x,y\in B_1^n$ we have
$
 |f(x)-f(y)| \le |h(x)-h(y)| + \tfrac12\ep
$
by \eqref{e:flat1} and Lemma \ref{l:toponogov-lite} applied to $u=h(x)$, $v=h(y)$ and $r=1$.
Hence
\be\label{e:hx-hy-below}
 |h(x)-h(y)| \ge |f(x)-f(y)| - \tfrac12\ep \ge |x-y| - \tfrac52\ep
\ee
since $f$ is a $2\ep$-isometry.
In particular, since $h(0)=0$,
\be\label{e:hx-below}
 |h(x)| \ge |x| - \tfrac52\ep
\ee
for all $x\in B_1^n$. 

Pick $x,y\in S^{n-1}$.
By \eqref{e:hx-hy-below} with $-x$ in place of $x$,
we have
\be\label{e:hopp1}
 |h(y)-h(-x)|^2 \ge \left(|x+y| - \tfrac52\ep\right)^2 \ge |x+y|^2 - 10\ep
\ee
since $|x+y|\le 2$.
By the parallelogram identity, we have
$$
 |h(y)+h(-x)|^2 + |h(y)-h(-x)|^2 = 2(|h(y)|^2+|h(-x)|^2) \le 4
$$
and
$
 |x+y|^2 + |x-y|^2 = 2(|x|^2+|y|^2) = 4 .
$
These relations and \eqref{e:hopp1} imply that
$$
 |h(y)+h(-x)|^2 \le 4 - |h(y)-h(-x)|^2 \le 4 -|x+y|^2 + 10\ep = |x-y|^2 + 10\ep .
$$
Therefore 
\be\label{e:hopp2}
|h(y)+h(-x)| \le |x-y| + \sqrt{10\ep}
\ee
for all $x,y\in S^{n-1}$.
In particular, substituting $y=x$ yields that 
\be\label{e:hopp3}
 |h(x)+h(-x)| \le \sqrt{10\ep}
\ee
for all $x\in S^{n-1}$. 
By the triangle inequality, \eqref{e:hopp2} and \eqref{e:hopp3} imply that
\be\label{e:hx-hy-above}
 |h(x)-h(y)| \le |h(x)+h(-x)| + |h(y)+h(-x)| \le |x-y| + \ep_1
\ee
where $\ep_1=2\sqrt{10\ep}$.
Note that $\ep_1\le\frac15$ since $\ep\le 10^{-3}$.

\smallbreak
{\it Step 2.}
Let $Z$ be a maximal $\ep_1$-separated subset of $S^{n-1}$.
The inequality \eqref{e:hx-hy-above} implies that the restriction $h|_Z$
is 2-Lipschitz: $|h(x)-h(x)|\le 2 |x-y|$ for any $x,y\in Z$.
Since $T_pM$ is isometric to $\R^n$, 
Kirszbraun's theorem \cite{Kir} implies that $h|_Z$ admits
a 2-Lipschitz extension $\tilde h\co\R^n\to T_pM$.
We need only the restriction of $\tilde h$ to the unit sphere.

Pick $x\in S^{n-1}$. By the maximality of $Z$, there exists $z\in Z$
such that $|x-z|\le\ep_1$.
The 2-Lipschitz continuity of $\tilde h$ implies that $|\tilde h(x)-\tilde h(z)| \le 2\ep_1$,
and \eqref{e:hx-hy-above} implies that $|h(x)-h(z)|\le 2\ep_1$.
Since $\tilde h(z)=h(z)$, it follows that
\be\label{e:tildeh-h}
 |\tilde h(x) - h(x) | \le 4\ep_1 \le \tfrac45 .
\ee
This and \eqref{e:hx-below} imply that
$
 |\tilde h(x)| \ge |h(x)|-\tfrac45 \ge \tfrac15 - \tfrac52\ep > 0 .
$
Thus we can define a continuous map $\phi\co S^{n-1}\to \pd B$ by
$$
 \phi(x) = \frac{\tilde h(x)}{|\tilde h(x)|}, \qquad x\in S^{n-1} .
$$

\smallbreak
{\it Step 3.} We show that $\phi$ maps $S^{n-1}$ surjectively onto $\pd B$.
Arguing by contradiction, suppose that there exists 
$w_0\in\pd B\setminus\phi(S^{n-1})$.
Then $\phi$ is a map from $S^{n-1}$ to the set $\pd B\setminus\{w_0\}$, which is homeomorphic to $\R^{n-1}$.
Hence, by the Borsuk-Ulam Theorem, there exists $x_0\in S^{n-1}$ such that
$\phi(x_0)=\phi(-x_0)$.
This means that the vectors $\tilde h(x_0)$ and $\tilde h(-x_0)$ are positively proportional:
\be\label{e:tildeh-prop}
  \tilde h(-x_0) = \lambda \tilde h(x_0) \qquad\text{for some $\lambda>0$} .
\ee

Let $u=h(x_0)$ and $v=h(-x_0)$.
By \eqref{e:hx-below} we have $1-\frac52\ep\le |u|,|v|\le 1$. Therefore
\be\label{e:plus-product}
 \langle u,u-v\rangle + \langle v,u-v\rangle = |u|^2-|v|^2 \ge (1-\tfrac52\ep)^2-1 \ge -5\ep
\ee
On the other hand, $|u-v|\ge 2-\frac52\ep$ by \eqref{e:hx-hy-below}. Hence
\be\label{e:minus-product}
 \langle u,u-v\rangle - \langle v,u-v\rangle = |u-v|^2 \ge (2-\tfrac52\ep)^2 \ge 4-10\ep .
\ee
Adding \eqref{e:minus-product} to \eqref{e:plus-product} and dividing by two yields that
$
 \langle u,u-v\rangle \ge 2 - \tfrac{15}2\ep 
$.
Since $|\tilde h(x_0)-u|\le 4\ep_1$ by \eqref{e:tildeh-h}
and $|u-v|\le 2$, it follows that
$$
 \langle \tilde h(x_0),u-v\rangle \ge \langle u,u-v\rangle - |\tilde h(x_0)-u|\cdot|u-v| 
 \ge 2 - \tfrac{15}2\ep - 8\ep_1 > 0
$$
where the last inequality follows from the bounds 
$\ep_1\le\frac15$ and $\ep\le 10^{-3}$.
Similarly, switching the roles of $x_0$ and $-x_0$ we obtain that
$$
  \langle \tilde h(-x_0),u-v\rangle =  - \langle \tilde h(-x_0),v-u\rangle < 0 .
$$
Thus the products $\langle \tilde h(x_0),u-v\rangle$ and $\langle\tilde h(-x_0),u-v\rangle$
have opposite signs.
This contradicts \eqref{e:tildeh-prop},
therefore $\phi$ is surjective.

\smallbreak
{\it Step 4.}
The surjectivity of $\phi$ implies that for every $v_0\in \pd B$ there exists 
$z\in S^{n-1}$ such that
\be\label{e:angle-to-hz}
  \angle(v_0,h(z)) \le \arcsin \frac {2\ep_1}{|h(z)|} .
\ee
Indeed, let $x\in S^{n-1}$ be such that $\phi(x)=v_0$. By the construction of $Z$,
there exists $z\in Z$ such that $|x-z|\le\ep_1$.
Then $|\tilde h(x)-h(z)|\le 2\ep_1$ since $\tilde h$
is a 2-Lipschitz extension of $h|_Z$.
The direction of $v_0=\phi(x)$ is the same as that of $\tilde h(x)$, hence
\be\label{e:angle-to-hz0}
 \sin\angle (v_0,h(z)) = \sin\angle(\tilde h(x),h(z)) \le \frac{|\tilde h(x)-h(z)|}{|h(z)|}
  \le \frac{2\ep_1}{|h(z)|}
\ee
where the first inequality follows from the Euclidean law of sines
in the triangle with vertices $0$, $\tilde h(x)$, $h(z)$.
Also note that $|\tilde h(x)-h(z)|<2\ep_1\le\frac25<|h(z)|$ by \eqref{e:hx-below}.
Hence $|\tilde h(x)-h(z)|$ is not the largest side of this Euclidean triangle
and therefore $\angle (v_0,h(z))<\frac\pi2$.
This and \eqref{e:angle-to-hz0} imply \eqref{e:angle-to-hz}.

\smallbreak
{\it Step 5.}
Now we prove \eqref{e:flat0}. Let $r_0=\inj_M(p)$.
We may assume that $r_0<1$, otherwise \eqref{e:flat0} holds trivially.
Since $|\Sec_M|\le\ep<1$ and $r_0<1<\pi$, 
Klingenberg's Lemma (see \cite[Lemma 16 in Ch.\ 5]{Pe})
implies that there exists a geodesic loop
$\gamma$ of length $2r_0$ in $M$ starting and ending at~$p$.
Equivalently, there exists $v\in T_pM$ such that $|v|=2r_0$ and $\exp_p(v)=p$.
Let $v_0=v/2r_0$, then $v_0\in\pd B$.
By the result of Step~4, there exists $z\in S^{n-1}$ satisfying \eqref{e:angle-to-hz}.
Define $u=h(z)$, $a=|u|=|h(z)|$, and $\beta=\angle(v,u)=\angle(v_0,u)$,
then \eqref{e:angle-to-hz} takes the form
\be\label{e:betabound}
  \beta\le \arcsin\frac{2\ep_1}{a}
\ee
By the definition of $h$ and \eqref{e:hx-below} we have
\be\label{e:acloseto1}
  1-\tfrac52\ep \le a \le 1
\ee
We apply Lemma \ref{l:toponogov-lite} with parameters $K=\ep$ and $r=2r_0<2$ and obtain that
$$
 d_M(p,f(z)) = d_M(\exp_p(v),\exp_p(u)) \le \curlyvee_{-\ep}(2r_0,a,\beta)
$$
(recall that $\exp_p(u)=\exp_p(h(z))=f(z)$ by the definition of~$h$).
On the other hand,
$
 d_M(p,f(z)) = |h(z)| = a
$
by the definition of~$h$. Thus
\be\label{e:distbound}
 a \le \curlyvee_{-\ep}(2r_0,a,\beta) .
\ee

We deduce \eqref{e:flat0} from
\eqref{e:betabound}, \eqref{e:acloseto1}, and \eqref{e:distbound}
via elementary hyperbolic geometry.
Let $\tilde M=M^2_{-\ep}$ and construct points $\tilde p,\tilde p_1,\tilde z\in\tilde M$
such that $d_{\tilde M}(\tilde p,\tilde p_1)=2r_0$, $d_{\tilde M}(\tilde p,\tilde z)=a$
and $\angle\tilde z\tilde p\tilde p_1=\beta$.
By construction we have $d_{\tilde M}(\tilde z,\tilde p_1)=\curlyvee_{-\ep}(2r_0,a,\beta)$,
hence, by \eqref{e:distbound},
\be\label{e:ephyp1}
 d_{\tilde M}(\tilde z,\tilde p) \le d_{\tilde M}(\tilde z,\tilde p_1)  .
\ee
Let $\ell\subset\tilde M$ be the geodesic line through $\tilde p$ and $\tilde p_1$.
Let $\tilde q\in\ell$ be the orthogonal projection of $\tilde z$ to~$\ell$
and define $r_1=d_{\tilde M}(\tilde p,\tilde q)$.
The inequality \eqref{e:ephyp1} implies that
$
 d_{\tilde M}(\tilde p,\tilde q) \le d_{\tilde M}(\tilde p_1,\tilde q) 
$,
hence $r_1\le\frac12d_{\tilde M}(\tilde p,\tilde p_1) = r_0$.
Thus it suffices to prove that $r_1\ge 1-100\ep$.
Define $b=d_{\tilde M}(\tilde z,\tilde q)$ and let $\tilde z_1\in\tilde M$
be the point symmetric to $\tilde z$ with respect to~$\ell$.
 From the triangle $\triangle\tilde p\tilde z\tilde z_1$
with sides $d_{\tilde M}(\tilde p,\tilde z)=d_{\tilde M}(\tilde p,\tilde z_1)=a$
and angle $\angle \tilde z\tilde p\tilde z_1=2\beta$
one sees that
$$
 2b = d_{\tilde M}(\tilde z,\tilde z_1) = \curlyvee_{-\ep}(a,a,2\beta)
 \le \curlyvee_0(a,a,2\beta) + \tfrac12\ep = 2a\sin\beta + \tfrac12\ep,
$$
where the inequality in the middle follows from \eqref{e:hyperb-vs-eucl}
since $a\le 1$ by \eqref{e:acloseto1}.
Hence
\be\label{e:ephyp2}
 b \le a\sin\beta + \tfrac14\ep \le 2\ep_1+\tfrac14\ep
\ee
by \eqref{e:betabound}.
Now from the triangle $\triangle \tilde q\tilde p\tilde z$
with sides $d_{\tilde M}(\tilde q,\tilde p)=r_1$ and 
$d_{\tilde M}(\tilde q,\tilde z)=b$ 
and angle $\angle \tilde z\tilde q\tilde p=\frac\pi2$ one sees that
$$
 a = d_{\tilde M}(\tilde p,\tilde z) = \curlyvee_{-\ep}(r_1,b,\tfrac\pi2)
 \le \curlyvee_0(r_1,b,\tfrac\pi2) + \tfrac12\ep = \sqrt{r_1^2+b^2} + \tfrac12\ep .
$$
where the inequality again follows from \eqref{e:hyperb-vs-eucl}.
This, \eqref{e:acloseto1}, and \eqref{e:ephyp2} imply that
$$
 r_1^2 \ge (a-\tfrac12\ep)^2 - b^2 \ge (1-3\ep)^2 - (2\ep_1+\tfrac14\ep)^2
 \ge 1 -4\ep_1^2-7\ep
$$
where the last inequality holds since 
$\ep_1\le\frac15$ and $\ep\le 10^{-3}$.
Substituting $\ep_1=2\sqrt{10\ep}$ we obtain that $r_1^2 \ge 1-167\ep$.
Since $\ep\le 10^{-3}$, $\sqrt{1-167\ep} \ge 1-100\ep$,
thus $r_1\ge 1-100\ep$.
Since $\inj_M(p)=r_0\ge r_1$, this proves \eqref{e:flat0}
and Lemma \ref{l:injrad-flat} follows.
\end{proof}

{
Now we prove Proposition \ref{p:injrad}.
We restate the first part of Proposition \ref{p:injrad}
as the following lemma,
which also provides an explicit value of the constant $\cfour$.
}


\begin{lemma}
\label{l:injrad}
Let $K>0$ and let $M, \tilde M$ be complete $n$-dimensional
Riemannian manifolds with $|\Sec_M|\le K$ and $|\Sec_{\tilde M}|\le K$,
and
\be\label{e:injrad00}
 0 < r \le \min\{\tfrac\pi{\sqrt K}, \inj_{\tilde M}(\tilde x) \} .
\ee
Then
\be\label{e:injrad-goal}
 \inj_M(x) \ge r - 10^6 \cdot d_{GH}(B_r^M(x),B_r^{\tilde M}(\tilde x)) .
\ee
%
%
\end{lemma}

\begin{proof}
%
Define 
\be\label{e:injrad-delta}
 \de =  d_{GH}(B_r^M(x),B_r^{\tilde M}(\tilde x)) .
\ee
We may assume that $\de<10^{-6}r$, otherwise \eqref{e:injrad-goal} is trivial. 
We may also assume that $\de>0$, otherwise $\inj_M(x)\ge r$ since $B_r^M(x)$
is isometric to $B_r^{\tilde M}(\tilde x)$.
Below we prove a stronger inequality $\inj_M(x) \ge r - 20\de$
assuming that $0<\de<10^{-6}r$.

First we apply Lemma \ref{l:injrad-flat}
to the smaller ball $B_\rho^M(x)$ where $\rho=10^{-2}r$.
To verify the assumptions of Lemma \ref{l:injrad-flat},
observe that
\be\label{e:Krho^2}
 K\rho^2 = 10^{-4} Kr^2 \le 10^{-4} \pi^2 < 10^{-3} .
\ee
{\ctext since $ Kr^2 \le\pi^2<10$ by \eqref{e:injrad00}.}
Then by \eqref{eq: GH distance of balls}, \eqref{e:injrad00}, and \eqref{e:Krho^2},
$$
 d_{GH}(B_\rho^{\tilde M}(\tilde x),B_\rho^n)\le {\tfrac14} K\rho^3 \le \tfrac14\cdot 10^{-3} \rho.
$$
Since $\rho<r$, \eqref{e:injrad-delta} and the definition of the pointed GH distance
imply that
$$
d_{GH}(B_\rho^M(x),B_\rho^{\tilde M}(\tilde x)) \le 3\de .
$$
Hence, by the triangle inequality for $d_{GH}$,
\be\label{e:rho-flatness}
 d_{GH}(B_\rho^M(x),B_\rho^n) \le 3\de +  \tfrac14\cdot 10^{-3}\rho < 10^{-3}\rho
\ee
since $\de \le 10^{-6} r=10^{-4}\rho$.
By \eqref{e:Krho^2} and \eqref{e:rho-flatness},
the assumptions of Lemma \ref{l:injrad-flat} are satisfied 
for $p=x$ and $\rho$ in place of~$r$.
Now  Lemma \ref{l:injrad-flat} implies that
\be\label{e:r0below}
 \inj_M(x) > \tfrac9{10}\rho > 20\de .
\ee
We need this preliminary lower bound for the subsequent argument to work.

Let $r_0=\inj_M(x)$ and assume towards a contradiction that
\be\label{e:r0bounds}
r_0 < r-20\de .
\ee
Since $\Sec_M\le K$ and $r_0<r\le\frac\pi{\sqrt K}$, 
by Klingenberg's Lemma (see \cite[Lemma 16 in Ch.\ 5]{Pe}) 
there is a geodesic loop
$\gamma$ of length $2r_0$ in $M$
starting and ending at~$x$.
Let $y$ be the midpoint of this loop
and $\gamma_1$, $\gamma_2$ the two halves of~$\gamma$
between $x$ and $y$.
Note that $\gamma_1$ and $\gamma_2$
are minimizing geodesics and $d_M(x,y)=r_0$.

{
By \eqref{e:injrad-delta}, there is a correspondence $\mathcal R$
between the balls $B_r^M(x)$ and $B_r^{\tilde M}(\tilde x)$
with distortion at most $2\de$, see \cite[Theorem 7.3.25]{BBI}.
Recall that a \textit{correspondence} between metric spaces $X$ and $\tilde X$
is a subset $\mathcal R\subset X\times\tilde X$ with surjective 
coordinate projections to $X$ and~$\tilde X$,
and the \textit{distortion} of $\mathcal R$ is defined by
$$
 \operatorname{dis}\mathcal R := \sup \{|d_X(x,y)-d_{\tilde X}(\tilde x,\tilde y)| : (x,\tilde x), (y,\tilde y)\in\mathcal R \} 
$$
We fix $\mathcal R$ with $ \operatorname{dis}\mathcal R\le 2\de$
for $X=B_r^M(x)$ and $\tilde X=B_r^{\tilde M}(\tilde x)$
and say that $y\in B_r^M(x)$ and $\tilde y\in B_r^{\tilde M}(\tilde x)$
\textit{correspond} to each other if $(y,\tilde y)\in\mathcal R$.
Since we are working with pointed GH distance, the centers $x$ and $\tilde x$ 
correspond to each other.}


Pick $\tilde y\in B_r^{\tilde M}(\tilde x)$
corresponding to the point $y$ constructed above.
Then
$$
 d_{\tilde M}(\tilde x,\tilde y)\le d_M(x,y)+2\de = r_0+2\de <r-18\de
$$
by \eqref{e:r0bounds}.
Since $\inj_{\tilde M}(\tilde x)\ge r$, it follows that
there is a point $\tilde z\in B^{\tilde M}_r(\tilde x)$
such that $\tilde y$ belongs to the minimizing geodesic from $\tilde x$ to $\tilde z$
and $d_{\tilde M}(\tilde y,\tilde z)=18\de$.
Pick $z\in B_r^M(x)$ corresponding to~${\tilde z}$ and 
let $a=d_M(y,z)$.
Since
{$\operatorname{dis}\mathcal R\le 2\de$}
and the triangle inequality in $\tilde M$ turns to equality 
for $\tilde x,\tilde y,\tilde z$,
we have
$$
 r_0+a=d_M(x,y)+d_M(y,z) \le {d_M(\tilde x,\tilde z)+4\de  } \le d_M(x,z)+6\de .
$$
Thus
\be\label{e:topo0}
 d_M(x,z) \ge r_0+a-6\de .
\ee
{Also note that
$
 |a-18\de| = |d_M(y,z)-d_{\tilde M}(\tilde y,\tilde z)| \le \operatorname{dis}\mathcal R\le 2\de .
$
Therefore
\be\label{e:abounds}
 16\de \le a \le 20\de < r_0
\ee
where the last inequality follows from \eqref{e:r0below}.}

Let $\gamma_3$ be a minimizing geodesic between $y$ and $z$.
Consider the angles $\angle(\gamma_3,\gamma_1)$ and $\angle(\gamma_3,\gamma_2)$ at $y$.
Their sum equals $\pi$, hence at least one of them is no greater than~$\frac\pi2$.
{Assume w.l.o.g.\ that $\angle(\gamma_3,\gamma_1)\le\frac\pi2$
and let $u,v\in T_yM$ be the vectors tangent to $\gamma_3$ and $\gamma_1$, resp.,
and such that $|u|=|v|=a$. Since $\angle(u,v)=\angle(\gamma_3,\gamma_1)\le\frac\pi2$,
we have $|u-v|\le \sqrt2 a$.
Note that $z=\exp_y(u)$. Let $x'=\exp_y(v)$.
Then, by \eqref{e:abounds}, $x'$ lies on $\gamma_1$
at distance $r_0-a$ from~$x$.
By Lemma \ref{l:toponogov-lite} applied to $p=y$
and $a$ in place of $r$,
$$
 d_M(x',z) \le |u-v| + \tfrac12 Ka^3 \le \sqrt2a + \tfrac12 Ka^3 .
$$
Hence
\be\label{e:topo01}
 d_M(x,z) = d_M(x,x') + d_M(x',z) \le r_0-a + \sqrt2a + \tfrac12 K a^3 .
\ee
By \eqref{e:abounds} and the assumption $\de<10^{-6}r$,
$$
\tfrac12 K a^2 \le 200 K\de^2 < 10^{-9} Kr^2 < 10^{-8}
$$
since $ Kr^2 \le\pi^2<10$ by \eqref{e:injrad00}.
This and \eqref{e:topo01} imply that
$$
 d_M(x,z) \le r_0 + (\sqrt2+10^{-8} -1 )a < r_0+\tfrac12a .
$$
This and \eqref{e:topo0} imply that $a<12\de$.
This contradicts \eqref{e:abounds},
hence the assumption \eqref{e:r0bounds} was false.
Thus $r_0\ge r-20\de$ and Lemma \ref{l:injrad} follows.
}
\end{proof}

\begin{proof}[Proof of Proposition \ref{p:injrad}]
Lemma \ref{l:injrad} implies the first claim of Proposition \ref{p:injrad}
for any $\cfour\ge 10^6$.
To prove the second one,
let $f\co M\to\tilde M$ be a $(1+\ep,\de)$-quasi-isometry
and $\rho=\min\{\inj_{\tilde M} , \tfrac\pi{\sqrt K} \}$.
By Lemma \ref{l:QI-for-balls},
$$
 d_{GH}(B^M_\rho(x),B^{\tilde M}_\rho(f(x))) \le 2\ep\rho+{5}\de
$$
for all $x\in M$.
Hence by Lemma \ref{l:injrad},
$$
 \inj_M(x) \ge \rho - 10^6 (2\ep\rho+{5}\de)
 = (1- 2\cdot 10^6 \ep) \rho -  {5} \cdot 10^6 \de
$$
for all $x\in M$. Thus the second claim of Proposition \ref{p:injrad}
holds for any $\cfour\ge {5}\cdot 10^6$.
\end{proof}


\section{Proof of Theorem \ref{t:manifold}}
\label{sec:manifold-proof}

Similarly to the proof of Theorem \ref{t:surface}, 
we first observe that the the statement of Theorem \ref{t:manifold}
is scale invariant and it suffices to prove it for $r=1$.
When $r=1$, Theorem \ref{t:manifold} is equivalent to the following proposition
with $\de_0(n)={\sigma_2}(n)>0$.

\begin{proposition}
\label{p:manifold}
For every positive integer $n$
there exists $\de_0=\de_0(n)>0$
such that the following holds.
Let $0<\de<\de_0$ and let $X$ be a metric space
which is $\de$-intrinsic and $\de$-close to $\R^n$ at scale~1.
Then there exists a complete $n$-dimensional Riemannian manifold $M$
such that

1. There is a $(1+\Cone\de,\Cone\de)$-quasi-isometry from $X$ to $M$.

2. The sectional curvature $\Sec_M$ of $M$ satisfies $|\Sec_M|\le \Ctwo\de$.

3. The injectivity radius of $M$ is bounded below by $1/2$.
\end{proposition}

The proof of Proposition \ref{p:manifold} occupies
the rest of this section, which is split into several subsections.

\begin{remark}\label{de 0 value}
{In the proof of Proposition \ref{p:manifold} the bounds $C_j$ etc.\ are constructed via explicit arguments.
Thus, by following the steps of the proof, one can obtain
an explicit formula for the value  $\de_0(n)$.
However, the details of this go outside the framework of this {\ctext paper.}}
\end{remark}

%
We recycle the letter $r$ for use in other notation.
We fix $n$ and assume that a metric space $X$ satisfies the assumption
of the proposition for a sufficiently small $\de>0$.

Fix a maximal $\frac1{100}$-separated set $X_0\subset X$.
We say that two points $x,y\in X_0$ are \textit{adjacent}
if $d_X(x,y)<1$
and say that they are \textit{neighbors} if $d_X(x,y)<\tfrac12$.

The adjacency relation defines a graph which
we refer to as the \textit{adjacency graph}.
The set of vertices of this graph is $X_0$
and the edges are between all pairs of adjacent points.
We need the following properties of this graph.

\begin{lemma}
\label{l:adjgraph}
1. The adjacency graph is connected.

2. Its vertex degrees are bounded by a {number $N_1(n)$} depending only on $n$.
\end{lemma}

\begin{proof}
1. Let $x,y\in X_0$.
Since $X$ is $\de$-intrinsic, there is a $\de$-chain
$x_1,\dots,x_N\in X$ with $x_1=x$ and $x_N=y$.
For each $x_i$, there is a point $x_i'\in X_0$
with $d_X(x_i,x_i')\le\frac1{100}$.
By the triangle inequality, $d_X(x_i',x_{i+1}')<2\de+\frac1{50}<1$ for all $i$,
and we may assume that $x_1'=x$ and $x_N'=y$.
Then the sequence $x_1',\dots,x_N'$
is a path connecting $x$ to $y$ in the adjacency graph.

2. Let $q\in X_0$.
Since $d_H(B_1(q),B_1^n)<\de$, there exists a $2\de$-isometry
$f\co B_1(q)\to B_1^n$.
Let $Y=X_0\cap B_1(q)$ be the set of points
adjacent to~$q$. Since $Y$ is $\frac1{100}$-separated, its image
$f(Y)$ is a $(\frac1{100}-2\de)$-separated subset of $B_1^n$.
We may assume that $\de$ is so small that $\frac1{100}-2\de>\frac1{200}$.
Then the cardinality of~$Y$
is no greater than the maximum possible number of $\frac1{200}$-separated points in $B_1^n$.
\end{proof}

Lemma \ref{l:adjgraph} implies that the set $X_0$ is at most countable.
In the sequel we assume that $X_0$ is countably infinite, $X_0=\{q_i\}_{i=1}^\infty$.
In the case when $X_0$ is finite, the proof is the same
except that the indices are restricted to a finite set.

\subsection{Approximate charts}
\label{subsec:app charts}

Fix a collection of points $\{p_i\}_{i=1}^\infty$ in $\R^n$ such that
the Euclidean unit balls $D_i:=B_1(p_i)$ are disjoint.
For $r>0$, we denote by $D_i^r$ the Euclidean ball $B_r(p_i)\subset\R^n$.

Recall that $X_0=\{q_i\}_{i=1}^\infty$.
For each $i\in\N$ we have $d_{GH}(B_1(q_i),D_i)<\de$
since $D_i$ is isometric to $B_1^n$.
Recall that here we are dealing with pointed GH
distance {between}  balls where the centers
are distinguished points.
Hence there exists a $2\de$-isometry $f_i\co B_1(q_i)\to D_i$
such that $f_i(q_i)=p_i$.

We fix $2\de$-isometries $f_i\co B_1(q_i)\to D_i$, $i\in\N$,
for the rest of the proof.
The balls $D_i$ and the maps $f_i$ play the role
of coordinate charts in~$X$.
The next lemma provides {a kind of transition map}
between charts.

\begin{lemma}
\label{l:transition}
For each pair of adjacent points $q_i,q_j\in X_0$, there exists
an affine isometry $A_{ij}\co\R^n\to\R^n$ such that 
\be
\label{e:transition}
 |A_{ij}(f_i(x))-f_j(x)|<\Ctwentytwo\de
\ee
for every $x\in B_1(q_i)\cap B_1(q_j)$.
\end{lemma}

\begin{proof} 
Let $Y=B_1(q_i)\cap B_1(q_j)$.
Since $d_{GH}(B_1(q_i),B_1^n)<\de$ and $q_j\in B_1(q_i)$, there exists
$x_0\in Y$ such that
$$
 \max\{d_X(x_0,q_i), d_X(x_0,q_j) \} < {\tfrac12+3\de} .
$$
{
Define maps $h_1,h_2\co Y\to\R^n$ by $h_1(x)=f_i(x)-f_i(x_0)$
and $h_2(x)=f_j(x)-f_j(x_0)$. 
Since $B_{1/2-3\de}(x_0)\subset Y\subset B_{3/2+3\de}(x_0)$
and $f_i,f_j$ are $2\de$-isometries, $h_1$ and $h_2$ satisfy
the assumptions of Lemma \ref{l:rotation} with parameters
{$3/2+3 \delta, 3/2-3 \delta,  2\delta$}
in place of $R,r,\de$, respectively. Hence by Lemma \ref{l:rotation}
}
there exists an orthogonal map $U\co\R^n\to\R^n$ such that
\be\label{e:transition2}
 |U(h_1(x))-h_2(x)| < {12\Cthirteen n\de}
\ee
for all $x\in Y$. Now define $A_{ij}\co\R^n\to\R^n$ by
\be\label{e:transition2b}
 A_{ij}(y) = U(y-f_i(x_0)) + f_j(x_0), \qquad y\in\R^n .
\ee
This definition and \eqref{e:transition2} implies \eqref{e:transition}.
\end{proof}

We fix maps $A_{ij}$ constructed
in Lemma \ref{l:transition} for the rest of the proof.
We may assume that $A_{ji}=A_{ij}^{-1}$ for all $i,j$
and $A_{ii}$ is the identity map.

\begin{lemma}
\label{l:3way}
Let $q_i,q_j,q_k\in X_0$ be three pairwise adjacent points.
Then
\be\label{e:3way}
 |A_{jk}(A_{ij}(x))-A_{ik}(x)|< \Ctwentythree\de
\ee
for all $x\in D_i$.
\end{lemma}

\begin{proof} 
{Let} $a=f_i(q_j)$ and $b=f_i(q_k)$.
{Consider the intersection of Euclidean balls}
$$
 Z:=D_i\cap B_{1-2\de}(a)\cap B_{1-2\de}(b) \subset\R^n .
$$
Let $x\in Z$.
{Since $f_i$ is a $2\de$-isometry,
there is $q\in B_1(q_i)$ such that
$|f_i(q)-x|<2\de$.
Note that $q$ belongs to the balls $B_1(q_j)$ and $B_1(q_k)$ as well.
Then
\be\label{e:3way1}
 |A_{ik}(x)-f_k(q)| \le |A_{ik}(f_i(q))-f_k(q)| + |A_{ik}(x) - A_{ik}(f_i(q))| < (\Ctwentytwo+2) \de
\ee
by \eqref{e:transition} and the fact that $A_{ik}$ is an isometry.
Similarly,
\be\label{e:3way2}
 |A_{ij}(x)-f_j(q)| < (\Ctwentytwo+2) \de
\ee
and therefore
\be\label{e:3way3}
 |A_{jk}(A_{ij}(x))-f_k(q)| 
 \le |A_{jk}(A_{ij}(x))-A_{jk}(f_j(q))| + \Ctwentytwo\de
 \le (2\Ctwentytwo+2) \de
\ee
where the first inequality follows from \eqref{e:transition}
and the second one from \eqref{e:3way2} and the fact that $A_{jk}$ is an isometry.
Now \eqref{e:3way1} and \eqref{e:3way3} imply that
\be\label{e:3way-semifinal}
|A_{jk}(A_{ij}(x)) - A_{ik}(x)| < (3\Ctwentytwo+4)\de =: \de_1
\ee
for all $x\in Z$.
}

{
Observe that $Z$ contains a ball of radius $\frac13-3\de$.
Indeed, consider the point $p=\frac13(p_i+a+b)$.
By the triangle inequality,
$
 |p-p_i| = \tfrac13 |(a-p_i) + (b-p_i) | \le \tfrac23
$
since $a,b\in D_i=B_1(p_i)$.
Hence $D_i$ contains the ball $B_{1/3}(p)$.
Similarly, since {$|b-a|<1+2\de$}, we have $|p-a|<\frac23+\de$ and
$|p-b|<\frac23+\de$, hence the ball $B_{1/3-3\de}(p)$ is contained in
$B_{1-2\de}(a)$ and in $B_{1-2\de}(b)$.

Thus \eqref{e:3way-semifinal} holds for all $z\in B_{1/4}(p)$.
The affine map $A=A_{jk}\circ A_{ij}-A_{ik}$ can
be written in the form
$A(x) = A(p) + L(x-p)$ where $L\co\R^n\to\R^n$ is a linear map.
Then \eqref{e:3way-semifinal} implies that $|A(p)|<{\ctext 2\de_1}$
and $|L(v)|<{\ctext  2\de_1}$ for all $v\in B_{1/4}^n$.
Hence $\|L\|\le {\ctext 8\de_1}$.
Therefore for all $x\in B_2(p)$,
$$
|A(x)|\le |A(p)|+|L(x-p)| <{\ctext 17\de_1 = 17}(3\Ctwentytwo+4)\de  =: \Ctwentythree \de
$$
Since $D_i\subset  B_2(p)$, it follows that \eqref{e:3way} holds for all $x\in D_i$.
}
\end{proof}


\begin{lemma}
\label{l:centerdistance}
Let $q_i,q_j,q_k\in X_0$. Then

1. If $q_i$ and $q_j$ are adjacent, then
\beq\label{e C24}
 \bigg|{|A_{ij}(p_i)-p_j|}-d_X(q_i,q_j)\bigg| < {\Ctwentyfour}\de .
 \eeq

2. If $q_k$ is adjacent to both $q_i$ and $q_j$, then
$$
 \bigg|{|A_{ik}(p_i)-A_{jk}(p_j)|}-d_X(q_i,q_j)\bigg| < {\Ctwentyfour}\de
$$
\end{lemma}

\begin{proof}
The first assertion follows from the second one by setting $k=j$
(recall that $A_{jj}$ is the identity map).
Let us prove the second assertion.

Since $p_i=f_i(q_i)$, \eqref{e:transition} implies that
$A_{ik}(p_i)$ is $\Ctwentytwo\de$-close to $f_k(q_i)$.
Similarly, $A_{jk}(p_j)$ is $\Ctwentytwo\de$-close to $f_k(q_j)$.
Hence the distance $|A_{ik}(p_i)-A_{jk}(p_j)|$ differs from $|f_k(q_i)-f_k(q_j)|$
by at most $2\Ctwentytwo\de$.
In its turn, the distance
$|f_k(q_i)-f_k(q_j)|$ differs from $d_X(q_i,q_j)$
by at most $2\de$ because $f_j$ is a $2\de$-isometry.
Thus $|A_{ik}(p_i)-A_{jk}(p_j)|$ differs from $d_X(q_i,q_j)$
by at most $(2\Ctwentytwo+2)\de={\Ctwentyfour}\de$ and the lemma follows.
\end{proof}

\begin{lemma}
\label{l:nearcenter}
For every $i\in\N$ and every $x\in D_i^{1/3}$ there exist $j\in\N$
such that $q_i$ and $q_j$ are neighbors and
$A_{ij}(x)\in D_j^{1/50}$.
\end{lemma}

\begin{proof}
Since $f_i$ is a $2\de$-isometry from $B_1(q_i)$ to $D_i$,
there exists $y\in B_1(q_i)\subset X$ such that $|f_i(y)-x|\le 2\de$.
Since $X_0$ is a $\frac1{100}$-net in $X$,
there is a point $q_j\in X_0$ such that $d_X(y,q_j)\le\frac1{100}$.
For this point $q_j$ we have
$$
 |x-f_i(q_j)|<|f_i(y)-f_i(q_j)|+2\de < d_X(y,q_j) + 4\de \le \tfrac1{100}+4\de
$$
since $f_i$ is a $2\de$-isometry.
This and the fact that $x\in D_i^{1/3}$
imply that
$$
 |p_i-f_i(q_j)|<\tfrac13+\tfrac1{100}+4\de  .
$$
Since $p_i=f_i(q_i)$ and $f_i$ is a $2\de$-isometry, it follows that
$$
 d_X(q_i,q_j) < \tfrac13+\tfrac1{100}+6\de < \tfrac12 .
$$
Thus $q_i$ and $q_j$ are neighbors, 
in particular there is a well-defined map $A_{ij}$.
Since $A_{ij}$ is an isometry, we have
$$
|A_{ij}(x)-A_{ij}(f_i(q_j))|= |x-f_i(q_j)|< \tfrac1{100}+4\de .
$$
By \eqref{e:transition} we have $|A_{ij}(f_i(q_j))-f_j(q_j)|< \Ctwentytwo\de$,
hence 
$$
|A_{ij}(x)-p_j|=|A_{ij}(x)-f_j(q_j)|<\tfrac1{100}+ {(\Ctwentytwo+4)}\de < \tfrac1{50}
$$
provided that $\de$ is sufficiently small.
Thus $A_{ij}(x)\in D_j^{1/50}$ as claimed.
\end{proof}

\subsection{Approximate Whitney embedding} \label{AWE}

At this point we essentially forget about the original metric space $X$
and use the collection of balls $D_i\subset\R^n$
and maps $A_{ij}$ from the previous section for the rest of the construction.
{Let $\Om=\bigcup D_i\subset \R^n$ and $\Omega_0= \bigcup D_i^{1/10}$.}

Let $S=\mathbb S^n$ be the unit sphere in $\R^{n+1}$ centered at $e_{n+1}$,
where $e_1,\dots,e_{n+1}$ is the standard basis of {$\R^{n+1}$}.
Note that $S$ contains the points 0 and $2e_{n+1}$.
For every $r>0$ we denote by $S_r$ the set
of points in $S$ lying at distance less than $r$ from
the `north pole' $2e_{n+1}$.
That is, $S_r=S\cap B_r(2e_{n+1})$.

Fix a smooth map 
\be\label{smooth map phi}
{\phi}\co\R^n\to S
\ee
 with the following properties:
\begin{enumerate}
\item ${\phi}(x)=0$ for all $x\in\R^n\setminus B_{1/5}(0)$.

\item ${\phi}|_{B_{1/5}(0)}$ is a diffeomorphism onto $S\setminus\{0\}$.

\item ${\phi}|_{B_{1/10}(0)}$ is a diffeomorphism onto
the spherical cap $S_1$.

\item ${\phi}|_{B_{1/50}(0)}$ is a diffeomorphism onto
the spherical cap $S_{1/10}$

\end{enumerate}

{For algorithmic constructions discussed below, we can assume that ${\phi}$  is given
as a  function that is defined piecewisely by explicit, real analytic formulas.}
For each $i$ let ${\phi}_i(x)={\phi}(x-p_i)$
and define a map $F_i\co\Om\to S\subset\R^{n+1}$ as follows.
If a point $x\in\Om$ belongs to a ball $D_j$,
put
\be\label{e:Fi definition}
 F_i(x) =
 \begin{cases}
 {\phi}_i(A_{ji}(x)), &\text{if $D_j$ is adjacent to $D_i$} \\
\quad \quad 0, &\text{otherwise} .
 \end{cases}
\ee
In particular $F_i(x)={\phi}_i(x)$ if $x\in D_i$.

\begin{lemma}
\label{l:nonzeroF}
If $F_i(x)\ne 0$ for some $x\in D_j^{1/5}$, then $q_i$ and $q_j$ are neighbors.
\end{lemma}

\begin{proof}
The assumption $F_i(x)\ne 0$ implies that
$q_i$ and $q_j$ are adjacent and therefore $F_i(x) = {\phi}_i(A_{ji}(x))$.
Thus ${\phi}_i(A_{ji}(x))\ne 0$ and hence
$
 |A_{ji}(x)-p_i| < \tfrac15
$.
Since $A_{ji}$ is an isometry and $|p_j-x|<\frac15$, we have
$$
|A_{ji}(p_j)-p_i| \le |p_j-x| + |A_{ji}(x)-p_i| < \tfrac25 .
$$
This and Lemma \ref{l:centerdistance}(2) imply that
$d_X(q_i,q_j) < \tfrac25+{\ctext {\Ctwentyfour}}\de < \tfrac12$,
hence $q_i$ and $q_j$ are neighbors.
\end{proof}

Let $E$ be the space of square-summable sequences
$(u_i)_{i=1}^\infty$ in $\R^{n+1}$
equipped with the norm defined by $|u|^2=\sum |u_i|^2$
for $u=(u_i)_{i=1}^\infty$.
This is a Hilbert space naturally isomorphic to $\ell^2$.
Define a map $F\co\Om\to E$ by
\be\label{e:F definition}
F(x)=(F_i(x))_{i=1}^\infty
\ee
Lemma \ref{l:adjgraph} implies that
for every $x\in U$ there are only finitely many indices $i$
such that $F_i(x)\ne 0$.
Therefore the sequence $F(x)\in (\R^{n+1})^\infty$
is finite and hence indeed belongs to $E$.

\begin{lemma}
\label{l:Fsmooth}  
1. $F$ is smooth and {there is $\Ctwentyfive(k)>0$ 
depending only on $n$ and $k$ such that}
\be\label{e:Fsmooth}
 \|F\|_{C^k(\Om)} \le \Ctwentyfive(k)
\ee
for all $k\ge 0$.

2. For every $i\in\N$ the restriction $F|_{D_i^{1/10}}$
is uniformly bi-Lipschitz, that is,
\be
\label{e:Fbilip}
\Ctwentysix^{-1} |x-y| \le |F(x)-F(y)| \le \Ctwentysix |x-y|
\ee
for all $x,y\in D_i^{1/10}$.
\end{lemma}

\begin{proof}
1. Let $x\in D_i$.
By Lemma \ref{l:adjgraph},
there is at most ${N_1(n)}$ indices $j$ such that
$F_j|_{D_i}\ne 0$.
For every such $j$ we have
$\|d^k_xF_j\| \le \|{\phi}\|_{C^k(\R^n)}$,
therefore $\|d^k_xF\| \le {N_1(n)}\cdot \|{\phi}\|_{C^k(\R^n)}=\Ctwentyfive(k)$.

2. The second inequality in \eqref{e:Fbilip} follows from \eqref{e:Fsmooth} when $\Ctwentysix\geq \Ctwentyfive(1)$.
To prove the first one, observe that $\Ctwentysix>0$ {can
be chosen so that} $|F(x)-F(y)|\ge |F_i(x)-F_i(y)|\ge \Ctwentysix^{-1}{|x-y|}$ {for $x,y\in D_i^{1/10}$}
since the $i$th coordinate projection from $E$ to $\R^n$
does not increase distances and $F_i|_{D_i^{1/10}}={\phi}_i|_{D_i^{1/10}}$ is bi-Lipschitz.
\end{proof}

Eq.\ \eqref{e:Fbilip} implies that the
first derivative of $F$ is 
uniformly bi-Lipschitz, i.e.,
\be\label{e:dFbilip}
\Ctwentysix^{-1} |v| \le |d_xF(v)| \le\Ctwentysix |v|
\ee
for all $x\in D_i^{1/10}$ and $v\in\R^n$.

Lemma \ref{l:Fsmooth} implies that for each $i$ the image $\Sigma_i:=F(D_i^{1/10})$
is a smooth submanifold of~$E$.
We are going to apply Theorem \ref{t:surface} to the union 
$\Sigma=\bigcup_i\Sigma_i$ in $E$.
As the first step, we show that these submanifolds lie close to one another.

\begin{lemma}\label{l:sheets}
{There are $\Ctwentyseven=\Ctwentyseven(0)>0$ and $ \Ctwentyseven(m)>0$  such that if}
  $q_i$ and $q_j$ are neighbors and let $x\in D_i^{1/5}$,
then $A_{ij}(x)\in D_j$ and
\beq\label{e C27}
 |F(x)-F(A_{ij}(x))|<\Ctwentyseven(0)\de ,
\eeq
and
\be
\label{e:sheets}
  \|d_x^m(F- F\circ A_{ij})\| < \Ctwentyseven(m)\de
\ee
for all $m\ge 1$.
\end{lemma}

\begin{proof}
By Lemma \ref{l:centerdistance},
$$
|A_{ij}(p_i)-p_j| < d_X(q_i,q_j)+{\ctext \Ctwentyfour}\de < \tfrac12+{\ctext \Ctwentyfour}\de .
$$
Since $A_{ij}$ is an isometry,
$
 |A_{ij}(x)-A_{ij}(p_i)| = |x-p_i| < \tfrac15
$.
Therefore
$$
 |A_{ij}(x)-p_j| \le |A_{ij}(x)-A_{ij}(p_i)| + |A_{ij}(p_i)-p_j|
  <\tfrac12+\tfrac15+{\ctext \Ctwentyfour}\de < 1 ,
$$
hence $A_{ij}(x)\in D_j$.
Since $x$ is an arbitrary point of $D_i^{1/5}$,
we have shown that $A_{ij}(D_i^{1/5})\subset D_j$.

Recall that the number of indices $k$ such that $F_k$ does not
vanish on $D_i\cup D_j$ is bounded by a constant depending only on $n$.
Hence in order to verify \eqref{e:sheets} it suffices to show that
\be
\label{e:sheets1}
 \| d_x^m(F_k-F_k\circ A_{ij})\| < C_m\de
\ee
for every fixed~$k$.
Consider four cases.

\textit{Case 1}: $q_k$ is adjacent to both $q_i$ and $q_j$.
In this case
$$
 F_k|_{D_i^{1/5}}  = {\phi}_k\circ A_{ik}|_{D_i^{1/5}}
$$
and
$$
 F_k\circ A_{ij}|_{D_i^{1/5}} = {\phi}_k\circ A_{jk}\circ A_{ij}|_{D_i^{1/5}} .
$$
Now \eqref{e:sheets1}  {follows} from the fact that
the affine isometries $A_{ik}$ and $A_{jk}\circ A_{ij}$
are $4\Ctwentythree\de$-close on $D_i$
by Lemma \ref{l:3way}.

\textit{Case 2}: $q_k$ is not adjacent to $q_i$ and $q_j$.
This case is trivial because $F_k|_{D_i}$ and $F_k\circ A_{ij}|_{D_i}$
both vanish by definition.


\textit{Case 3}: $q_k$ is adjacent to $q_j$ but not to $q_i$.
In this case $F_k|_{D_i}=0$ by definition. Let us show
that $F_k\circ A_{ij}|_{D_i^{1/5}}$ also vanishes.
Since $d_X(q_k,q_i)\ge 1$, Lemma \ref{l:centerdistance} implies that
$
|A_{kj}(p_k)-A_{ij}(p_i)| > 1 -{\ctext \Ctwentyfour}\de
$.
Hence for every $y\in D_i^{1/5}$,
$$
|A_{kj}(p_k)-A_{ij}(y)| > 1-\tfrac15-{\ctext \Ctwentyfour}\de > \tfrac15 .
$$
Since $A_{kj}=A_{jk}^{-1}$ and $A_{kj}$ is an isometry, this implies that
$
 |p_k-A_{jk}\circ A_{ij}(y)| > \tfrac15
$
and hence
$$
 F_k\circ A_{ij}(y) = {\phi}_k\circ A_{jk}\circ A_{ij}(y) = 0
$$
for every $y\in D_i^{1/5}$.

\textit{Case 4}: $q_k$ is adjacent to $q_i$ but not to $q_j$.
In this case $F_k\circ A_{ij}|_{D_i^{1/5}}=0$, so it suffices to
prove that $F_k|_{D_i^{1/5}}=0$.
Suppose the contrary, then Lemma \ref{l:nonzeroF}
implies that $q_k$ and $q_i$ are neighbors.
Since $q_i$ and $q_j$ are also neighbors,
it follows that $q_k$ and $q_j$ are adjacent. This  contradiction proves the claim.
\end{proof}

\begin{figure}

\psfrag{0}{\hspace{-5mm}$D_j$}
\psfrag{1}{$D_k$}
\psfrag{2}{\hspace{-1mm}$F^{(k)}$}
\psfrag{3}{\hspace{-1mm}$F^{(j)}$}
\psfrag{4}{}
\psfrag{5}{$M$}
\psfrag{6}{$\Sigma_k$}
\psfrag{7}{$\Sigma_j$}%
\psfrag{8}{$P_M$}%

\figcommented{
}

\epscommented{
\includegraphics[width=10.5cm]{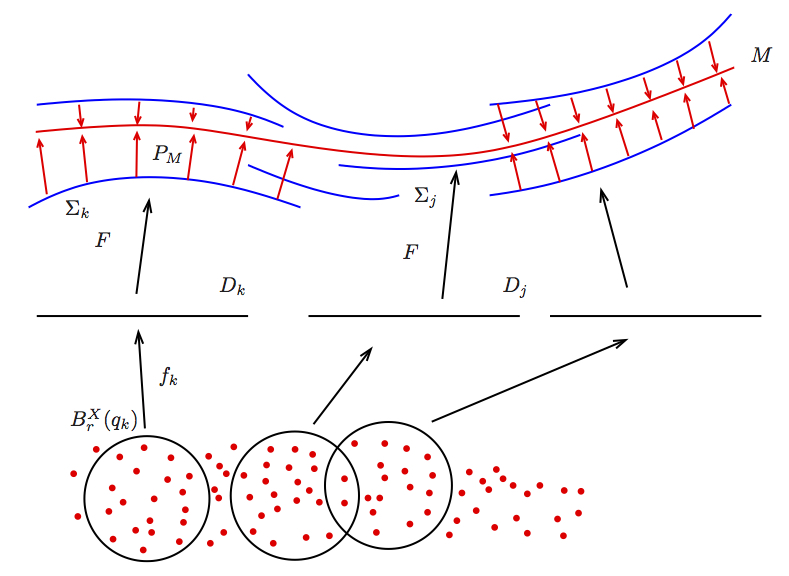}
}
\caption{
A schematic visualisation of the interpolation algorithm `ManifoldConstruction' based on Theorem \ref{t:manifold},
see Section \ref{sec:constructive}. Assume that a finite metric space $(X,d_X)$ is given.
{\ctext First we construct a maximal $(r/100)$-separated subset $X_0=\{q_i,\ i=1,2,\dots,N\}\subset X$  and $r$-neighborhoods $B^X_r(x_i)\subset X$ of points $q_i\in X_0$.
Then, we construct balls $D_i\subset \R^n$
approximating the $r$-balls $B^X_r(q_i)$ and local embeddings $f_i:B^X_r(q_i)\to D_i$. The balls $D_i$ are considered as  local coordinate charts.}
We embed these local charts to an Euclidean space $E=\R^m$ {using}  Whitney-type
embeddings ${\ctext F^{(i)}=F|_{D_i^{r/10}}:D_i^{r/10}\to \Sigma_i}$. Submanifolds $\Sigma_i\subset E$  are denoted by blue curves. Using  the algorithm SubmanifoldInterpolation,
the union $\bigcup_i\Sigma$ is interpolated to a red {submanifold} $M\subset E$. 
When $P_M$ is the normal projector onto $M$, denoted by the red arrows, we can determine
a metric tensor $g_i$ on $P_M(\Sigma_i)$ by pushing forward the Euclidean
metric from $D_i$  to $P_M(\Sigma_i)$  by the map $P_M\circ F|_{D_i}$. 
The metric tensor $g$  on $M$ is obtained by computing a smooth weighted average
of tensors $g_i$.
}
\label{fig: manifold}
\end{figure}

We introduce the following notation for some important subsets of $E$.
For every $i\in\N$ define 
\be\label{Sigma}
\Sigma_i=F(D_i^{1/10})
\quad\text{and}\quad
\Sigma_i^0=F(D_i^{1/50}) .
\ee
Let $\Sigma=\bigcup_i\Sigma_i$
and $\Sigma^0=\bigcup_i\Sigma_i^0$.

Recall that $\Sigma_i$ is a smooth $n$-dimensional submanifold of~$E$.
For a point $x\in\Sigma_i$, we denote by $T_x\Sigma_i$
the tangent space of $\Sigma_i$ at $x$ realized as an affine subspace of~$E$
containing~$x$.
That is, $T_x\Sigma_i$ is the $n$-dimensional affine subspace of $E$
tangent to $\Sigma_i$ at~$x$.

\begin{lemma}
\label{l:Sigma0} {There  is $\Ctwentyeight>\Ctwentyseven(0)$ such that the following holds.}
For every $x\in\Sigma_i$ there exist $j\in\N$
and $y\in\Sigma_j^0$ such that
\be\label{e:Sigma01}
 |x-y| <  {\Ctwentyeight}\de
\ee
and
\be\label{e:Sigma02}
 \angle(T_x\Sigma_i,T_y\Sigma_j) <  {\Ctwentyeight}\de .
\ee
\end{lemma}

\begin{proof}
Since $x\in\Sigma_i$, we have $x=F(z)$ for some $z\in D_i^{1/10}$.
By Lemma \ref{l:nearcenter} there exists $j$ such that
$q_i$ and $q_j$ are neighbors and {$A_{ij}(z)\in D_ j^{1/50}$}.
Let $y=F(A_{ij}(z))$, then $y\in\Sigma_j^0$.
Lemma \ref{l:sheets} for $m=0$ implies that
$$
|x-y| = |F(z)-F(A_{ij}(z))| <  \Ctwentyseven(0)\de
$$
proving \eqref{e:Sigma01}.
To prove \eqref{e:Sigma02}, observe that $T_x\Sigma_i$ and $T_y\Sigma_j$
are parallel to the images of the derivatives
$d_zF$ and $d_{A_{ij}(z)}F$, resp.
The image of $d_{A_{ij}(z)}F$ coincides with
the image of $d_z(F\circ A_{ij})$.
By Lemma \ref{l:sheets} for $m=1$ we have
$$
 \|d_z F - d_z(F\circ A_{ij})\| <  \Ctwentyseven(1)\de .
$$
This and \eqref{e:dFbilip} imply \eqref{e:Sigma02} {with an appropriate 
$\Ctwentyeight>\Ctwentyseven(0)$.}
\end{proof}

We use general metric space notation for subsets of $E$.
In particular, for a set $Z\subset E$ and $r>0$
we denote by ${\mathcal U}_r(Z)$ the $r$-neighborhood of $Z$ in $E$.

\begin{lemma}
\label{l:Sigmai}
$\Sigma\cap {\mathcal U}_{1/2}(\Sigma_i^0)\subset {\mathcal U}_{\Ctwentyeight\de}(\Sigma_i)$
for every $i\in\N$.
\end{lemma}

\begin{proof}
Let $q\in\Sigma\cap {\mathcal U}_{1/2}(\Sigma_i^0)$.
Since $q\in {\mathcal U}_{1/2}(\Sigma_i^0)$,
there exists $y\in D_i^{1/50}$ such that
$|q-F(y)|<\frac12$.
Since $q\in\Sigma$, we have
$q=F(z)$ where $z\in D_j^{1/10}$ for some $j$.
Since the $i$th coordinate projection from $E$ to $\R^{n+1}$
does not increase
distances,
$$
 |F_i(z)-F_i(y)|\le |F(z)-F(y)| = |q-F(y)| < \tfrac12 .
$$
Recall that $F_i(y)={\phi}_i(y)$ because $y\in D_i$.
Since $y\in D_i^{1/50}$, the point ${\phi}_i(y)$
belongs to the spherical cap $S_{1/10}$.
Hence $|F_i(y)-2e_{n+1}|<\frac1{10}$.
Therefore
$$
 |F_i(z)-2e_{n+1}| \le |F_i(z)-F_i(y)|+|F_i(y)-2e_{n+1}|
<\tfrac12+\tfrac1{10}<1 .
$$
Thus $F_i(z)$ belongs to the spherical cap $S_1\subset S\subset\R^{n+1}$,
in particular $F_i(z)\ne 0$.
Hence $F_i(z)={\phi}_i(A_{ji}(z))$ and
therefore $A_{ji}(z)\in{\phi}_i^{-1}(S_1)=D_i^{1/10}$.

Since $F_i(z)\ne 0$, Lemma \ref{l:nonzeroF} implies that $q_i$ and $q_j$
are neighbors. 
Now, by Lemma \ref{l:sheets} (for $m=0$) {and the inequality $\Ctwentyeight>\Ctwentyseven(0)$,} {we have} 
$$
|q-F(A_{ji}(z))| = |F(z)-F(A_{ji}(z))| <   {\Ctwentyeight}\de .
$$
Since  $A_{ji}(z)\in D_i^{1/10}$, this inequality implies that
$$
 q \in {\mathcal U}_{{\Ctwentyeight}\de}(F(D_i^{1/10}))={\mathcal U}_{{\Ctwentyeight}\de}(\Sigma_i) .
$$
Since $q$ is an arbitrary point from the set $\Sigma\cap {\mathcal U}_{1/2}(\Sigma_i^0)$,
the lemma follows.
\end{proof}

\begin{lemma}
\label{l:Sigmai-near-flat}
{There is $\Ctwentynine{\ctext >\Ctwentyseven(0)}$ 
such that}
for every $q\in\Sigma_i^0$
and {$r>0$,} 
\be\label{e:Sigmai-near-flat}
 d_H(\Sigma_i\cap B_r(q),T_q\Sigma_i\cap B_r(q)) <  \Ctwentynine r^2 .
\ee
\end{lemma}

\begin{proof}
By Lemma \ref{l:Fsmooth},
$\Sigma_i=F(D_i^{1/10})$ is a {submanifold} parametrized
by a uniformly bi-Lipschitz smooth map $F|_{D_i^{1/10}}$. 
We may assume that $r<\frac1{50\Ctwentysix}$ where $\Ctwentysix$ is the
bi-Lipschitz constant in (\ref{e:Fbilip}).
{
Indeed,  if $r\ge\frac1{50\Ctwentysix}$ then \eqref{e:Sigmai-near-flat}
holds for any $\Ctwentynine > 50\Ctwentysix$
since the left-hand side of \eqref{e:Sigmai-near-flat} is bounded 
by~$r$.}

Let $q=F(x)$ where $x\in D_i^{1/50}$.
Then every point $q'\in \Sigma_i\cap B_r(q)$ is the image
of some {$x'\in B_{\Ctwentysix r}(x)\subset B_{1/50}(x)\subset D_i^{1/10}$.} 
Hence
$$
 \dist(q',T_q\Sigma_i) \le {\Ctwentynine} r^2,
$$
where $  \Ctwentynine= \Ctwentyfive(2)$ is the uniform bound of the second derivatives of 
$F|_{D_i^{1/10}}$, see (\ref{e:Fsmooth}).
This means that $\Sigma_i$ deviates from its tangent space $T_q\Sigma_i$
within the $r$-ball $B_r(q)$ by distance at most $ \Ctwentynine r^2$.

In addition, the point $q\in\Sigma_i^0=F(D_i^{1/50})$
is separated by a distance at least $\frac1{20\Ctwentysix}>2r$
from the boundary of~$\Sigma_i$. Therefore, for each point
from $T_q\Sigma_i\cap B_r(q)$ there exists a point in $\Sigma_i$
within distance $   \Ctwentynine r^2$.
\end{proof}

The next lemma essentially says that the $\Sigma\subset E$
is $C\de$-close to affine spaces in $E$ 
at a scale of order $\de^{1/2}$.

\begin{lemma}
\label{l:Sigma-near-flat}
{There is $\Cthirtythree>0$  such that} for every $x\in\Sigma_i$
{and $r\ge\de^{1/2}$},
\be\label{e:Sigma-near-flat0}
 d_H(\Sigma\cap B_r(x),T_x\Sigma_i\cap B_r(x)) < \Cthirtythree r^2 .
\ee
\end{lemma}

\begin{proof}
{
We may assume that $r\le\frac14$, otherwise \eqref{e:Sigma-near-flat0}
holds for any $\Cthirtythree\ge 4$ since the left-hand side is bounded by~$r$.}
By Lemma \ref{l:Sigma0}, there exists $j\in\N$ and
$q\in\Sigma^0_j$
such that $|x-q|<\Ctwentyeight\de$ and
$\angle(T_x\Sigma_i,T_q\Sigma_j)<\Ctwentyeight\de$.
Let $A=T_q\Sigma_j$. Observe that
the Hausdorff distance between the affine balls
$T_x\Sigma_i\cap B_r(x)$ and $B_r^A(q)=A\cap B_r(q)$
is bounded by 
$$
|x-q|+r\sin\angle(T_x\Sigma_i,A) <  {\Ctwentyeight}\de +\Ctwentyeight r\de {\  \le 2\Ctwentyeight r^2}
$$
since $\de\le r^2$ {and $\de\le\de^{1/2}\le r$}.
{Assuming that $\Cthirtythree >4\Ctwentyeight$,}
it suffices to verify that
$
 d_H(\Sigma\cap B_r(x),B_r^A(q)) < \frac 12 \Cthirtythree r^2 .
$
By the definition of the Hausdorff distance,
this is equivalent to the following pair of inclusions:
\be\label{e:Sigma-near-flat2}
 \Sigma\cap B_r(x)\subset {\mathcal U}_{\Cthirtythree r^2/2}(B_r^A(q))
\ee
and
\be\label{e:Sigma-near-flat3}
 B_r^A(q)\subset {\mathcal U}_{\Cthirtythree r^2/2}(\Sigma\cap B_r(x)) .
\ee
Since $|x-q|<\Ctwentyeight\de$, we have $B_r(x) \subset B_{r+\Ctwentyeight \de}(q)$ and therefore
$$
 \Sigma\cap B_r(x) \subset \Sigma\cap B_{r+\Ctwentyeight\de}(q)
 \subset \Sigma\cap {\mathcal U}_{r+\Ctwentyeight \de}(\Sigma_j^0) \subset {\mathcal U}_{\Ctwentyeight\de}(\Sigma_j)
$$
where the
last inclusion follows from Lemma \ref{l:Sigmai}
{provided that $r+\Ctwentyeight \de<\frac12$.
The latter follows from the inequality $r\le\frac14$ 
if $\de$ is so small that $\Ctwentyeight\de<\frac14$.}
Hence
\begin{align}\label{e:Sigma-near-flat1}
\Sigma\cap B_r(x) &\subset {\mathcal U}_{\Ctwentyeight\de}(\Sigma_j) \cap B_{r+\Ctwentyeight \de}(q) \\
\nonumber
 &\subset {{\mathcal U}_{\Ctwentyeight \de}(\Sigma_j \cap B_{r+2\Ctwentyeight \de}(q))}
\subset {\mathcal U}_{\Ctwentynine r^2}(B^A_{r+{2}\Ctwentyeight\de}(q))
\end{align}
where the last inclusion follows from 
Lemma \ref{l:Sigmai-near-flat}.
Since 
$$
B^A_{r+2\Ctwentyeight \de}(q)\subset {\mathcal U}_{2\Ctwentyeight \de}(B_r^A(q))
 {\subset {\mathcal U}_{2\Ctwentyeight       r^2}(B_r^A(q))},
$$
this implies \eqref{e:Sigma-near-flat2} {when $\Cthirtythree \geq 2(\Ctwentynine+2\Ctwentyeight    )$}.

It remains to verify \eqref{e:Sigma-near-flat3}.
{
We assume that $\de$ is so small that $\Ctwentyeight\de^{1/2}<1$,
then $\Ctwentyeight\de< \de^{1/2} \le r$.
%
%
Since $|x-q|<\Ctwentyeight \de$, this implies that}
 $|x-q|<r$. 
Let $r_1=r-|x-q|$.
By Lemma \ref{l:Sigmai-near-flat},
$$
 B^A_{r_1}(q) \subset {\mathcal U}_{{\ctext \Ctwentynine} r^2}(\Sigma\cap B_{r_1}(q)) 
 \subset  {\mathcal U}_{{\ctext \Ctwentynine} r^2}(\Sigma\cap B_r(x)) .
$$
Since $B^A_r(q)\subset {\mathcal U}_{r-r_1}(B^A_{r_1}(q))$,  {and inequality ${\ctext \Ctwentynine} >\Ctwentyseven(0)$}
and $r-r_1=|x-q| <  {{\ctext \Ctwentynine} }  \de<  {{\ctext \Ctwentynine} }      r^2$, this implies \eqref{e:Sigma-near-flat3}
{when $\Cthirtythree \geq 4{\ctext \Ctwentynine} . $}
{By choosing $\Cthirtythree$  so that the above inequalities are valid, the lemma follows.}
\end{proof}

\subsection{The manifold $M$}

We choose a positive constant $r_0<1$ such that
\be
\label{e:2C0de0}
\Ctwentyeight r_0 < {\sigma_2}
\ee
where {$\Ctwentyeight$ is the constant from {Lemma \ref{l:Sigma0}}}
and ${\sigma_2}$ is the constant from Theorem~\ref{t:surface}.
Some additional requirements on $r_0$ arise in the course of
the argument below, but the final value of $r_0$ depends only on~$n$.

We may assume that the constant $\de_0$ in
Proposition \ref{p:manifold} {satisfies $\de_0< r_0^2$
(see Lemma \ref{l:Sigma-near-flat}).}
Then, for $\de<\de_0$, Lemma \ref{l:Sigma-near-flat} implies that
\be \label{20.7.1}
d_H(\Sigma\cap B_{r_0}(x),T_x\Sigma_i\cap B_{r_0}(x)) <  \Cthirtythree  r_0^2
\ee
for every $x\in\Sigma_i$.
This and \eqref{e:2C0de0} imply that the assumptions of
Theorem \ref{t:surface} are satisfied for $\Sigma$ in place of $X$,
$r_0$ in place of $r$, $ \Cthirtythree  r_0^2$ in place of $\de$,
and $T_x\Sigma_i$ in place of $A_x$ (for $x\in\Sigma_i$).
Then, the conclusion of Theorem \ref{t:surface}
with these settings, {and with $ \Cthirtysix(m)= \Cthirtythree\Cten(m)$
and $\Cthirtysix'={\Cthirtythree\Cnine}$,} is the following lemma.

\begin{lemma}
\label{l:submanifold}
{There are $\ctwelve>0$ and $\Cthirtysix(m)>0,$
$\Cthirtysix'>0$ such that the following holds.} 
If {$r_0<\ctwelve$}  
 and {$\de< r_0^2$}, then
there exists a closed $n$-dimensional smooth submanifold $M\subset E$
such that

1. $d_H(\Sigma,M)< {5}\Cthirtythree  r_0^2<\frac{1}{10}r_0<\frac1{100}$.

2. The second fundamental form of $M$ at every point is {bounded by $\Cthirtythree\Ceight$}.

3. 
{$\Reach(M)\ge r_0/3$.}

4. The normal projection $P_M\co {\mathcal U}_{r_0/3}(M)\to M$
satisfies {for all $x\in \mathcal U_{r_0/3}(M)$}
\beq\label{e C36}
 \|d^m_x P_M\| <\Cthirtysix(m) 
 r_0^{2-m}, \qquad {m\ge 2},
\eeq
{
and
\be\label{e:dxPM}
 \|d_x P_M-P_{\vec T_yM}\| < \Cthirtysix(1)r_0 < \tfrac1{10}, \qquad y=P_M(x) .
\ee
}

5. $\angle(T_x\Sigma_i,T_{y}M) < 
\Cthirtysix'
r_0 {< \frac1{10}}$ for every $x\in\Sigma_i$
{and $y=P_M(x)$}.
\qed
\end{lemma}

{
These inequalities in Lemma \ref{l:submanifold} that are not present in Theorem \ref{t:surface}
follow from the choice of $\ctwelve$ (and thus $r_0$) sufficiently small.
The inequality $d_H(\Sigma,M)<\frac{1}{10}r_0$ ensures that $\Sigma$
lies `deep inside' the domain of~$P_M$.
The second inequality in \eqref{e:dxPM} implies that 
\be\label{e:dPMbound}
 \|d_xP_M\|<1+\tfrac1{10}<2
\ee
and hence $P_M$ is locally 2-Lipschitz.
}



%

Let $M$ be a submanifold from Lemma \ref{l:submanifold}.
Recall that 
$\Sigma=\bigcup_i\Sigma_i=\bigcup_i F(D_i^{1/10})$
{and $\Sigma$} is contained in the domain of~$P_M$.
For each $i$, define a map
$\psi_i\co D_i^{1/10}\to M$ by
\be
\label{e:psi-F-mod1}
 \psi_i = P_M\circ F|_{D_i^{1/10}}
\ee
and let $V_i$ be the image of $\psi_i$, that is
\be\label{Vi set}
 V_i = P_M(F(D_i^{1/10})) = P_M(\Sigma_i) .
\ee
Observe {that
\be
\label{e:GH Sigma M}
d_H(\Sigma,M)<{\ctext 5}\Cthirtythree r_0^2<\tfrac1{10}
\ee
and}
\be
\label{e:psi-F}
 |\psi_i(x)-F(x)| \le d_H(\Sigma,M)<{\ctext 5}\Cthirtythree r_0^2<\tfrac1{10}
\ee
for every $x\in D_i^{1/10}$.
This follows from Lemma \ref{l:submanifold}(1)
and the fact that $\psi_i(x)$ is the nearest point in $M$
to $F(x)$.

The next lemma shows that the maps $\psi_i$ provide a nice
family of coordinate charts for $M$.

\begin{lemma}
\label{l:charts}
If {$r_0<\Cthirtyeight$, where $\Cthirtyeight$  is sufficiently small, and $\de<r_0^2$}, then

1. $\psi_i$ is uniformly bi-Lipschitz, that is,
\be
\label{e:psi-F-mod2}
  \Cthirtynine^{-1} |x-y | \le |\psi_i(x)-\psi_i(y)| \le   \Cthirtynine|x-y|
\ee
for all $x,y\in D_i^{1/10}$.
In particular, $V_i$ is an open subset of $M$ and
$\psi_i$ is a diffeomorphism between $D_i^{1/10}$ and $V_i$.

2. $\bigcup_i \psi_i(D_i^{1/{30}})=M$.

3. If $i,j\in\N$ are such that $V_i\cap V_j\ne\emptyset$,
then $q_i$ and $q_j$ are neighbors.
\end{lemma}

\begin{proof}
1. {The Lipschitz continuity of $\psi_i$ follows from the bounds on the
first derivatives of $F$ and $P_M$, see Lemma \ref{l:Fsmooth}
and \eqref{e:dPMbound}.
More precisely, the second inequality in \eqref{e:psi-F-mod2} holds for any $\Cthirtynine\geq 2\Ctwentyfive(1)$.
}

It remains to prove that {with a suitable $\Cthirtynine>0$, we have}
\be
\label{e:charts1}
|\psi_i(x)-\psi_i(y)| \ge \Cthirtynine^{-1}|x-y|
\ee
for all $x,y\in D_i^{1/10}$.
{
For every $x\in D_i^{1/10}$ and $v\in\R^n$ we have
\be\label{e:charts2}
 |d_x\psi_i(v)| = | d_{F(x)} P_M(d_xF(v)) | \ge \tfrac12 |d_xF(v)| \ge (2\Cthirtyseven)^{-1}|v| .
\ee
The first inequality in \eqref{e:charts1} follows from {the first inequality in} \eqref{e:dPMbound}, Lemma \ref{l:submanifold}(5),
and the fact that $d_xF(v)$ belongs to $T_{F(x)}\Sigma_i$.
The second inequality in \eqref{e:charts1} follows from \eqref{e:dFbilip}.
}



{
By Lemma \ref{l:submanifold}(4) and \eqref{e:Fsmooth}, the first and second derivatives
of $P_M$ and $F$ are bounded by constants independent of~$r_0$.
These bounds imply that
$
 \|d^2_x\psi_i\| \le \Cforty
$
for all $x\in D_i^{1/10}$ and a suitable constant $\Cforty>0$.
Hence
\beq\label{e C40}
 | \psi_i(x)-\psi_i(y) - d_x\psi_i(x-y) | \le \tfrac12 \Cforty |x-y|^2
\eeq
for all  $x,y\in D_i^{1/10}$.
This and \eqref{e:charts2} imply that
$$
 | \psi_i(x)-\psi_i(y) | \ge \tfrac12 | d_x\psi_i(x-y) | \ge (4\Cthirtyseven)^{-1}|x-y|
$$
for all $x,y\in D_i^{1/10}$ such that 
\beq\label{e: C41}
|x-y|\le(2\Cthirtyseven\Cforty)^{-1}=:\Cfortyone.
\eeq

}


To handle the case when $|x-y|>\Cfortyone$, observe that
$$
 |\psi_i(x)-\psi_i(y)| > |F(x)-F(y)|-2\Cthirtythree r_0^2
$$
by \eqref{e:psi-F}.
Since $F|_{D_i}$ in uniformly bi-Lipschitz (by Lemma \ref{l:Fsmooth}),
it follows that
\be
\label{e:charts3}
 |\psi_i(x)-\psi_i(y)| \ge \Ctwentysix^{-1}|x-y| - 2\Cthirtythree r_0^2
\ee
for all $x,y\in D_i^{1/10}$.
If $|x-y|>\Cfortyone$ and $r_0$ is so small that $2\Cthirtythree r_0^2<\tfrac12 \Ctwentysix^{-1}\Cfortyone$
then the right-hand side of \eqref{e:charts3}
is bounded below by $\frac12\Ctwentysix^{-1}|x-y|$. Thus \eqref{e:charts1}
holds {with a suitable constant $ \Cthirtynine>0$} for all $x,y\in D_i^{1/10}$
and the first claim of the lemma follows.

2. Let $x\in M$.
By Lemma \ref{l:submanifold}(1)
there exists $z\in\Sigma$ such that $|x-z|< \Cthirtythree r_0^2$.
By Lemma \ref{l:Sigma0}
there exists $i\in\N$ and $y\in\Sigma_i^0$
such that $|y-z|<  {{\Ctwentyeight}}\de$. Then
$$
|x-y|<\Cthirtythree r_0^2+\Ctwentyeight \de< (\Cthirtythree +\Ctwentyeight)r_0^2<r_0/3
$$
where in the last inequality we assume that {$r_0<\Cthirtyeight$  and $\Cthirtyeight<\frac 13(\Cthirtythree +\Ctwentyeight)^{-1}$.} We are going to show that $x\in F(D_i^{1/{30}})$.

Since $x\in M$ and $|x-y|<r_0/3$, the straight line segment $[x,y]$ is contained
in the domain of $P_M$. Let $\gamma$ be the image of this segment under $P_M$.
Then $\gamma$ is a smooth curve in $M$ connecting $x$ to the point
$P_M(y)\in P_M(\Sigma_i^0)=\psi_i(D_i^{1/50})$.
Since $P_M$ is locally 2-Lipschitz, we have
$\length(\gamma) \le 2 |x-y| < 2(\Cthirtythree +\Ctwentyeight)r_0^2$.
%
We parametrize $\gamma$ by $[0,1]$ in such a way that $\gamma(0)=P_M(y)$
and $\gamma(1)=x$.
Suppose that $x\notin\psi_i(D_i^{1/{30}})$ and let
$$
 t_0 = \min\{t\in[0,1]: \gamma(t)\notin\psi_i(D_i^{1/{30}}) \} .
$$
This minimum exists since $\psi_i(D_i^{1/{30}})$ is an open subset of $M$.
Define
$
\tilde\gamma(t)=\psi_i^{-1}(\gamma(t))
$
for all $t\in[0,t_0)$.
Note that $t_0>0$ and $\tilde\gamma(0)\in D_i^{1/50}$
because $P_M(y)\in \psi_i(D_i^{1/50})$.
Since $\psi_i$ is a diffeomorphism onto its image,
$\tilde\gamma$ is a smooth curve in $D_i$.
Moreover, since $\psi_i$ is uniformly bi-Lipschitz, we have
$$
 \length(\tilde\gamma) \le C\length(\gamma)<   \Cthirtynine r_0^2 .
$$
Hence the limit point
$ p = \lim_{t\to t_0} \tilde\gamma(t) $
exists and satisfies
$$
 |p-\tilde\gamma(0)| \le \length(\tilde\gamma) <    \Cthirtynine r_0^2 .
$$
We may assume that $r_0$ is so small that the right-hand side
of this inequality is smaller than $\frac1{{30}}-\frac1{50}$.
Since $\tilde\gamma(0)\in D_i^{1/50}$, it follows that $z\in D_i^{1/50}$.
Hence $\gamma(t_0)=\psi_i(p)\in \psi_i(D_i^{1/{30}})$,
contrary to the choice of $t_0$.
This contradiction shows that $x\in\psi_i(D_i^{1/{30}})$.
Since $x$ is an arbitrary point of $M$,
the second claim of the lemma follows.

3. Assume that $V_i\cap V_j\ne\emptyset$.
Then there exist $x\in D_i^{1/10}$ and $y\in D_j^{1/10}$
such that $\psi_i(x)=\psi_j(y)$.
This equality and \eqref{e:psi-F} imply that
$|F(x)-F(y)| < \tfrac15$,
hence
\be\label{e:charts4}
|F_i(x)-F_i(y)|<\tfrac15 
\ee
(recall that $F_i\co\Om\to\R^{n+1}$ is the $i$th coordinate projection of $F$).
Since $x\in D_i^{1/10}$, the point $F_i(x)\in\R^{n+1}$
belongs to the spherical cap $S_1$
and therefore $|F_i(x)|>1$.
This and \eqref{e:charts4}
imply that $F_i(y)\ne 0$ and hence $q_i$ and $q_j$ are neighbors
by Lemma \ref{l:nonzeroF}.
\end{proof}

Note that Lemma \ref{l:charts}(3) and Lemma \ref{l:adjgraph}(2)
imply that the sets $V_i$ cover $M$ with {bounded multiplicity $N_1(n)$,
that is, for every $x\in M$ the number of indices $i$
such that $x\in V_i$ is bounded by a $N_1(n)$} depending
only on $n$.

Now we can fix the value of $r_0$ such that
Lemma \ref{l:submanifold} and Lemma \ref{l:charts} work.
Since $r_0$ is yet another constant depending only on $n$,
we omit the dependence on $r_0$ in subsequent estimates
and just use the generic notation $C$.
In particular, the fourth assertion of
Lemma \ref{l:submanifold} now implies that
\be
\label{e:PMbounded}
 \|dP_M\|_{C^{k}({\mathcal U}_{r_0/3}(M))} \le {\Cfortyfour(k)}
\ee
where {$\Cfortyfour(k)=\Cten(k) r_0^{1-k}$
for all $k\ge 1$ and $\Cfortyfour(0)=2$}. {By applying Lemma \ref{l:iterated chainrule}(1) in Appendix A, 
 \eqref{e:Fsmooth}, and \eqref{e:PMbounded}, we obtain} 
\be
\label{e:psibounded}
 \|d\psi_i\|_{C^m(D_i^{1/10})} <
 \Cfortyfive(m):=
 {
  2^{m(m+1)/2+m}\Cfortyfour(m) \Ctwentyfive(m+1)^{m+1}.}
\ee
for all $m\ge 0$.
%
%
%

%
%
%
%
%
%

\begin{lemma}
\label{l:sameprojection}
{There is $\Cfortysix>0$  such that the following is valid.}
If $x\in D_i^{1/10}$, $y\in D_j^{1/10}$ and
$\psi_i(x)=\psi_j(y)$, then
\be
\label{e:sameprojection}
|F(x)-F(y)|< \Cfortysix\de .
\ee
\end{lemma}

\begin{proof}
Applying  Lemma 
\ref{l:Sigma0} to the point $F(x)\in\Sigma_i$
yields that there exists $k\in\N$ and a point $z\in D_k^{1/50}$
such that $|F(x)-F(z)|<   {{\Ctwentyeight}}\de$. {Since $P_M$ is uniformly Lipschitz,
see (\ref{e:PMbounded}) with $\Cfortyfour(0)=2$,  {and  $\Ctwentyeight>\Ctwentyseven(0)$,}}
it follows that
\be
\label{e:charts6}
 |\psi_i(x)-\psi_k(z)| <  {2}  {{\Ctwentyeight}}\de
\ee
and (since $\psi_i(x)=\psi_j(y)$)
\be
\label{e:charts7}
 |\psi_j(y)-\psi_k(z)| < {2}  {{\Ctwentyeight}}\de .
\ee
This and \eqref{e:psi-F} imply that
$
 |F(y)-F(z)| < \tfrac15+{4} 
   {{\Ctwentyeight}}\de < \tfrac12
$,
hence $F(y)\in {\mathcal U}_{1/2}(\Sigma_k^0)$.
By Lemma \ref{l:Sigmai} it follows that $F(y)\in {\mathcal U}_{\Ctwentyeight\de}(\Sigma_k)$.
This means that there exists $z'\in D_k^{1/10}$ such that
\be
\label{e:charts8}
|F(z')-F(y)|<{\Ctwentyeight}\de.
\ee
Then
$$
|\psi_k(z')-\psi_j(y)| = |P_M(F(z'))-P_M(F(y))| <   {{\Cten(1)}\,}{{\Ctwentyeight}} \de 
$$
{since} $P_M$ is uniformly Lipschitz.
This and \eqref{e:charts7} imply that
$|\psi_k(z)-\psi_k(z')|<( {{2}} +{\Cten(1)})\Ctwentyeight \de$.
Since $\psi_k$ is uniformly bi-Lipschitz
by the first claim of the Lemma \ref{l:charts},
it follows that
\beq\label{e C47}
 |z-z'|\le \Cthirtynine^{-1} |\psi_i(z)-\psi_i(z')| < \Cfortyseven \de,\quad \Cfortyseven=\Cthirtynine^{-1}( {{2}} +{\Cten(1)})\Ctwentyeight,\hspace{-15mm}
\eeq
and hence $|F(z)-F(z')|<\Ctwentysix   \Cfortysix\de$ by Lipschitz continuity of $F$, see Lemma \ref{l:Fsmooth}.
This and \eqref{e:charts8} imply that $|F(y)-F(z)|<\tfrac 12\Cfortysix,$ where 
$\Cfortysix=
(\Ctwentysix   \Cfortyseven+\Ctwentyeight)\de$.

Thus we have shown that \eqref{e:charts7} implies that $|F(y)-F(z)|<\tfrac 12\Cfortysix\de$.
Similarly \eqref{e:charts6} implies that $|F(x)-F(z)|<\tfrac 12\Cfortysix \de$
and \eqref{e:sameprojection} follows.
\end{proof}

We are going to restrict our coordinate maps $\psi_i$
to smaller balls $D_i^{1/15}$.
Let $V_i'=\psi_i(D_i^{1/15})$ and $U_{ij}=\psi_i^{-1}(V_i'\cap V_j')$.
The set $U_{ij}\subset D_i^{1/15}$ is the natural domain
of the transition map $\psi_j^{-1}\circ\psi_i$ between
the restricted coordinate charts.

\begin{lemma}
\label{l:transdelta} {There is $\Cfortyeight=\Cfortyeight(m)>0$  such that the following is valid.}
Let $i,j\in\N$ be such that $V'_i\cap V'_j\ne\emptyset$.
Then
\be
\label{e:transdelta}
 \|\psi_j^{-1}\circ\psi_i-A_{ij}\|_{C^m(U_{ij})} < \Cfortyeight(m)\de
\ee
for all $m\ge 0$.
\end{lemma}

\begin{proof}
Note that $q_i$ and $q_j$ are neighbors by Lemma \ref{l:charts}(3).
By Lemma \ref{l:sheets} it follows that $A_{ij}(D_i^{1/10})\subset D_j$.
Consider the map $G\co D_i^{1/10}\to E$ defined by $G=F\circ A_{ij}|_{D_i^{1/10}}$.
By Lemma \ref{l:sheets} we have
\be
\label{e:transdelta1}
 \|G-F\|_{C^m(D_i^{1/10})} < \Ctwentyseven(m)\de .
\ee
This and Lemma \ref{l:submanifold}(1) imply
that the image of $G$ is contained in the domain of $P_M$,
so we can consider a map $\tilde\psi_i\co D_i^{1/10}\to M$
defined by $\tilde\psi_i=P_M\circ G$.

{Next we apply Lemma \ref{l:iterated chainrule}(2) in Appendix A with
$f=P_M$, $g=F$, and $h=G$, where $F$ and $G$ are defined in arbitrary ball $B^n(x',\rho)\subset \R^n$, centered at $x'\in D_i^{1/10}$,  and having radius $\rho<r_0/(10\Ctwentyfive({m}))$. 
Then by Lemma \ref{l:Fsmooth}, $F(B^n({p_i},\rho)) \subset B(F(x'),r_0/10)\subset E$.
Assuming that $\de<1/(10{\Ctwentyeight})$, the inequality (\ref{e:transdelta1}) implies that
$G(B^n(x_0,\rho)) \subset  Y=B(F(x'),r_0/3)\subset E$.
As $F(x')\in M$, these imply that the image of both $F$ and $G$ are in $Y\subset {\mathcal U}_{r_0/3}(M)$ where $P_M$ is defined.
Using  Lemma \ref{l:iterated chainrule}(2) in Appendix A   in these small balls with 
\eqref{e:PMbounded} and \eqref{e:transdelta1}} and combining these local estimates,
we obtain
\begin{eqnarray}
\label{e:tilde psibounded}
 & &\|\tilde\psi_i-\psi_i\|_{C^m(D_i^{1/10})} <
 \Cfortynine(m)\de ,\quad\hbox{where}\\ \nonumber& & \Cfortynine(m):= 
 {(m+1)2^{m(m-1)}  \Cfortyfour(m+1)
\,\cdotp(1 +  \Ctwentyfive(m))^m\Ctwentyeight(m).}
\end{eqnarray}
%

%
%
%
%
%
%

{Assume that $\de< \min(1,\Cthirtynine/(2\Cfortynine(1)))$.
 {Then 
 (\ref {e:psibounded}) and (\ref{e:tilde psibounded}) imply that
 $\tilde\psi_i^{-1}$ has locally the Lipschitz constant $2\Cthirtynine^{-1}$.}
 Then (\ref{e:tilde psibounded})} and Lemma \ref{l:charts}(1) imply
that
$\tilde\psi_i$ is a diffeomorphism onto its image, and 
the image of $\tilde\psi_i$ contains~$V_i'$. {
Using {(\ref{e:psibounded}) and} (\ref{e:tilde psibounded}), we see that 
\begin{eqnarray}\label{e: inverse tilde psi}
 & & \|d\tilde\psi_i\|_{C^{m-1}(D_i^{1/10})} \leq 
 \Cfortyfive(m-1)+\Cfortynine(m),\\ \nonumber
 & & \|d\tilde\psi_i^{-1}\|_{C^{m-1}(V_i')} \leq 
(3m)^m(1+ \Cfortyfive(m-1)+\Cfortynine(m))^{2m}(2\Cthirtynine^{-1})^m=:\Cfifty(m).\hspace{-3mm}
\end{eqnarray}
Moreover, by Lemma \ref{l:iterated chainrule}(1) in Appendix A}, \eqref{e:tilde psibounded} and \eqref{e: inverse tilde psi} imply that the composition
$\tilde\psi_i^{-1}\circ\psi_i$ is $C\de$-close to the identity,
more precisely,
\be
\label{e:transdelta2}
 \|\tilde\psi_i^{-1}\circ\psi_i-\text{id}\|_{C^m(D_i^{1/15})} <  \Cfortyeight(m)\de , 
\ee
{where 
$
 \Cfortyeight(m)=m^m \Cfifty(m)
 (1+ \Cfortynine(m))^m+{\Cfortynine(0)}.
$

 Let} us show that $A_{ij}(U_{ij})\subset D_j^{1/10}$.
Let $x\in U_{ij}$ and $z=A_{ij}(x)$.
Then $|F(x)-F(z)|<{\Ctwentyeight}\de$ by Lemma \ref{l:sheets}.
Let $y\in U_{ji}$ be such that $\psi_j(y)=\psi_i(x)$.
Then $|F(x)-F(y)|<\Cfortysix\de$ by Lemma \ref{l:sameprojection}.
Therefore $|F(y)-F(z)|<({\Ctwentyeight}+\Cfortysix)\de$.
Since $F|_{D_j}$ is uniformly bi-Lipschitz by Lemma \ref{l:Fsmooth}(2),
it follows that
$$
|y-z|<  \Ctwentysix |F(y)-F(z)|<  \Ctwentysix ({\Ctwentyeight}+\Cfortysix) \de<\tfrac1{10}-\tfrac1{15},
$$
if $\de$ is sufficiently small.
Since $y\in U_{ji}\subset D_j^{1/15}$,
this implies that $z\in D_j^{1/10}$.

Thus we have shown that $A_{ij}(U_{ij})\subset D_j^{1/10}$.
This implies that
$$
 \tilde\psi_i|_{U_{ij}} = P_M\circ F\circ A_{ij}|_{U_{ij}} = \psi_j\circ A_{ij}|_{U_{ij}}
$$
and therefore
$$
 \tilde\psi_i^{-1}|_{V_i'\cap V_j'} = A_{ij}^{-1}\circ \psi_j^{-1}|_{V_i'\cap V_j'}.
$$
Then  \eqref{e:transdelta2} implies that
$$
\|A_{ij}^{-1}\circ \psi_j^{-1} \circ\psi_i-\text{id}\|_{C^m(U_{ij})} <  \Cfortyeight(m)\de .
$$
and \eqref{e:transdelta} follows {as $A_{ij}$ is an affine isometry}.
\end{proof}


\subsection{Riemannian metric and quasi-isometry}
\label{subsec:riemannian metric}

Now we are going to equip $M$ with a Riemannian metric $g$
such that the resulting Riemannian manifold $(M,g)$ satisfies
the assertions of Proposition \ref{p:manifold}.
(The metric induced from $E$ is not suitable for this purpose.
One of the reasons is that its curvature is bounded by $C$ but not by $C\de$.
{\ctext Another reason is that the map $\phi$ is arbitrary, so
distances may be distorted}.)

First we observe that there exists a smooth partition
of unity $\{u_i\}$ on $M$ subordinate to the covering $\{V_i'\}$ {and $\Cfiftyone(m)>0$}
 such that 
\be
\label{e:ubounded}
 \|u_j\circ \psi_i\|_{C^m(D_i^{1/15})} < \Cfiftyone(m)
\ee
for all $i,j\in\N$ and all $m\ge 0$.
To construct such a partition of unity,
fix a smooth function {$h\co\R^n\to\R_+$ which equals 1 within
the ball $B_{1/{30}}(0)$ and 0 outside the ball $B_{1/15}(0)$, given by
$h(t)=\alpha_{1/30,1/15}(t)$, see (\ref{mu function})}.
Then define $\tilde u_i\co M\to\R_+$ by
\be\label{tilde ui}
 \tilde u_i(x) =
 \begin{cases}
 h(\psi_i^{-1}(x)-p_i), &\quad\text{if $x\in V_i'$} \\
  \quad   \quad   \quad 0, &\quad\text{otherwise}.
 \end{cases}
\ee
Finally, let $u(x)=\sum_i\tilde u_i(x)$ and $u_i(x)=\tilde u_i(x)/u(x)$.
Lemma \ref{l:transdelta} implies that {there is $\Cfiftyone(m)>0$ such that} 
$$
 \|\tilde u_j\circ\psi_i\|_{C^m(D_i^{1/15})} < \Cfiftyone(m)
$$
for all $i,j\in\N$ and all $m\ge 0$.
{
As in Lemma \ref{l:charts}(3) and Lemma \ref{l:adjgraph}(2),
we see that there is $N_2(n)$ depending
only on $n$ such that  for every $x\in M$ the number of indices $i$
such that $x\in V_i'$ is bounded by a $N_2(n)$. Hence, the sets $V_i'$ cover $M$ with bounded multiplicity $N_2(n)$ and}
it follows from Lemma \ref{l:charts}(2) that
a similar estimate holds for $u\circ\psi_i$ and
\eqref{e:ubounded} follows.

For every $i\in\N$, define a Riemannian metric $g_i$ on $V_i$
by 
\be
\label{e: gi metric}
g_i=(\psi_i^{-1})^*g_E
\ee 
where $g_E$ is the
standard Euclidean metric in $D_i^{1/10}\subset \R^n$
and the star denotes the pull-back of the metric by a map.
In the other words, $g_i$ is the unique Riemannian metric on $V_i$
such that $\psi_i$ is an isometry between $D_i^{1/10}$
and $(V_i,g_i)$.
Then Lemma \ref{l:transdelta} implies that
\be
\label{e:gi-estimate}
 \|\psi_j^*g_i-g_E\|_{C^m(U_{ij})} <{2^m n^4} \Cfortyeight(m)^2\de
\ee
for all $m\ge 0$ and $i,j\in\N$ such that $V_i'\cap V_j'\ne\emptyset$.
Define a metric $g$ on $M$ by
\be\label{e: g metric} 
g=\sum_i u_i g_i.
\ee
The pull-back $\psi_j^*g$ of this metric by a coordinate map $\psi_j$
has the form
\be\label{e:g in coordinates}
 \psi_j^*g = \sum\nolimits_i(u_i\circ\psi_j)\cdot\psi_j^*g_i .
\ee
By \eqref{e:ubounded} and \eqref{e:gi-estimate} it follows that
\be
\label{e:g-estimate}
 \|\psi_j^*g-g_E\|_{C^m(D_j^{1/15})} < \Cfiftytwo({n},m)\de,
 \quad 
  \Cfiftytwo({n},m):={4^mn^4  \Cfortyeight(m)^2\Cfiftyone(m)N_2(n)}.
\ee
Let $\Cfiftythree=\Cfiftytwo({n},0).$
So in the local coordinates defined by $\psi_j$ on $V_j'$
the metric tensor is $ \Cfiftythree\de$-close to the
Euclidean one and its derivatives up to the second order
are bounded by $  \Cfiftytwo({n},2)\de$.
So are the sectional curvatures of the metric.
Thus $(M,g)$ satisfies the second assertion of Proposition \ref{p:manifold} {with a suitable 
constant $\Ctwo$.}

Let $d_g\co M\times M\to\R_+$ be the distance induced by~$g$.
The estimate \eqref{e:g-estimate} implies that the coordinate
maps $\psi_i$ are almost isometries between the Euclidean metric on $D_i^{1/15}$
and the metric $g$ on~$V_i'$. More precisely, $\psi_i$ distorts the lengths
of tangent vectors by a factor of at most $1+\Cfiftythree\de$.
Therefore
\be
\label{e:dg-estimate}
(1+\Cfiftythree\de)^{-1}<\frac{d_g(\psi_i(x),\psi_i(y))}{|x-y|} < 1+\Cfiftythree\de,
\ee
for all $x,y\in D_i^{1/30}$.
(The ball $D_i^{1/30}$ here is twice smaller than the domain where
$\psi_i$ is almost isometric. This adjustment is needed because
the $d_g$-distance between points in $V_i'$ can be realized by paths
that leave $V_i'$.)

{Below we will assume that  $\de<\Cfiftythree^{-1}$ so that $1+ \Cfiftythree\de<2$ in
(\ref{e:dg-estimate}).}
{
This and bi-Lipschitz continuity of charts $\psi_i$ (see Lemma \ref{l:charts}(1)) imply
that $d_g$ is bi-Lipschitz equivalent to the intrinsic metric $d_M$ induced on $M$ from~$E$. Namely
\be\label{e:dM-vs-dg}
 (2\Cthirtynine)^{-1} \le \frac{d_g(x,y)}{d_M(x,y)} \le 2\Cthirtynine
\ee
for all $x,y\in M$.
}

Now we construct a $(1+C\de,C\de)$-quasi-isometry $\Psi\co X\to M$.
Recall that
$X_0=\{q_i\}_{i=1}^\infty$ is a $\frac1{100}$-net in our original metric space $X$
and for each $i\in\N$ we have a $2\de$-isometry
$f_i\co B_1(q_i)\to D_i$ such that $f_i(q_i)=p_i$.
We construct $\Psi\co X\to M$ as follows.
For every $x\in X$, pick a point $q_j\in X_0$ such that $d_X(x,q_j)\le\frac1{100}$
and define 
\be\label{Psi map}
\Psi(x)=\psi_j(f_j(x)).
\ee
The next lemma shows that the choice of $q_j$
does not make much difference.

\begin{lemma}
\label{l:another-qi}{There is $\Cfiftyfour>0$  such that the following holds.}
Let $x\in X$ and $q_i\in X_0$ be such that $d_X(x,q_i)<\frac1{20}$.
Then $f_i(x)\in D_i^{1/15}$ and
\be
\label{e:another-qi}
 d_g(\Psi(x),\psi_i(f_i(x))) <\Cfiftyfour\de .
\ee
\end{lemma}

\begin{proof}
Let $q_j$ be the point of $X_0$ chosen for $x$
in the construction of~$\Psi$.
Then $d_X(x,q_j)\le\frac1{100}$ and $\Psi(x)=\psi_j(f_j(x))$.
By the triangle inequality,
$$
d_X(q_i,q_j)<\tfrac1{20}+\tfrac1{100}<\tfrac12 ,
$$
hence $q_i$ and $q_j$ are neighbors.
Observe that
$
 |f_i(x)-p_i|<\tfrac1{20}+{2}\de
$
since $p_i=f_i(q_i)$ and $f_i$ is a $2\de$-isometry.
Similarly,
$
 |f_j(x)-p_j|<\tfrac1{100}+{2}\de
$.
Hence $f_i(x)\in D_i^{1/15}$ and $f_j(x)\in D_j^{1/50}$.
By \eqref{e:transition}, the point $f_j(x)$ is $\Ctwentytwo\de$-close
to $A_{ij}(f_i(x))$.
{
Hence, by Lemma \ref{l:Fsmooth},
$$
 | F(f_j(x)) - F( A_{ij}(f_i(x)) ) | 
 < \Ctwentysix |f_j(x)-A_{ij}(f_i(x))|
 < \Ctwentysix\Ctwentytwo\de .
$$
By Lemma \ref{l:sheets} we have $|F(f_i(x))-F(A_{ij}(f_i(x)))|<{\Ctwentyeight}\de$.
Therefore
\be\label{e:another-qi2}
 | F(f_j(x)) - F(f_i(x)) | < (  \Ctwentysix\Ctwentytwo + {\Ctwentyeight}) \de =: \Cfiftyfive\de .
\ee
Denote $a=F(f_j(x))$ and $b=F(f_i(x))$, then $\Psi(x)=P_M(a)$ and $\psi_i(f_i(x))=P_M(b)$.
Assuming that $\Cfiftyfive\de<r_0/10$, \eqref{e:another-qi2} implies that
$|a-b|<\Cfiftyfive\de<r_0/10$.
Since $a\in\Sigma\subset\mathcal U_{r_0/10}(M)$ and $P_M$ is defined in $\mathcal U_{r_0/3}(M)$,
it follows that the line segment $[a,b]$ is contained in the domain of~$P_M$.
This and \eqref{e:dPMbound} giving  the  upper bound 2
for the local Lipschitz constant of $P_M$ imply that
$$
 d_M(P_M(a),P_M(b)) \le \text{length}(P_M([a,b])) \le 2\Cfiftyfive\de
$$
Hence, by \eqref{e:dM-vs-dg},
$$
 d_g(P_M(a),P_M(b)) \le 4\Cthirtynine\Cfiftyfive\de =: \Cfiftyfour\de .
$$
Since $P_M(a)=\Psi(x)$ and $P_M(b)=\psi_i(f_i(x))$,
\eqref{e:another-qi} follows.
}
\end{proof}

Now let us show that $\Psi(X)$ is a $\Cfiftysix\de$-net in $(M,d_g)$,
{where $\Cfiftysix=\Cfiftyfour+4$.}
{
Pick $z\in M$. By Lemma \ref{l:charts}(2), $z\in\psi_i(D_i^{1/30})$ for some~$i$.
Let $y\in D_i^{1/30}$ be such that $\psi_i(y)=z$.
Since $f_i$ is a $2\de$-isometry, there is $x\in B_1(q_i)$ such that $|y-f_i(x)|<2\de$.
By the bi-Lipschitz estimate \eqref{e:dg-estimate},
\be\label{e:Cdenet1}
 d_g(z,\psi_i(f_i(x))) \le 2 |y-f_i(x)| < 4\de .
\ee
Since $y\in D_i^{1/30}$, $|y-f_i(x)|<2\de$, and $f_i$ is a $2\de$-isometry
with $f_i(q_i)=p_i$, we have
$
 d(x,q_i) < \tfrac1{30} + 4\de < \tfrac1{20} .
$
Hence $d_g(\Psi(x),\psi_i(f_i(x))) <\Cfiftyfour\de$
by Lemma \ref{l:another-qi}.
This and \eqref{e:Cdenet1} imply that 
\beq\label{e C56}
d_g(z,\Psi(x))<(\Cfiftyfour+4)\de=\Cfiftysix\de.
\eeq
Since $z\in M$ is arbitrary, it follows that
$\Psi(X)$ is a $\Cfiftysix\de$-net in $(M,d_g)$.
}



\begin{lemma}
\label{l:smaller-scale}{There is $\Cfiftyseven\geq \Cfiftysix$ such that the following holds.}
For all $x,y\in X$ such that
$d_X(x,y)<\frac1{100}$ or $d_g(\Psi(x),\Psi(y))<\frac1{100}$,
one has
\be
\label{e:distancedelta}
 |d_g(\Psi(x),\Psi(y))-d_X(x,y)| < \Cfiftyseven\de .
\ee
\end{lemma}

\begin{proof} 
Let  $x\in X$ and $q_i$ be the point of $X_0$ chosen for $x$
in the construction of~$\Psi$, so that
$d_X(x,q_i)\le\frac1{100}$. Then $\Psi(x)=\psi_i(f_i(x))$.
Note that $|f_i(x)-p_i|<\frac1{100}+{2}\de<\frac{1}{30}$ 
since $p_i=f_i(q_i)$ and $f_i$ is a $2\de$-isometry.
 
First, we consider the case  when $y\in X$ is such that $d_X(y,q_i)<\frac{3}{100}$.
Since $f_i$ is a $2\de$-isometry,
$|f_i(y)-p_i|<\frac{3}{100}+{2}\de<\frac 1{30}$ and
the distance $|f_i(x)-f_i(y)|$ differs from $d_X(x,y)$
by at most $2\de$.
The above and \eqref{e:dg-estimate} imply that
$$
 |d_g(\psi_i(f_i(x)),\psi_i(f_i(y)))-d_X(x,y)| <{\Cfiftythree\de |f_i(x)-f_i(y)| +2\de\leq 
 (2+\Cfiftythree)\de} .
$$
This and Lemma \ref{l:another-qi}
 prove \eqref{e:distancedelta} when   $d_X(y,q_i)<\frac{3}{100}$.

 In particular, this proves the claim of the lemma
 in the case when  $d_X(x,y)<\frac{1}{100}$
 as then by the triangle inequality we have
$d_X(y,q_i)<\frac1{100}+\frac1{100}<\frac 3{100}$ .

Second, we consider the case when $y\in X$ is such that   $d_g(\Psi(x),\Psi(y))<\frac1{100}$.
For every $r>0$, denote by $B_i(r)$ the ball of radius $r$ in $M$
with respect to $d_g$ centered at $\psi_i(p_i)$.
Since $\psi_i$ almost preserves the metric tensor {in the sense of  (\ref{e:g-estimate}),
by denoting $\Cfiftyfour=\Cfiftytwo(0)$}
we have
\be\label{medium balls}
 B_i(\tfrac1{15}-\Cfiftyfour\de)\subset V_i'=\psi_i(D_i^{1/15}) \subset B_i(\tfrac1{15}+\Cfiftyfour\de) .
\ee
Since $|f_i(x)-p_i|<\frac1{100}+{2}\de$,
it follows {\ctext from (\ref{e:dg-estimate}) that the point $\Psi(x)=\psi_i(f_i(x))$ belongs
to $B_i((1+\Cfiftythree\de)(\frac1{100}+{2}\de))\subset 
B_i(\frac1{100}+3(1+\Cfiftythree)\de)$ and hence
$$\Psi(y)\in B_i(\frac1{100}+\frac1{100}+3(1+\Cfiftythree)\de)=B_i(\frac {1}{50}
+3(1+\Cfiftythree)\de)\subset V_i'.$$}
%
%
%
%

Let $q_j$ be the point of $X_0$ chosen for $y$ when defining $\Psi$
that satisfies $d_X(y,q_j)\le\frac1{100}$.
Since $\Psi (y)\in V_i'$, the point
$z:=\psi_i^{-1}(\Psi(y))=\psi_i^{-1}\circ\psi_j(f_j(y))$ is well-defined.
Moreover, $z$ lies within distance $\frac {1}{50}+{2}\de$ from $p_i$
since $\Psi(y)\in B_i(\frac {1}{50}+{2}\de)$.
By Lemma \ref{l:transdelta}, $z$ is $ \Cfortyeight(0)\de$-close to $A_{ij}(f_j(y))$
and the latter is $\Ctwentytwo\de$-close to $f_i(y)$ by \eqref{e:transition}.
Hence $|f_i(y)-p_i|<\frac {1}{50} +{( \Cfortyeight(0)+\Ctwentytwo+2)}\de$.
Since $f_i$ is a $2\de$-isometry, it follows that
$d_X(y,q_i)<\frac {1}{50}+{( \Cfortyeight(0)+\Ctwentytwo+4)}\de<\frac {3}{100}$.
Thus, \eqref{e:distancedelta}
follows from the first part of the proof.
\end{proof}

Lemma \ref{l:smaller-scale} and the fact that
$\Psi(X)$ is a $\Cfiftysix\de$-net in $(M,d_g)$ {imply}
that $\Psi$ satisfies the assumptions of
Lemma \ref{l:local-gh-implies-qi}
with $r=\frac1{100}$ and $ \Cfiftyseven\de$ in place of $\de$.
Thus
$\Psi$ is a $(1+{10^3\Cfiftyseven}\de,{3\Cfiftyseven}\de)$-quasi-isometry
from $X$ to $(M,d_g)$
and the first claim of Proposition \ref{p:manifold} follows.
The second claim is already proven above.
It remains to prove the third claim of Proposition \ref{p:manifold}.

Since $\Psi$ is a $(1+{10^3}{\Cfiftyseven}\de,{3}{\Cfiftyseven}\de)$-quasi-isometry,
{Lemma \ref{l:QI-for-balls} implies that}
every unit ball in $(M,d_g)$ is GH {($ 2015\Cfiftyseven\delta$)}-close to a unit ball in $X$,
and hence also ${(2015\Cfiftyseven+1)}\de$-close to a unit ball in $\R^n$.
{
Thus, for every $p\in M$, 
\beq\label{e C58}
d_{GH}(B^M_1(p),{B_1^n} )<(2015\Cfiftyseven+1)\de < 10^{-3},
\eeq
see \eqref{delta condition}.}
{Also, we have already shown that  $(M,g)$ satisfies the second assertion of Proposition \ref{p:manifold} 
so that its sectional curvature is bounded by $\Ctwo{\de}$.}
Therefore one can apply {Lemma \ref{l:injrad-flat} with $r=1$ and $K=\Ctwo \de<10^{-3}$ (see \eqref{delta condition}),
and conclude that $\inj_M\ge \frac9{10}>\frac12$}.
This finishes the proof of Proposition \ref{p:manifold}
and the proof of Theorem~\ref{t:manifold}.


%
%
%
%
%
%
%
%

\begin{remark}\label{rem: Thm 1 converse}
The quasi-isometry parameters in Theorem \ref{t:manifold} 
are optimal up to constant factors.
To see this, assume that a metric space $X$ is
$(1+\de r^{-1},\de)$-quasi-isometric to  an $n$-dimensional manifold $M$
with $|\Sec_M|\le \de r^{-3}$ and $|\inj_M|\ge 2r$.
Then by Lemma \ref{l:QI-for-balls} the $r$-balls in $X$ are GH $C\de$-close
to $r$-balls in $M$. 
Furthermore, by  \eqref{eq: GH distance of balls}  the 
$r$-balls in $M$ are GH $C\de$-close to $r$-balls in $\R^n$. 
Hence $X$ is  $C\delta$-close to $\R^n$ at scale $r$.

Thus the assumption of Theorem \ref{t:manifold}
that $X$ is $\de$-close to $\R^n$ at scale $r$
is necessary, up to multiplication of the parameters by a constant factor depending on~$n$.
The assumption that $X$ is $\de$-intrinsic could be weakened,
but it is not really restrictive due to Lemma \ref{l:minchain}.
\end{remark}

\begin{remark}\label{centers in X0}
{We note in the proof of 
Proposition 
\ref{p:manifold} the construction of the manifold $M$ uses only the $r$-balls $B_r(q_i)$ centered in a maximal
$\frac r{100}$-separated subset $X_0=\{q_i\}_{i=1}^N$ in $X$ and the fact that
the  Gromov-Hausdorff distance 
between any balls $B_r(q_i)$ and $B_r^n$ is less than $\delta$.
We will later use this observation in the Algorithm {\it ManifoldConstruction}.

We also note that 
the assumptions of Theorem \ref{t:manifold} can be relaxed: It is enough to assume
that $X$ is $\de$-intrinsic  and  there is a $(r/100)$-net $X_0\subset X$ such that for any $x\in X_0$ the ball $B_r(x)\subset X$  is
$\delta$-close to the Euclidean ball $B_r^n$. Indeed, when this is valid, we  
see 
that the ball of radius $\tfrac{99}{100}r$ centered 
at any point $x\in X$ is $3\delta$-close to the Euclidean 
ball of the same radius. Then the assumptions in the claim of  
  Theorem \ref{t:manifold}  are valid with parameters $r$  and $\delta$  replaced by
   $\tfrac{99}{100}r$  and $3\delta$, respectively.}

\end{remark}

\section{Proof of Corollaries \ref{cor:compact class}, \ref{cor:alexandrov} and \ref{cor 2:manifold}}
\label{sec:proof-corollaries}

\begin{proof}[Proof of Corollary \ref{cor:compact class}]
First we prove the first inclusion in \eqref{e:class inclusion}.
Let $X$ be a metric space from the class $\mathcal M_{\de/6}(n,K/2,2i_0,D-\de)$.
Then there exists a manifold $M\in\mathcal M(n,K/2,2i_0,D-\de)$
such that $d_{GH}(M,X)<\frac\de 6$.
Hence every ${r}$-ball in $X$ is GH $\frac\de2$-close
to an ${r}$-ball in $M$.
Since $r=(\de/K)^{1/3}$, by \eqref{eq: GH distance of balls} we have
$d_{GH}(B_r^M(x),B^n_r)<\frac12 Kr^3=\frac\de2$ for every $x\in M$.
Hence every $r$-ball in $X$ is GH $\de$-close to $B^n_r$.
Thus $X$ is $\de$-close to $\R^n$ at scale $r$.
Similarly $X$ is $\de_0$-close to $\R^n$ at scale ${i_0}$.
Since $d_{GH}(M,X)<\frac\de 6$, 
Lemma \ref{l:intrinsic metric}(1) implies that $X$ is $\de$-intrinsic.
We also have $\diam(X)\le\diam(M)+2d_{GH}(X,M)\le D$.
Thus $X\in\mathcal X$, proving the first inclusion in \eqref{e:class inclusion}.

Now we prove the second inclusion in \eqref{e:class inclusion}.
Let $X\in\mathcal X$. 
Recall that $\de=Kr^3$, $\de_0=Ki_0^3$, $\de<\de_0$, {and $Ki_0^2<{\sigma_1}$}.
Therefore $r<i_0$ and $\de r^{-1}<\de_0 i_0^{-1}<{\sigma_1}$.
If ${\sigma_1}$ is sufficiently small then by
Theorem \ref{t:manifold} there exists a manifold $M$
which is $(1+\Cone \de r^{-1},\Cone\de)$-quasi-isometric to $X$ and has
$|\Sec_M|\le \Ctwo\de r^{-3}=\Ctwo K$.
Let us show that $\inj_M>{2}i_0/3$.
{
Since  $X\in \mathcal X(n,\de_0,i_0,D)$ and $r<i_0$,
the above quasi-isometry and Lemma \ref{l:QI-for-balls} imply that.  
$$
 d_{GH}(B^M_{i_0}(x),B^n_{i_0}) \le 2C_1\de + 5C_1\de + \de_0 < (7C_1+1)\de_0
$$
for all $x\in M$.
By Proposition \ref{p:injrad}(1) applied to $\tilde M=\R^n$ and $\rho=i_0$ it follows that
$$
 \inj_M \ge i_0 - (7C_1+1)\cfour\de_0 > 2i_0/3
$$
provided that ${\sigma_1}$ is sufficiently small.
(Recall that $i_0<\sqrt{{\sigma_1}/K}$ and $\de_0<{\sigma_1}i_0$.)
}

By \eqref{e:t1-gh-estimate} we have $d_{GH}(X,M) \le {2\Cone} \de r^{-1} D$,
hence $\diam(M) \le D (1+{4\Cone}\de r^{-1})$.
{
We may assume that ${\sigma_1}$ is so small that ${4\Cone}\de r^{-1}<\tfrac 14$  and} 
let $M_1$ be the result of rescaling $M$ by the factor $(1+{4\Cone}\de r^{-1})^{-1}$.
Then $\diam(M_1)\le D$ and $d_{GH}(M,M_1)\le {2\Cone}\de r^{-1}D$.
Hence
\be\label{e:dGH(X,M1)}
 d_{GH}(X,M_1) \le d_{GH}(X,M) + d_{GH}(M,M_1) 
 \le  {4\Cone}  \de r^{-1}D ={4\Cone} DK^{1/3}\de^{2/3} .
\ee
{Note that the above} scale factor between $M$ and $M_1$
is greater than $\frac34$. Then $\inj_{M_1}\ge \frac34\inj_M \ge i_0/{2}$
and therefore $M_1\in\mathcal M(n,{\tfrac 43\Ctwo} K,i_0/{2},D)$. 
This and \eqref{e:dGH(X,M1)} imply the second inclusion in \eqref{e:class inclusion}
and Corollary \ref{cor:compact class} follows.
\end{proof}

{
\begin{proof}[Proof of Corollary \ref{cor:alexandrov}]
In this proof we assume that the reader is familiar with basics
of Alexandrov space geometry, see e.g.\ \cite{BBI,BrHa,BGP}.

The implication (2)$\Rightarrow$(1) of Corollary \ref{cor:alexandrov}
is standard.
Let $X$ be an $n$-manifold equipped with a metric of
curvature bounded between $-K_0$ and $K_0$ in the sense of Alexandrov
and injectivity radius bounded below by $i_0>0$.
Then (see e.g.\ \cite{BerNik} for proofs) the tangent cone of
$X$ at every point is isometric to $\R^n$,
all geodesics are uniquely extensible
and hence $X$ has a well-defined exponential map.
The definition of Alexandrov curvature bounds implies that
the exponential map features the same distance comparison properties
as a Riemannian manifold with $|\Sec|\le K_0$,
in particular \eqref{final comparision estimate} holds.
Hence, just like in the Riemannian case
(cf.\ \eqref{eq: GH distance of balls} and Section \ref{sec:injrad}), we have
$d_{GH}(B^X_r(x), B_r^n) \le K_0r^3$
for all $x\in X$ and $0<r\le r_0:=\min\{K^{-1/2},i_0/2\}$.
For $r\ge r_0$ one can use the trivial estimate
$ d_{GH}(B^X_r(x), B_r^n) \le 2r $.
Thus \eqref{e:BXKr3} holds for $K=\max\{K_0,2r_0^{-2}\}$
and all $r>0$.
This proves the implication (2)$\Rightarrow$(1) of Corollary \ref{cor:alexandrov}.

Now we prove the implication (1)$\Rightarrow$(2).
Let $X$ be a complete geodesic space satisfying
\be\label{e:BXKr3}
  d_{GH}(B^X_r(x),B^n_r) \le Kr^3 .
\ee
for some $K>0$ and all $r>0$.
Pick a decreasing sequence $r_i\to 0$ 
such that $r_1\le {\min(1,}\sqrt{\sigma_1/K})$ where $\sigma_1$ is the constant from Theorem~\ref{t:manifold}.
Due to \eqref{e:BXKr3} and the bound on $r_1$, we can apply 
Theorem \ref{t:manifold} to $X$ with $r=r_i$ and $\de=Kr_i^3$.
This yields a sequence of Riemannian $n$-manifolds $M_i$
such that
$|\Sec_{M_i}|\le \Ctwo K$, $\inj_{M_i}\ge r_i/2$,
and $M_i$ is $(1+\Cone Kr_i^2,\Cone Kr_i^3)$-quasi-isometric to~$X$.
By Lemma \ref{l:QI-for-balls} it follows that
$(X,p)$ is a pointed GH limit of $(M_i,p_i)$
where $p\in X$ is an arbitrary marked point and $p_i\in M_i$
corresponds to~$p$ via the quasi-isometry.

Let us  show that the injectivity radii of $M_i$
are uniformly bounded away from~0.
To do this we apply Proposition \ref{p:injrad} to $M=M_i$ and $\tilde M=M_1$.
Due to their quasi-isometry to~$X$,
the manifolds $M_i$ and $M_1$ are $(1+2\Cone Kr_1^2,2\Cone Kr_1^3)$-quasi-isometric to each other.
Hence by Proposition \ref{p:injrad}(2)
\be\label{e:injMi}
 \inj_{M_i} \ge (1-2\cfour \Cone Kr_1^2)\min\left\{\frac{r_1}{2},\frac{\pi}{\sqrt{\Ctwo K}}\right\} -  2\cfour \Cone Kr_1^3 .
\ee
The right-hand side of \eqref{e:injMi} is bounded below by $r_1/4$ if
$r_1\le \cfourteen/\sqrt{K}$ for a suitable constant $\cfourteen>0$
(whose value is determined by the constants provided by Theorem~\ref{t:manifold} and
Proposition \ref{p:injrad}).
Thus all $M_i$ have a uniform lower bound for the injectivity radius
and hence $X$ is a non-collapsed limit of $\{M_i\}$.
As explained in e.g.\ \cite[\S8.20]{Gr}, it follows that
$X$ is an $n$-manifold (with a low regularity Riemannian metric)
with the same bounds for curvature and injectivity radius.

It remains to prove the last claim of Corollary \ref{cor:alexandrov}
(that estimates curvature and injectivity radius of $X$ in terms of~$K$).
The curvature bound $\Ctwo K$ obtained above already has the desired form.
The injectivity radius bound $r_1/4$ gets the desired form $1/(\cfourteen\sqrt{K})$
if we choose $r_1=\min\{\sqrt{\sigma_1},\cfourteen\}/\sqrt{K}$.
\end{proof}

}

\medskip

Finally, we prove Corollary \ref{cor 2:manifold}.

\begin{proof}[Proof of Corollary \ref{cor 2:manifold}]  
Let us consider $\hat \de<\de_0$, where $\delta_0=\delta_0(n,K)$ is chosen later in the proof,
and $r=(\hat \delta/K)^{1/3}$. Then
$r<r_0$, where 
$r_0=(\delta_0/K)^{1/3}$. 
By \eqref{eq: GH distance of balls},
the manifold $N$ is $\hat\de$-close to $\R^n$ at scale $r/2$
provided that above $r_0\le\min\{K^{-1/2},\frac12\inj_N\}$.
Hence the set $X$
with the approximate distance function $\tilde d$
is ${2\hat \de}$-close to $\R^n$ at scale $r/2$.
As in Lemma \ref{l:minchain},
we can replace $\tilde d$ by a ${10}\hat\de$-intrinsic metric $d'$ on~$X$. 
This can be done with standard algorithms for finding
shortest paths in graphs.
By Lemma \ref{l:local-gh-implies-qi}, $(X,d')$ 
is $(1+{40}\hat\de r^{-1},{20}\hat\de)$-quasi-isometric to~$N$.

The metric space $(X,d')$ 
is ${40}\hat\de$-close to $\R^n$ at scale $r/2$.
We may assume that $\delta_0=\delta_0(n,K)$ satisfies $\de_0< K^{- 1/2} \rhode_1^{3/2}$,
where $\rhode_1=\rhode_1(n)$ is given in Theorem \ref{t:manifold}.
Then  
%
 $\de_0< \rhode_1 r_0$.

%
%

As in {Theorem \ref{t:manifold} (see also the algorithm {\it ManifoldConstruction} below),}
using the given data
one can construct a manifold $M=(M,g)$ 
which is 
 $(1+{40}\Cone\hat\de r^{-1},{20}\Cone\hat\de)$-quasi-isometric
 to {$(X,d')$}
and has  {$|\Sec_M|\le \Ctwo \,{40}\hat\de\, (r/2)^{-3}=
 \Csix K,$  where $ \Csix ={80}\Ctwo $.}
Since both $M$ and $N$ are quasi-isometric to $X$ with these parameters,
they are $(1+{80}\Cone\hat\de r^{-1},{40}\Cone\hat\de)$-quasi-isometric to each other.
By Proposition \ref{p:diffeo} it follows that there exists a
bi-Lipschitz diffeomorphism between $M$ and $N$ with bi-Lipschitz
constant {$1+\cthree\,{80}\Cone\hat\de r^{-1}=
1+{80}\cthree\Cone K^{1/3}\hat \delta\,{}^{2/3}$.}
Thus $M$ satisfies the statements 1 and 2 of Corollary~\ref{cor 2:manifold}.
%
%
%
%


To verify the last statement of Corollary \ref{cor 2:manifold},
assume that $\de_0=\de_0(n,K)$ is chosen to be so small that 
$r_0=(\delta_0/K)^{1/3}<(\Ctwo K)^{-1/2}$.
Then Proposition \ref{p:injrad}(2) applies to $M$ and $\tilde M=N$
with ${40}\Cone\hat\de$ in place of $\de$ and $\Ctwo K$ in place of $K$.
It implies that
$$
 \inj_M \ge (1-{80} \cfour \Cone\hat\de r^{-1} )\min \{\inj_N, \pi(\Ctwo K)^{-1/2} \} {-\cfour {40}\Cone\hat\de}
$$
We may assume that $\de_0$ is so small that the term {$1-
{80} \cfour \Cone\hat\de r^{-1}
=1-{80}\cfour \Cone K^{1/3}\hat \delta\,{}^{2/3}$
in this estimate is greater than $\frac12$.}
Then the last statement of Corollary \ref{cor 2:manifold} follows.
Choosing  $\de_0=\de_0(n,K)$ so that the above conditions for $\de_0$ and  $r_0$ are
satisfied,  we  obtain   Corollary  \ref{cor  2:manifold}.
\end{proof}

\section{Manifold reconstructions based on Theorems \ref{t:manifold} and \ref{t:surface}}
\label{sec:constructive}


{Now we change the gear and explain how the above geometric proofs can be developed to {manifold reconstruction} procedures. }

\subsection{Outline of reconstruction procedures}

The constructive proofs of
Theorems \ref{t:manifold} and \ref{t:surface} yield
algorithms that can be used to produce {submanifold}s or manifolds from
finite data sets. We give only the 
{sketches of the algorithms that could be written of based on these theorems.
Adding the necessary details to make these sketches {\ctext numerically implementable algorithms  needs more work and it is outside
the scope of this paper.
However, for sake of brevity, we below refer these sketches as algorithms and procedures.
These} algorithms use the sub-algorithms
{\it FindDisc} and {\it GHDist} given in Sections \ref{subsec:algorithm-gh}
and \ref{subsec:algorithm-finddisc}. 
In the description of the algorithm we assume
that the data set $X$ is finite.

First we outline the algorithm based on Theorem \ref{t:surface}.

\medskip
\underline{Algorithm SubmanifoldInterpolation:}
Assume that we are given the dimension $n$, the scale parameter $r$, and 
a finite set  points $X\subset E=\R^m$.
We suppose that $X$ is $\de r$-close to $n$-flats at scale $r$
where $\de$ is sufficiently small.
Our aim is to construct a {submanifold} $M\subset {\R^m}$ that approximates  $X$.
We implement the following steps: 
\begin{enumerate}
\item We rescale $X$ by the factor $1/r$. After this scaling, the problem
is reduced to the case when $r=1$.

\item 
We choose a maximal $\frac1{100}$-separated set $X_0\subset X$ and
enumerate the points of $X_0$ as $\{q_i\}_{i=1}^{N}$.
We apply the algorithm {\it FindDisc} to every point $q_i\in X_0$
to find an affine subspace $A_i$ through $q_i$ such that the
unit $n$-disc $A_i\cap B_1(q_i)$ lies within 
Hausdorff distance ${\Ctwelve {n}}\delta$ from the set $X \cap B_1(q_i)$.
We construct the orthogonal projectors
$P_i\co {\R^m}\to {\R^m}$ onto $A_i$.

\item
We construct the functions $\varphi_i:{\R^m}\to {\R^m}$, defined in 
(\ref{e:def of phi_i}), that are convex combinations of the projector $P_i$ and  the identity map.
Then we iterate these maps to construct 
 $f\co {\R^m}\to {\R^m}$, $f= \varphi_{N}\circ\varphi_{{N}-1}\circ\ldots\circ\varphi_1$, see (\ref{e:def of fi}).
 
 \item We construct the image $M= f({\mathcal U}_\de(X))$  of the $\de$-neighborhood
 of the set $X$  in the map $f$, see Remark \ref{rem:small-tube}.
 \end{enumerate}

The output of the algorithm  SubmanifoldInterpolation is the $n$-dimensional
{submanifold} $M\subset {\R^m}$.
\medskip

The algorithm based on Theorem \ref{t:manifold} is the following.
\medskip

\underline{Algorithm ManifoldConstruction:}
Assume that we are given the dimension $n$, the scale parameter $r$, and 
a finite metric space $(X,d)$. 
Our aim is to construct a smooth $n$-dimensional 
Riemannian manifold $(M,g)$ approximating $(X,d)$.
We implement the following steps: 
\begin{enumerate}
\item We multiply all distances by $1/r$. After this scaling, the problem
is reduced to the case when $r=1$.

{
\item
{We select a maximal $\frac 1{100}$-separated subset $X_0\subset X$
and enumerate the points of $X_0$ as $\{q_i\}_{i=1}^N$.}
{
We choose a set $\{p_i\}_{i=1}^N$ such that the unit balls
$D_i=B_1^n(p_i)\subset\R^n$ are disjoint.}

\item
{For each $q_i\in X_0$, we apply the algorithm {\it GHDist} 
to the ball $B_1(q_i)\subset X$ to find the value $\de_a(q_i)$.
Define $\de_a=\max_{q\in X_0} \de_a(q)$, see
\eqref{delta alpha estimate0} and \eqref{delta alpha estimate}.

}}

\item For all $q_i,q_j\in X_0$  such that $d_X(q_i,q_j)<1$, we
construct the affine transition maps $A_{ij}\co\R^n\to\R^n$ using the maps
${{\bf F}_{(i)}}:B_1(q_i)\to D_i$ and ${{\bf F}_{(j)}}:B_1(q_j)\to D_j$ and the construction given in
{
Lemma \ref{l:rotation} and formula \eqref{e:transition2b}.}

\item
Denote $\Omega_0= \bigcup_{i=1}^N D_i^{1/10}$,
where $D_i^{1/10}=B_{1/10}(p_i)\subset\R^n$,
and $E=\R^m$, $m={(n+1)N}$.
We construct a Whitney embedding-type map  
$$
F\co\Omega_0\to {\R^m},\quad F(x)=(F_i(x))_{i=1}^N
$$  
where $F_i:\Omega_0\to \R^{n+1}$  
are given by (\ref{e:Fi definition}).

 \item We construct the local patches $\Sigma_i=F(D_i^{1/10})$ {and 
${\kappa_0}$-net $Y_i=\{y_{i,k}\}_{k=1}^{K_i}$ in $\Sigma_i$ that is 
${(\kappa_0}/2)$-separated,}
where ${\kappa_0}$ is the constant from Proposition \ref{p:diffeo}.

\item We apply algorithm {\it SubmanifoldInterpolation} for the 
points $\{y_{i,k};\ 1\le i\le N,\ 1\le k\le K_j\}$ to obtain a {submanifold} $M\subset {\R^m}$.
We construct the normal projector $P_M\co {\mathcal U}_{2/5}(M)\to M$ for the {submanifold} $M$.

\item   We construct maps $
 \psi_i = P_M\circ F|_{D_i^{1/10}}:D_i^{1/10}\to P_M(\Sigma_i)\subset M$.
 
%
%
 \item We construct  metric tensors $g_i$ on sets  $P_M(\Sigma_i)\subset M$ 
   by pushing forward the Euclidean metric
$g^e$ on $\Omega_0$ to {the sets $P_M(\Sigma_i)$} 
using the maps $\psi_i$. Then metric $g$ on $M$  is constructed
by using a partition of unity to compute 
a weighted average of the obtained metric tensors, see (\ref{e:g in coordinates}).
\end{enumerate}

The output of the algorithm is the {submanifold} $M\subset {\R^m}$ and the metric $g$ on it.

\medskip

We note that  by Lemma \ref{l:gh algorithm} and formula \eqref{delta alpha estimate}, 
 we have for all $x\in X_0$ that the Gromov-Hausdorff distance 
between the ball $B_1(x)$ and $B_1^n$ is at most  $2\delta_a(x)$. 
A sufficient condition for the correctness of the {\ctext algorithm}  {\it ManifoldConstruction} is that  $\de_a$
computed in the step (3)  is smaller than 
the constant $\de_0(n)/2$, where $\de_0(n)$  is given in Proposition 
\ref{p:manifold}, see Remark \ref{de 0 value}. 

Also, we note that  {in the proof of 
Proposition 
\ref{p:manifold} the construction of  the submanifold $M$ and in the above algorithm we use  only  the $r$-balls in $X$ centered at the points of a maximal
$\frac r{100}$-separated subset $X_0$ of $X$,  see  Remark \ref{centers in X0}.}

\subsubsection{An alternative construction with the map $f$ replacing the projector $P_M$.} \label{rem. simplification} 
{The numerical computation of the projector $P_M$, mapping a point to the nearest point on manifold $M$, may
be difficult. To overcome this practical difficulty, we observe that
the manifold $M$ given by the algorithm 
{\it ManifoldConstruction} can be constructed using 
the functions $\tilde \psi_i=f\circ F|_{D_i^{1/10}}$, instead
of functions $\psi_i = P_M\circ F|_{D_i^{1/10}}$,
see Remarks \ref{rem:f and PM} {and \ref{e:def of f finite X0}. In other words, the steps (7), (8), and (9) can be replaced by
\begin{itemize}

\item [(7')] We apply algorithm {\it SubmanifoldInterpolation} for the 
points $\{y_{i,k};\ 1\le i\le N,\ 1\le k\le K_j\}$ to obtain a {submanifold} $M\subset {\R^m}$.
The map $f$ constructed in the step 3 of the algorithm {\it SubmanifoldInterpolation}
gives a map $f\co {\mathcal U}_{r/10}(M)\to M$ from the neighborhood of $M$ onto $M$,
see Remark \ref{rem:f and PM}.

\item[(8')]   We construct maps $
\tilde \psi_i=f\circ F|_{D_i^{1/10}}:D_i^{1/10}\to f(\Sigma_i)\subset M$.

 \item[(9')]  We construct  metric tensors $g_i$ on sets  $f(\Sigma_i)\subset M$ 
   by pushing forward the Euclidean metric
$g^e$ on $\Omega_0$ to {the sets $f(\Sigma_i)$} 
using the maps $\tilde \psi_i$. Then the metric $g$ on $M$ is 
a weighted average of the these metric tensors using a suitable partition of unity, {see (\ref{local metric}) below}.
 \end{itemize}

} 
When  $\tilde D_i=B^n_{1/{30}}(p_i)\subset \R^n$ and $\tilde \Sigma_i=\tilde \psi_i(\tilde D_i)$, $i=1,2,\dots,N$
the algorithm gives the maps
$\tilde \psi_i^{-1}:\tilde \Sigma_i\to \tilde D_i$
 that by Lemma \ref{l:charts} can be considered
as local coordinate charts of $M$ that cover the whole manifold $M$
and the transition functions {$\tilde \eta_{ji}= \tilde\psi_j^{-1}\circ \tilde\psi_i$
that map $$
\tilde \eta_{ji}:\tilde V_{ij}=\tilde\psi_i^{-1}( \tilde\psi_i(\tilde D_i)\cap  \tilde\psi_j(\tilde D_j))\to \tilde V_{ji}=\tilde\psi_j^{-1}( \tilde\psi_i(\tilde D_i)\cap  \tilde\psi_j(\tilde D_j)).$$
Note that the transition functions $\tilde \eta_{ji}$ need to be approximated numerically, e.g.\ using Newton's algorithm.}

{To construct the metric tensor $g$ on these charts, we can first take a family non-negative functions
${v}_i\in C^\infty_0(\tilde D_i)$ {\ctext such that   
${v}_i(x)=1$ for $x\in B^n_{1/{50}}(p_i)$.} Then 
the sets $\{x\in M:\ (\tilde \psi_i^{-1})^*{v}_i(x)>0\}$ cover $M$ and define a partition of unity on $M$  by
\beq\label{partition}
{\tilde v}_j(x)=\bigg(\sum_{i=1}^N  ((\tilde \psi_i^{-1})^*{v}_i)(x)\bigg)^{-1} ((\tilde \psi_j^{-1})^*{{v}}_j)(x).
\eeq
Then the  metric tensor  $(\tilde \psi_j)^*g$ on the chart $\tilde D_j$
is given by
\beq\label{local metric}
((\tilde \psi_j)^*g)(y)=\sum_{k=1}^N  \bigg(\sum_{i=1}^N  {{v}}_i(\tilde \eta_{ji}(y))\bigg)^{-1} 
{{v}}_k(\tilde \eta_{jk}(y))\,(\tilde \eta_{jk})_*g_e,
\eeq
where $g_e$  is the Euclidean metric on  $\tilde D_k$.}

The collection of local coordinate charts $\tilde D_j$, metric tensors $g^{(j)}=(\tilde \psi_j)^*g:\tilde D_j\to \R^{n\times n}$, and transition functions $\tilde \eta_{ij}
 :\tilde V_{ij}\to \tilde V_{ji}$   is a representation of the Riemannian manifold $M$
 in local coordinate charts.}
  {\ctext Using this representation we can determine
 the image of a geodesic $\gamma_{x_0,\xi_0}(s)$, emanating from $(x_0,\xi_0)\in TM$, on several
 coordinate charts $\tilde \psi_i^{-1}:\tilde \Sigma_i\to \tilde D_i$ and  determine the metric tensor in the Riemannian normal coordinates, \cite{Pe}.
 Thus, for practical imaging purposes, for instance to visualize the  $n$-dimensional manifold $(M,g)$, the  algorithm {\it ManifoldConstruction}  can be continued with the following steps
\begin{enumerate}
 \item[(10)] For given $x_0\in M$, determine the metric tensor $g$ in the 
 normal coordinates given by the map $\exp_{x_0}:\{\xi \in T_{x_0}M:\ \|\xi\|_g<\rho\}$,
 where $\rho<\hbox{inj}_M$.
 
 \item[(11)] For  given $x_0\in M$ and two linearly independent vectors $\xi_1,\xi_2 \in T_{x_0}M$, 
 visualize the properties of the metric $g$, e.g.\ the determinant of the metric, in the normal coordinates,
by computing
  the map \beq 
  s=(s_1,s_2)\mapsto \det(g(\exp_{x_0}(s_1\xi_1+s_2\xi_2))),\eeq 
  in the set  $\{s\in \R^2: \|s_1\xi_1+s_2\xi_2\|_g<\rho\}$. This
   produces an image of the metric
  in a two-dimensional slice of the manifold.  Moreover, consider a data point $x\in X$ 
  such that $x\in B^X_r(q_j)$, with some index $j$,  and that its image $y=\tilde \psi_j(f_j(x))$ on $M$
  satisfies $y\in B_\rho^M(x_0)$. Then the vector $\xi=\exp_{x_0}^{-1}(y)\in T_{x_0}M$
  corresponds to the data point $x$ in the tangent space of $M$ at $x_0$.
  In the visualization of the  two-dimensional slice, the data point $x$ can be visualized 
  as the projection $\bar \xi$  of the vector $\xi$ to the plane span$(\xi_1,\xi_2)$. In this way both the metric and the original data points $X$ can be visualized in two dimensional slices of the manifold.
\end{enumerate}
Practical imaging methods similar to step (11) above have been used in seismic imaging, for example in the imaging of the wave speed function in the time-migration coordinates, see e.g.\ \cite{Cameron}. 
}

\subsubsection{{\ctext Numerical approximation} of the extended transition functions using a Newton-type algorithm} 
%
\label{rem. simplification 2}
 {
 The functions  $\tilde \psi_i=f\circ F|_{D^{1/10}_i}$, $i=1,2,\dots, N$, discussed in subsection \ref{rem. simplification}
 and used
  in the steps (8')-(9') of the algorithm, are piecewisely defined by explicit formulas.
   Next we discuss, how {the inverse functions of these maps  and the extensions of the transition functions $\tilde \eta_{ji}$  can be approximated using a Newton-type algorithm.} 

To consider the inverse function of $\tilde \psi_i$, we first reduce the problem to finding an inverse function to a map between $n$-dimensional spaces.

We  construct the tangent spaces
$$T_{i}:=y_i+\hbox{Ran}(d\tilde \psi_i(p_{i}))$$ 
of the $n$-dimensional {submanifold}s $\tilde \psi_i(D^{1/10}_i)\subset \R^m$ at $y_i=\tilde \psi_i(p_{i})$, where $i=1,2,\dots, N$,
{and $\hbox{Ran}(A)$ denotes the range (i.e., the image) of the operator $A$.}
Recall that for $x\in M$,  the map $dP_M(x)$ is the orthogonal projector in $T_x{\R^m}={\R^m}$ 
onto $T_xM$. Denote $P_x=dP_M(x)$ and  $P_{i}=dP_M(y_{i})$.
Then $P_{i}:\R^m\to T_{i}$ are
the orthogonal projections. Below, $B^m_R(y)\subset \R^m$ is the ball having the radius $R$  and the   centre $y$.

%


%
%

 {Then, we compose $ \tilde \psi_i$ with a projector $P_j$ and an affine isometry $A_j: T_i \to\R^n$ and obtain a map 
\be\label{Gi functions}
 G_{j,i}:=A_j\circ P_j \circ \tilde \psi_i: {D^{1/10}_i} \to \R^n.
 \ee
In particular, we are interested in the maps $ G_{j}= G_{j,j}$. These maps are used below to determine the
 extended transition functions in formula (\ref{extended transition functions}).
 }

%
%
%
%
%
%
%
%

First we recall some estimates proven above.
{We recall that the constants $C$ and $C_k$, depend only on dimension $n$.}
By Lemma \ref{l:Fsmooth}, 
\beq\label{added C21}
\|F\|_{C^2(\Om)} \le {\Ctwentynine},
\eeq
 where {${\Ctwentynine}=\Ctwentyfive(2)$.
Next we use  {Lemma \ref{l:submanifold}, or equivalently, Theorem \ref{t:surface} with
$r_0\leq 1 $ in place of $r$ and $ \Cthirtythree  r_0^2$ in place of $\de$, and  choose later the value of $r_0$ so that it depends only on $n$. 
{We also use the fact that
 by (\ref{e:Fsmooth}) and (\ref{e:GH Sigma M}), $M\subset E$    is in the ball of radius
 $\Ctwentyfive(0)+\tfrac1{10}$ centered at zero, and hence
\be \label{e:sup  PM}
 \|P_M\|_{C^0( {\mathcal U}_{r_0/3}(M))} \le \Ctwentyfive(0)+\tfrac1{10}+\tfrac {r_0}3\le {\Clast}=
 \Ctwentyfive(0)+\tfrac1{10}+\tfrac {1}3.
\ee
When}} we
denote  {
 $\Cbracetthree={\ctext {{\Clast}}}+\Cthirtysix(2)+\Cthirtysix(1)$},  
%
%
%
%
%
%
%
 we have by Lemma \ref{l:submanifold}} {\ctext and the fact that $r_0\leq 1$,}
  \beq\label{e C61 interpolated}
 \|P_M(x)\|_{C^{2}( {\mathcal U}_{r_0/3}(M))}\leq {\ctext \Cbracetthree}.
  \eeq
 By Remark \ref{rem:f and PM} 
\be 
\|f - P_M\|_{C^k({\mathcal U}_{r_0/10}(M))} \leq \Cbracetfour r_0^{2-k},\quad k=0,1,2,
\label{added a1}
\ee 
{where $\Cbracetfour=\Cthirtythree \max (\Cseventeen(2),\Cseventeen(1),\Cseventeen(0)) $.}
Then, using interpolation in H\"older spaces \cite{BL} to inequalities 
(\ref{added a1}) with $k$ being 1 and 2, we see that 
\be 
\|f - P_M\|_{C^{1, 1/2}({\mathcal U}_{r_0/10}(M))} \leq \Cbracetfour r_0^{1/2}.
\label{added a1 interpolated}
\ee 

Lemma \ref{l:submanifold}(1) and the formulas (\ref{added C21}), (\ref{added a1 interpolated}), and (\ref{added a1}) yield that there is $\cbracetone>0$  such that 
\be\label{C2 estimate}
\|\tilde\psi_i\|_{C^2(D^{1/10}_i)} \leq   \cbracetone\quad\hbox{and}\quad \|G_{j,i}\|_{C^2(D^{1/10}_i)} \leq   \cbracetone.
\ee
By (\ref{e:psi-F-mod2}), there is a constant ${\Cbracetfive=\Cthirtynine^{-1}}>0$ such that
the maps $\psi_i = P_M\circ F:D^{1/10}_i\to \R^m$, defined in (\ref{e:psi-F-mod1}),
 satisfy 
\be\label{e C65}
\big| d\psi_i|_x(v)\big|\ge \Cbracetfive|v|,\quad \hbox{for }x\in D^{1/10}_i,\ v\in \R^n.
 \ee
 When $r_0 <(\Cbracetfive/ (2\Cbracetfour))^2,$ the formulas (\ref{added a1 interpolated}), \eqref{e C65}, and the identity
 $P_M\circ \tilde \psi_i=\tilde \psi_i$  
 imply that for  $z=\tilde \psi_i(x)$ 
 we have
 \be\label{eq dpsi1}
\big|P_z(d\tilde \psi_i|_x(v))\big |=\big| d\tilde \psi_i|_x(v)\big |\ge \frac 12 \Cbracetfive |v|,\quad x\in D^{1/10}_i.
\ee
 {Assume next that $$\tilde \psi_i(D^{1/10}_i)\cap \tilde \psi_j(D^{1/10}_j)\not=\emptyset.$$ Then we have, by {\ctext \eqref{added a1 interpolated}}
 and \eqref{C2 estimate}, for  $z\in \tilde \psi_i(D^{1/10}_i)$} that 
 \beq
 \|P_z-P_{y_j}\|
 \leq 
 \frac 2{10}{\ctext \Cbracetfour}
  \cbracetone r_0^{1/2}.\eeq
 So, when  
 $r_0 <  (\Cbracetfive/({\ctext 2\Cbracetfour} \cbracetone))^2$, we have
  \be\label{eq dpsi}
\big|{P_j}(d\tilde \psi_i|_x(v))\big |\ge \frac 14 \Cbracetfive |v|,\quad x\in D^{1/10}_i.
\ee
Now we choose $r_0=\min((\Cbracetfive/ (2\Cbracetfour))^2,(  \Cbracetfive/({\ctext 4\Cbracetfour}  \cbracetone))^2)$
in the above use of   Lemma \ref{l:submanifold} so that the above conditions for $r_0$ are valid.

As ${A_j}  :T_j\to \R^n$ is an affine isometry, (\ref{eq dpsi}) implies
\be\label{e C66}
\| (dG_{j,i}(x))^{-1}\|\le \cbracettwo=\frac 4{\Cbracetfive},
\quad x\in D^{1/10}_i.
 \ee

 {\ctext Denote
$
\cbracetthree=\max(3\cbracetone\cbracettwo,1),  
$ 
and choose 
\beq
\rho_n=\min(\frac1{100},\frac1{20 (\cbracetthree)^2},
\frac {r_0}{100\cbracetone}).\eeq   }

Recall that  $D_i=B_{1/10}^n(p_i)$ and $\tilde D_i=B_{1/30}^n(p_i)$ and let $R_n=\cbracetone\rho_n$.
As $\rho_n\leq 1/100,$ formula
(\ref {C2 estimate}) yields
 for $x\in B^n_{1/20}(p_i)$  that
 $\tilde \psi_i(B^n_{\rho_n}(x))\subset B^m_{R_n}(\tilde \psi_i(x)).$

{Our next aim is to  cover the set
 $\tilde D_i$ by small balls
 of radius $\rho_n$, and to use Newton's method to find the transition functions in these balls.


To consider how the transition functions can be constructed with a numerical algorithm, we first  will extend these functions to be defined in larger domains.} 
We call these functions the
extended transition functions.
To this end, 
let  $h_k\in B^n_{1/30}(0)\subset \R^n$, $k=1,2,\dots,K$ be a maximal set of $\rho_n$-separated points in $B^n_{1/30}(0)$.
Note that $K$ is bounded by $\hbox{vol}( B^n_{1/10}(0))/\hbox{vol}( B^n_{\rho_n/2}(0))$.

For $a\in \Z_+$, denote $$\mathcal V_i(a)=\bigcup_{k=1}^K  B^m_{aR_n}(\tilde \psi_i(p_i+h_k))\subset \R^m.$$
Assume next that  $\tilde\psi_i(\tilde D_i)\cap \tilde\psi_j(\tilde D_j)\not=\emptyset$.
If  
 $x_i\in \tilde D_i= B^n_{1/30}(p_i)$
 is such that  $$\tilde\psi_i(x_i)\in   \tilde\psi_i(\tilde D_i)\cap \tilde\psi_j(\tilde D_j),$$
 there exists $x_j\in \tilde D_j$ so that $\tilde\psi_j(x_j)=\tilde\psi_i(x_i)$.
 Then there are 
  $k_i$ and $k_j$  such that  $x_i\in B^n_{\rho_n}(p_i+h_{k_i})$ and $x_j\in B^n_{\rho_n}(p_j+h_{k_j})$ and
 we see that 
  \begin{eqnarray}\label{eq: inclusions 1}
& &  \tilde \psi_i(p_i+h_{k_i})\in
B^m_{R_n}(\tilde \psi_i(x_i))=B^m_{R_n}(\tilde \psi_j(x_j))\\
 \nonumber & &  \hspace{25mm}\subset
B^m_{2R_n}(\tilde \psi_j(p_j+h_{k_j}))\subset \mathcal V_j(2).
  \end{eqnarray}

%
%
%
%
Let $\mathcal K(j,i)=\{k\in \{1,2,\dots,K\}:\ \tilde \psi_i(p_i+h_{k})\in  \mathcal V_j(2)\}$ and  $$W_{ji}=\bigcup_{k\in \mathcal K(j,i)}  B^n_{\rho_n}(p_{i}+h_{k})\subset \R^n.$$
By (\ref{eq: inclusions 1}), we have $$\tilde\psi_i^{-1}( \tilde\psi_i(\tilde D_i)\cap  \tilde\psi_j(\tilde D_j))\subset W_{ji}$$ and
 the function 
 \be\label{extended transition functions 0}
 \tilde \eta_{ji}^e=\tilde\psi_j^{-1}\circ \tilde\psi_i:W_{ji}\to D_{j}
 \ee
  is an extension of the
transition function  $\tilde\eta_{ji}$, that is, it coincides with $ \tilde \eta_{ji}$ in the set $\tilde\psi_i^{-1}( \tilde\psi_i(\tilde D_i)\cap  \tilde\psi_j(\tilde D_j))$.
Moreover, for $k\in  \mathcal K(j,i)$
  \begin{eqnarray}\label{eq: inclusions 2}
 \tilde \psi_i(B^n_{\rho_n}(p_i+h_{k}))&\subset& 
 B^m_{R_n}(\tilde \psi_i(p_i+h_{k}))\\
 \nonumber
& \subset &
\bigcup_{z'\in   \mathcal V_j(2)} B^m_{R_n}(z')
\subset \mathcal V_j(3),
  \end{eqnarray}
that implies 
\beq\label{set W3}
& &G_{j,i}(W_{ji})\subset \mathcal W_j(3)\hbox{ and }G_j(W_{ji})\subset \mathcal W_j(3),\eeq
{\ctext where, as $P_j$ is an orthogonal projector and $A_j$ is an affine isometry,
 \ba
\nonumber \mathcal W_j(3)&:=&A_j(P_j(\mathcal V_j(3)))\\
&=&\bigcup_{k=1}^K  A_j(P_j(B^m_{3R_n}(\tilde \psi_j(p_j+h_k)))
\\
&=&\bigcup_{k=1}^K  B^n_{3R_n}(p_{j,k})\subset \R^n,
\ea
and $p_{j,k}=A_j(P_j(\tilde \psi_j(p_j+h_k)))$.
To} compute the extended transition function $\tilde \eta_{ji}^e$ 
it is enough to compute the inverse function
\beq\label{Gij inverse} G_j^{-1}=(A_j\circ P_j\circ \tilde \psi_j)^{-1}:\mathcal W_j(3)\to D_j,\eeq
as then we can write 
\be\label{extended transition functions}
\tilde \eta_{ji}^e= \tilde \psi_j^{-1}\circ \tilde \psi_i=(A_j\circ P_j\circ \tilde \psi_j)^{-1}\circ A_j\circ P_j\circ \tilde \psi_i=G_j^{-1}\circ G_{j,i}:
 W_{ji}\to  D_i.
\ee

Next, to consider various push forwards of metric tensors (see (\ref{g push forwards}) below), we analyse the computation of {\ctext the restrictions of the inverse functions in balls ${\ctext B^n_{3R_n}(p_{j,k})}$, that is,}  $G_{j,i}^{-1}|_{\ctext B^n_{3R_n}(p_{j,k})}:{\ctext B^n_{3R_n}(p_{j,k})}\to D_i$.
These maps give us the inverse maps $G_{j,i}^{-1}:\mathcal W_j(3)\to D_i$.
In the case when $i$ and $j$ are equal, this give us also  the function $G_j^{-1}:\mathcal W_j(3)\to D_j$.

To consider (\ref{Gij inverse}),
assume 
 that we are given $z\in \mathcal W_j(3)$. Then we can determine $k_0\in \{1,2,\dots,K\}$ such that $z\in 
 {\ctext  B^n_{3R_n}(p_{j,k_0})}$.
Then, we will start the iteration in Newton's algorithm from  ${x}_0=p_{i}+h_{k_0}$. 
The iterations of Newton's method proceed as follows. For $p\geq 0$, 
 \begin{equation} \label{eq: iteration}
{x}_{p+1} = {x}_p- (dG_{j,i}({x}_p))^{-1}(G_{j,i}({x}_p) - z). 
  \end{equation}
As 
$$
| (dG_{j,i}(x_0))^{-1}(z-G_{j,i}(x_0))|<3  \cbracettwo R_n=\cbracetthree\rho_n=r_1
$$
and  $ q:=\cbracetone\cbracettwo(\cbracetthree\rho_n){\ctext \leq 
(\cbracetthree)^2\rho_n}<\frac 12$, it follows by
the convergence theorem for Newton's algorithm \cite[Thm.\ 6.14]{Kress} that
the sequence $({x}_p)_{p=1}^\infty$ 
stays in ${\ctext B^n_{r_1}(p_{i}+h_{k_0})}
\subset D_{i}$ and it converges to the limit point ${x}={G_{j,i}^{-1}(z)}$ and moreover,
this sequence satisfies 
 \begin{equation} \label{eq:23h} |{x}_p- {x}| \leq 2\cbracetthree\rho_nq^{2^p-1}.
  \end{equation}
  Note that the  above also shows that $\mathcal W_j(3)\subset  G_{j,i}( D_i)$.  
Summarising, the above shows that  the Newton's algorithm can be used to compute the inverse functions of
${G_{j,i}}$ 
and of the extensions of the transition functions ${\tilde \eta}^e_{ji}$.

 }

{

\subsection{Analysis of the computational complexity}
\label{subsec :complexity 1}
\subsubsection{Computational complexity of the algorithm {\it SubmanifoldInterpolation} }
\label{rem:complexity 2a}


We analyse the computational complexity of the above algorithms  in terms the number
of elementary computational operations needed {(see Remark \ref{rem:complexity 1}). We note that as we have only presented the sketches of the above algorithms, in the considerations below we do not analyze the real computational requirements, in particular in the sense that we do not consider how much computational resources the needed elementary operations use}. {\ctext Below we consider two types of computational requirements: First the requirements for the one-time work that corresponds to the preparatory work and second the computational work required to answer to a query that produces a point of the constructed manifold. These 
are then combined to estimate the computational work required to obtain a grid of points on the manifold. We also recall that below $C$  denotes a constant which depends only on $n$,
that is $C=C(n)$, but the value of $C$ may change even inside a formula line.} 

Let $\ell=\# X$  be the number    of elements in the set $X$.
We assume that $E=\R^m$  and $X\subset {\R^m}$  satisfies the assumptions of Theorem \ref{t:surface} with sufficiently small $\delta$ 
and $r\leq 1$ so that 
 the set $X$ is $\delta$-close to $n$-flats in scale $r$. Also, we assume that $\diam(X)<D$.
In the step 2 of  {\it SubmanifoldInterpolation} we construct a   maximal $r/100$ separated subset $X_0\subset X$. Since
 $\sec(M)\leq K=C\de r^{-3}$, 
 the number  $N=\# X_0$  of elements in the set $X_0$ satisfies
  $N\leq \min(\ell,C\left( e^{CKD}/ r \right)^n).$ 
In the step 2  the construction of the set $X_0$ can be done by going through all points $x$ in $X$ one by one and include it in $X_0$ if the distance from $x$ to some of the points chosen earlier to be included in $X_0$ at least $1/100$. This  requires {at most $N\ell\leq \ell^2$ operations.}
Also, in the 
 steps 2-3 of the algorithm, one applies algorithm  {\it FindDisc}  $N$ times to a set consisting of $\ell$ points in an $m$-dimensional space
 and this requires $CmN\ell$ operations.
 Thus, the steps 1-3, that construct a function $f$, require altogether $CmN\ell$ operations.
  {\ctext Observe that $f$ is given as a composition of explicit functions containing parameters that depend on the data,
  that is, the coordinates of the points $q_i$ in \eqref{mu function 2} and the matrixes and vectors that determine the affine projectors $P_i=P_{A_{q_i}}$ in \eqref{PAx}.
   Thus the construction of $f$ means the determination of the values of these parameters. These steps for finish the one-time work requited for the preparatory steps. The preparatory steps require altogether
 $$
 N\ell+CmN\ell\leq  CN\ell m
 $$
elementary operations.}

Recall that the {submanifold} $M\subset {\R^m}$ is the image of a $\delta$-neighbourhood of the set $X$ in the constructed function $f$.  
{\ctext Moreover, by Lemma \ref{l:disc-cover}, $M$ is the union images of the
$n$-dimensional discs $A_i\cap B_1(q_i,r)$ of radii $r$, that is,
$M=\bigcup_{i=1}^N f(A_i\cap B_1(q_i,r))$.
Obtaining one point of manifold $M$ can be considered as
a query where the input consists of an index $i$  and a point $z\in B_1(q_i,r)$ 
and the query gives the answer $f(z)$.
 As we have already constructed the function $f$, or more precisely the parameters that determine this function taking values in $\R^m$, answering such a query requires 
$Cm$ elementary operators (that is, finitely many operations for each $m$ coordinates of $f(x)$).

Let us next consider computing a grid of points on $M$. To this end, let
$0<\eta<\delta$  be a small parameter. Then,} choose 
an $\eta$-dense computational grids in the $n$-dimensional discs $A_i\cap B_1(q_i,r)$, having radii $r$.
Let  us denote these computational grids by $(z_{i,j})_{j=1}^J$, $J\leq C(r/\eta)^n$. Then {\ctext we obtain a grid of points on the}
{submanifold} $M$ by constructing a $(C\eta)$-dense subset $\{f(z_{i,j});\ i=1,2,\dots,N,\ j=1,\dots,J\}$ of $M$. 
This requires $CNJm$  operations. {\ctext Summarising, computing
the points  $\{f(z_{i,j});\ i=1,2,\dots,N,\ j=1,\dots,J\}$ on $M$ using the algorithm  {\it SubmanifoldInterpolation}
require} altogether $$
CmN\ell+CNJm
\leq  Cm\ell ^2+C\ell m (r/\eta)^n$$
 operations. }

 

\subsubsection{Computational complexity of the algorithm {\it ManifoldConstruction}.}
\label{rem:complexity 2b}
We estimate the number of elementary operations needed in  {\it ManifoldConstruction} to compute
the transition functions between local charts and the metric on the charts with given numerical accuracy. By rescaling in the step 1, that is, by multiplying the distance function $d_X:X\times X\to \R$ and the geometric parameters $r$ and  
$\delta$ by factor $1/r$, it suffices to handle the case $r=1$. Because of this, we analyse the complexity of the algorithm
in the case when $r=1$ and $\delta<\delta_0(n)$, where $\delta_0(n)<1$ appearing 
in Proposition \ref{p:surface} depends only on $n$. 

We will analyse the computational complexity of the  alternative version of the algorithm described in
the subsections \ref{rem. simplification} and \ref{rem. simplification 2}.
%

{\ctext Again, we start with the requirements of the one-time work needed for the preparatory work.}
We assume that  a finite metric space $X$  satisfies assumptions of Theorem \ref{t:manifold} and that  $\diam(X)<D$.
Again,  let $\ell=\# X$ be the number    of elements in the set $X$.
Below, $\theta>0$ will be the parameter corresponding to the required numerical accuracy.

 
{
We first observe that the step 2 of the algorithm requires $C\ell^2$  steps.

In the step 3, we apply algorithm {\it GHDist} to all balls $X_1^{(i)}:=B_1(q_i)\subset X$ with $i=1,2,\dots,N$,
that is, for balls centered in points of $X_0$. Let $\ell_i=\#X_1^{(i)}$. By Lemma \ref{l:adjgraph}, each point $x\in X$ belongs at most $C=C(n)$ balls $X_1^{(i)}$,
and thus $\sum_{i=1}^N \ell_i\leq C\ell$. Applying algorithm {\it GHDist} to the ball $X_1^{(i)}$
requires $C\ell_i^2$ elementary operations, so step 3 requires altogether
\ba
\sum_{i=1}^N C\ell_i^2\leq C (\sum_{i=1}^N \ell_i)^2\leq C\ell^2
\ea
elementary operations.}

Since
by  Theorem \ref{t:manifold} the set $X$ is $(C\delta)$-close to a smooth $n$-dimensional  manifold $M$
such that $\sec(M)\leq K=C\de $, we have that 
 the number  $N=\# X_0$  of elements in the maximal $\frac 1{100}$-separated set $X_0$ satisfies $N\leq \min(\ell,C(e^{CKD})^n).$
Below, we give estimates in terms of  $N$  and $\ell$. 
In the step 3, we construct the set $X_0=\{q_i\}_{i=1}^N\subset X$, and at the same time make a record of those elements
$q_i,q_j\in X_0$  for which $d_X(q_i,q_j)<1$. This  requires $CN\ell$  operations. 
{Below, we use the fact that by Lemma \ref{l:adjgraph}, for any $i$ the number of $j$  such that 
 $d_X(q_i,q_j)<1$ is bounded by a number depending only on $n$.}

{In the step 4, finding  maps $A_{ij}$, $i,j=1,2,\dots,N$
as in Lemma \ref{l:rotation} and formula \eqref{e:transition2b} requires $CN\ell$ operations.
Indeed, for each $i\in\{1,\dots,N\}$ there is a bounded number of maps $A_{ij}$ to construct,
see Lemma \ref{l:adjgraph}.
The construction of $A_{ij}$ described in the proof of Lemma \ref{l:rotation}
requires a number of operations proportional to the number of points in the ball $B_1(q_i)$.}
In the step 5 we introduce the space $E=\R^{m}$, where $m=(n+1)N$.
We emphasize that here the dimension $m$ of the space ${\R^m}$ depends on $N$ and therefore on the metric space $X$.
 In this step, the construction of the map $F$  requires
$CN^2$ operations. The construction of {the ${\kappa_0}$-nets in $Y_i\subset \Sigma_i$ in the step 6 requires $CN\,\cdotp m$  operations. Note that by Lemma \ref{l:Fsmooth}, the maps $F:D_i^{1/10}\to \Sigma_i$ are bi-Lipschitz with the Lipschitz constant $\Ctwentysix$, and therefore we can first choose a $(2\Ctwentysix)^{-1}\kappa_0$-net $Z_i\subset D_i^{1/10}$ and then choose $Y_i\subset F(D_i^{1/10})$ to be a maximal $(\kappa_0/2)$-separated subset of $F(Z_i)$.} 

In the step 7' we apply the steps 1-3 of algorithm {\it SubmanifoldInterpolation} for the set $Y=\bigcup_i Y_i$, consisting of $CN$  points that are in
$E=\R^m$. This  requires $CmN^2$ operations, and gives
us functions $f$ and  $\tilde \psi_i=f\circ F|_{D^{1/10}_i}$, $i=1,2,\dots, N$. This implements also the step 8' of the algorithm. 

{\ctext Next we construct for all $j\in \{1,2,\dots,N\}$ the sets $\mathcal I(j)\subset \{1,2,\dots,N\}$ of those $i$ for which  $d_X(q_i,q_j)<1$. By Lemma \ref{l:adjgraph}, the number of
elements in the set $\mathcal I(j)$ is bounded by a number depending only on $n$.
Thus the
 construction of sets $\mathcal I(j)$  for all $j\in \{1,2,\dots,N\}$ requires 
 $CN^2$ elementary operations.}

Next we consider pairs  $(i,j)$  such that  $i\in \mathcal I(j)$.
The functions $f$ and  $\tilde \psi_i$ are  piecewisely defined by explicit formulas.
{As explained in subsection \ref{rem. simplification 2}, 
we  can construct  numerically 
 the inverse maps $$G_{j,i}^{-1}=(A_j\circ P_j\circ \tilde \psi_i)^{-1}:\mathcal W_j(3)\to  D_i=B^{1/{10}}(p_i),$$
 see (\ref{Gi functions}).
 There, the construction of the orthogonal projections $P_i$ onto the tangent spaces $T_i$ 
require 
$CmN$ operations.  
{\ctext This finishes  one-time work needed for the preparatory work that
has required
\ba
& &2C\ell^2+2CN\ell
+ 
CmN^2+CN^2 
 + CmN\leq  C\ell^2+CN^3
\ea
operations.

Next we consider several queries. The first query we consider takes as an input a point $z\in \mathcal W_i(3)$ and gives as the output a numerical approximation for the inverse of the map $G_{j,i}$ at the point $z$.
To do that, we}
compute the inverse of the map $G_{j,i}$ by using  
 Newton's method.} By (\ref{eq:23h}), to compute {numerically the value $G_{j,i}^{-1}(z)$ for a given $z\in \mathcal W_i(3)$} and with  the precision of $\theta\in (0,\frac 14)$, it suffices to take $C \log_2 \log_2 (1/\theta)$ iterations.  {Below in this section,
 when we consider the computation of the transition functions and the metric tensor, all computations are done with the precision  $C\theta$.
 
{\ctext Let $i\in \mathcal I(j) $ and 
 $y^{(j)}\in \tilde D_j$  and $y^{(j,i)}\in W_{ji}$. Using the numerical computations described above one can compute the values of 
the values of  $z^{(j,i)}=G_{j,i}(y^{(j,i)})$ and $G_j^{-1}(z^{(j,i)})$.
Observe that $G_{j,i}$ is a composition of the linear operators
$A_j$ and $P_j$ and $\R^m$-valued functions $ \tilde \psi_i$ having explicit formulas containing parameters that we have already computed using the data. Hence this requires $Cm$ elementary operations. 
The second query we consider  takes as an input a point $y^{(j,i)}$
and gives as an output the value
 the extended transition function $\tilde \eta_{ji}^e(y^{(j,i)})=G_j^{-1}(G_{j,i}(y^{(j,i)}))$ with  the precision of $C\theta$, see (\ref{extended transition functions}). 
 As the computation of the derivative of the 
  $\R^m$-valued function $G_j$ requires $Cm$  operations,
 this query requires $Cm \log_2 \log_2 (1/\theta)$ elementary operations.

Next we consider the query where the input is a pair $(i,j)$,
satisfying $i\in \mathcal I(j)$,  and  a point $y^{(j,i)}$ and the output is the} metric tensor
\beq\label{g push forwards}
({\tilde \eta}^e_{ji})_*g^e=(G_j^{-1})_*(G_{j,i})_*g^e, 
\eeq
 evaluated at the point $y^{(j,i)}$.
Here, using the matrix notation, $\tilde g_{(ji)}= ({\tilde \eta}^e_{ji})_*g^e$  is given by
\ba
\tilde g_{(ji)}(y)
=\bigg(\frac {\partial G_{j}}{\partial y}(y)\bigg)^{-1}\cdotp 
\bigg(\frac {\partial G_{j,i}}{\partial z}(z)\bigg)\,\cdotp
g^e\,\cdotp\bigg(\frac {\partial G_{j,i}}{\partial z}(z)\bigg)^t\cdotp\bigg((\frac {\partial G_{i}}{\partial y}(y))^t\bigg)^{-1}\bigg|_{z=G_{j,i}^{-1}(y)}
\ea
where $g^e_{ab}=\delta_{ab}$  is the Euclidean metric.
We note that since $G_{j,i}$ is a composition of linear operators and functions having explicit formulas,
the derivative $DG_{j,i}$ can be computed at any given point {\ctext using elementary operations which number depend only on $n$. Using the Lipschitz estimates \eqref{C2 estimate}
and \eqref{e C66}, we see that 
computing $({\tilde \eta}^e_{ji})_*g^e$ at the point $y^{(j,i)}$
with the precision  $C\theta$ requires
$C \log_2 \log_2 (1/\theta)$ elementary operations.

Next we consider  the computation of the metric tensor $(\tilde \psi_j)^*g$ at 
  $y^{(j)}\in \tilde D_j$ with the precision $C\theta$.  This is  done by using a  partition of unity
  (\ref{partition}) on $M$ and computing a weighted sum of tensors
   $({\tilde \eta}^e_{ij})_*g^e$, as described in formula (\ref{local metric})
  where the sum needs to be taken over the  indexes $i$ and $k$ that satisfy $i,k\in \mathcal I(j)$. As number of elements in
  $\mathcal I(j)$ is bounded by $C=C(n)$, this requires 
  $C \log_2 \log_2 (1/\theta)$ elementary operations.
  Finally, we consider the query, where the input is $y^{(j)}$  and the output
  is $f(y^{(j)})$. As discussed above in the context of 
 the algorithm  {\it SubmanifoldInterpolation}, this takes $Cm$ elementary operations as we have already computed the parameters that determine the function $f$.
 Summarising the above, after the one-time work described above, we
 can consider a query, where the input is an index $j$ 
 and a point $y^{(j)}\in \tilde D_j$ and the output
  is the collection
  \beq\label{final query}
  f(y^{(j)}),\quad ((\tilde \psi_j)^*g)_{ab}(y^{(j)}),\quad
  \{\tilde \eta_{ji}(y^{(j)}):\ i\in \mathcal I(j)\}
  \eeq
 with the precision  $C\theta$. That is, the output is the point $  f(y^{(j)})$  on $M$ that corresponds to $y^{(j)}$
 on the local coordinate chart  $D_j$, the metric tensor on $D_j$
 at the point $y^{(j)}$, and the values of the transition functions from the chart
 $D_j$ to the charts $D_i$ at  $y^{(j)}$. After the one-time work,
 answering the query (\ref{final query}) requires 
 \ba
& &\hspace{-1cm}
Cm+Cm\log_2 \log_2 (1/\theta)\leq Cm\log_2 \log_2 (1/\theta)
 \ea
elementary operations.


To construct a grid of points on the manifold $M$, let  $\mathcal Q_{j}=\tilde D_j\cap ({\rho}\Z^n),$ $j=1,2,\dots,N,$ be  computational grids in the sets $\tilde D_j$,  where 
 ${\rho}>0$ is the grid size parameter. We write these grids as
 $$
 \mathcal Q_{j}=\{y^{(j)}_l:\ l=1,2,\dots,L_{j}\}.$$
 We see that 
the numbers $L_j$ of the points in $\mathcal Q_{j}$  are bounded by $C{\rho}^{-n}$.
Answering to the query \eqref{final query} for all points $y^{(i)}_l$ in the computational grid
$\bigcup_{j=1}^N  \mathcal Q_{j}$ produces a $C\rho$-dense set of grid points on the manifold $M$ and
the values of the metric tensors on all local coordinate charts $D_j$ at these grid points.
Summarising the above analysis, the computation work to do this requires altogether 
\ba
& &\hspace{-1cm}2C\ell^2+2CN\ell
+ 
CmN^2+CN^2 
 + CmN +C{{\rho}^{-n}}Nm\log_2 \log_2 (1/\theta)
  \\ 
    & & \quad \quad \leq  C\ell^2+CN^3+
    C{{\rho}^{-n}}N^2\log_2 \log_2 (1/\theta)
 \ea
elementary operations.}} We recall that here $r=1$, $C$  depends on the intrinsic dimension $n$ of the manifold, $\rho$ is the size parameter of the computational grid,  $m=(n+1)N$, $C\theta$ is the required numerical accuracy, $N$  is the number of points in the maximal $(r/100)-$separated set in $X$ and $\ell$ is the number of points in $X$.
}

 \section*{Appendix A: Estimates for higher derivatives of
a composition of functions}
 
 {
For a function $f :X\to \mathbb{R}^n$ defined in an open set $X\subset \R^n$,
 let $D_vf(x)=\partial_vf(x)$ be derivative of $f$ to direction $v\in\R^n$ at $x$. We denote $D_vf(x)=\partial f(x)\,\cdotp v$.
 The derivatives of order $k$ are below considered as multilinear forms,
 $$
 \partial_{v_1}\partial_{v_2}\dots\partial_{v_k}f(x)=   \partial^k f(x)[v_1,v_2,\dots,v_k],
$$
where $v_1, \dots v_k\in \R^n $ and $x \in X$.

 We consider higher order derivatives as multilinear forms {\ctext and} 
use following lemma for multilinear forms, proven in \cite[Prop.\ 9.1.1] {Nesterov}.

\begin{lemma}\label{thm:nest}
 Let $A[h_1,\dots, h_k]$ be a $k-$linear symmetric form on $\mathbb{R}^n$ and $B[h_1, h_2]$ be a symmetric positive semidefinite
$2-$linear form on $\mathbb{R}^n$. Assume that for some $\alpha$ one has $|A[h,\dots,h]|\leq \alpha (B[h,h])^{k/2}$,  for $h\in \mathbb{R}^n$.
Then,
for all $h_1, \dots, h_k \in \mathbb{R}^n$,
\ba
|A[h_1, \dots, h_k]| \leq \alpha \prod_{i=1}^k (B[h_i, h_i])^{1/2}.
\ea
\end{lemma}

We will refer to the above result as the Nesterov-Nemirovski Lemma.

  When $h:X\subset \mathbb{R}^n\to \R^n$, we define 
 \begin{eqnarray} \label{Hari1a}
 \|h\|_{C^k(X)}&=&\sup_{0\le r\le k} \sup_{v_1, \dots v_r\in S^{n-1}}\sup_{x\in X}\|\partial_{v_1} \dots \partial_{v_r} h(x)\|_{\R^n}\\
  \nonumber 
 &=&\sup_{0\le r\le k} \sup_{v\in S^{n-1}}\sup_{x\in X}\sup_{a\in S^{n-1}} |\partial_{v}^r  h(x)\,\cdotp a|,\\
 \label{Hari1b}
 \|h\|_{\dot C^k(X)}&=& \sup_{v_1, \dots, v_k\in S^{n-1}}\sup_{x\in X}\|\partial_{v_1} \dots \partial_{v_k} h(x)\|_{\R^n}\\
&=& \sup_{v\in S^{n-1}}\sup_{x\in X}\sup_{a\in S^{n-1}} |\partial_{v}^k  h(x)\,\cdotp a|,   \nonumber 
  \end{eqnarray}
     where the equalities in (\ref{Hari1a}) and  (\ref{Hari1b}) follow from Lemma \ref{thm:nest}. 
     We note that by Lemma \ref{thm:nest}, we have for example that
 \beq\label{e: decomposition of derivative}
  \sup_{v_1, \dots, v_k\in S^{n-1}}\|\partial_{v_1} \dots \partial_{v_k} h(x)\|_{\R^n}=
   \sup_{v,w\in S^{n-1}}\|\partial_{v}^{k-j} \partial_{w}^j h(x)\|_{\R^n}.
 \eeq
   We denote $C^k(X)=C^k(X;\R^n)$ and use the Euclidean norm $a\mapsto \|a\|_{\R^n}$ in $\R^n$.

   When $L_s^m(\R^n)$ is the set of symmetric $m$-multilinear maps $A:(\R^n)^m\to\R^n$, we define
  $$
  \|A\|_{L_s^m}=\sup_{w_1, \dots w_m\in S^{n-1}}\|A[w_1,w_2,\dots, w_m]\|_{\R^n}=\sup_{\|a\|_ {\R^n}=1}(\sup_{w\in S^{n-1}}|A[w,w,\dots, w]\,\cdotp a|)
  $$
and
   for 
   $H:X\subset \mathbb{R}^n\to  L^m(\R^n)$, define 
    \begin{eqnarray}
 \|H\|_{C^k(X;L_s^m)}&=&\sup_{0\le r\le k} \sup_{v_1, \dots v_r\in S^{n-1}}\sup_{x\in X}\|\partial_{v_1} \dots \partial_{v_r} H(x)\|_{L_s^m} \\
   \nonumber 
 &=&\sup_{0\le r\le k} \sup_{v\in S^{n-1}}\sup_{x\in X}\|\partial_{v}^r  H(x)[w,w,\dots,w]\|_{\R^n} ,\\
 \|H\|_{\dot C^k(X;L_s^m)}&=& \sup_{v_1, \dots v_k\in S^{n-1}} \sup_{w\in S^{n-1}}\sup_{x\in X}\|\partial_{v_1} \dots \partial_{v_k} H(x)\|_{L_s^m}    \\
&=& \sup_{v\in S^{n-1}} \sup_{w\in S^{n-1}}\sup_{x\in X}\|\partial_{v}^k  H(x)[w,w,\dots,w]\|_{\R^n},   \nonumber 
  \end{eqnarray}
    where we have again used the Nesterov-Nemirovskii lemma.
%


In the main text of the paper, we use the following lemma:

\begin{lemma}
\label{l:iterated chainrule}
(1) Let $k\in \N$ and $f_i : \mathbb{R}^n \to \mathbb{R}^n$, $i=1,2,\dots,N$   be $C^k-$smooth functions. Then
\beq\label{e- products}
\|f_1 \circ \dots \circ f_N\|_{C^k}\leq 2^{k(k+1)(N-1)/2}\|f_1\|_{C^k}(1+ \|f_2\|_{C^k})^k\dots (1+\|f_N\|_{C^k})^k.\hspace{-2cm}
\eeq

\noindent
(2)  Let $k\in \Z_+$. When $X\subset \mathbb{R}^n$ is open, $Y\subset \mathbb{R}^n$ is open and convex,  and 
$f :Y\to \mathbb{R}^n$ is  $C^{k+1}-$smooth and $g,h:X\to \mathbb{R}^n$ are  $C^{k}-$smooth functions such that 
$g(X)\subset Y$ and $h(X)\subset Y$, 
we have \begin{eqnarray*}
& &    \| (f \circ g)-( f \circ h)\|_{C^{k}(X)}
\\ \nonumber 
& &  \leq  
(k+1)2^{k(k-1)}\|  f  \|_{C^{k+1}(Y)}(1+\| g\|_{C^{k}(X)}+\|  h  \|_{C^{k}(X)})^{k} \| g- h \|_{C^{k}(X)}.
\end{eqnarray*}
\end{lemma}

 \begin{proof} 
 
(1) Consider first the case $N=2$. The claim is clearly valid for $k=0$.
Let us next assume that (\ref{e- products}) is valid for $k-1$. Then 
\begin{eqnarray}
\|f \circ g\|_{\dot C^k}\hspace{-3mm}  & \leq &\hspace{-3mm} \sup_{v, w:|w| = |v| = 1} \|((\partial_w f)\circ g)(\partial_v g)\|_{\dot C^{k-1}}\\ \nonumber
& \leq & \hspace{-2mm} 2^{k-1} \sup_{v, w:|w| = |v| = 1}\|(\partial_w f)\circ g\|_{C^{k-1}}\|\partial_v g\|_{C^{k-1}}\\
 \nonumber
& \leq &\hspace{-2mm}  2^{k-1}\,\cdotp 2^{k(k-1)/2} \sup_{v, w:|w| = |v| = 1}\|\partial_w f\|_{C^{k-1}}
(1+\|g\|)^{k-1}_{C^{k-1}}
\|\partial_v g\|_{C^{k-1}}\hspace{-1cm}\\
 \nonumber
& \leq & \hspace{-2mm} 2^{k(k+1)/2} \|f\|_{C^k} (1+\|g\|_{C^k})^k.\end{eqnarray}
This and (\ref{e- products}) for $k-1$ yields that (\ref{e- products}) is valid for $k$. By induction, (\ref{e- products}) is valid for $N=2$.

Iterating (\ref{e- products}) for composition of two functions $N$  times yields (\ref{e- products}) for general $N$.
This {\ctext proves} the claim (1).

(2) We have
\begin{eqnarray*} 
\partial_v(f \circ g)(x)=( ( \partial_{w}   f) \circ g)(x)\bigg|_{w=\partial_v g}= ( ( \partial  f) \circ g)(x)) [\partial_v g(x)] 
 \end{eqnarray*}
 and 
$
\partial_w ((\partial^i_v f) \circ g)= ( \partial (\partial^i_v f)) \circ g) \partial_w g 
$, where ${w=\partial_v g}$,
 and
 \begin{eqnarray*} 
& &\partial_w((\partial^i_v f) \circ g)-\partial_w((\partial^i_v f) \circ h)\\
&=&( (\partial (\partial^i_v f)) \circ g) [\partial_w g]- ((\partial (\partial^i_v f) )\circ h) [\partial_w h] \\
&=& \bigg((\partial (\partial^i_v f)) \circ g-(\partial (\partial^i_v f)) \circ h\bigg) [\partial_w g]+ ((\partial (\partial^i_v f)) \circ h)\bigg[\partial_w g- \partial_w h \bigg]  \end{eqnarray*}
For  an integer $i\leq k$, let $a_{i, k-i} :=  \| ((\partial^i f) \circ g)-((\partial^i f) \circ h)\|_{C^{k-i}(X;L_s^i)}.$

 When $g(X)\subset X$ and $i\leq k-1$ is  an integer,  the Nesterov-Nemirovskii lemma (see also (\ref{e: decomposition of derivative})) implies
   \begin{eqnarray*} 
& &\| ((\partial^i f) \circ g)-((\partial^i f) \circ h)\|_{\dot C^{k-i}(X;L_s^i)}\\
&=&\sup_{\|v\|=\|w\|=\|w'\|=1}\| \partial^{k-1-i} _{w'}\partial_w((\partial^i_v f) \circ g)-\partial^{k-1-i}_{w'}\partial_w((\partial^i_v f) \circ h)\|_{C^0(X;\R^n)}\\
&\leq &2^{k-1-i} \|(\partial_{w'} \partial^i_v f) \circ g-(\partial_{w'} \partial^i_v f) \circ h\|_{C^{k-1-i}(X;\R^n)}
\,\cdotp \|
\partial_w g\|_{C^{k-1-i}(X;\R^n)}\\
& &+2^{k-1-i}  \sup_{\|v\|=\|w\|=\|w'\|=1}
 \|
(\partial_{w'}\partial^i_v f) \circ h\|_{C^{k-1-i}(X)}
\,\cdotp \| \partial_w g- \partial_w h \|_{C^{k-1-i}(X)} \\
&\leq &2^{k-1-i}  \|( \partial^{i+1} f) \circ g-(\partial^{i+1}f) \circ h\|_{C^{k-1-i}(X;L_s^{i+1})}
\,\cdotp \|
\partial g\|_{C^{k-1-i}(X;L_s^1)}\\
& &+2^{k-1-i} 
 \|
(\partial^{i+1} f) \circ h\|_{C^{k-1-i}(X;L_s^{i+1})}
\,\cdotp \| \partial g- \partial h \|_{C^{k-1-i}(X;L_s^1)}.
 \end{eqnarray*} 
 Thus
  \begin{eqnarray*} 
& &\sup_{i \leq r \leq k-1} \| ((\partial^i f) \circ g)-((\partial^i f) \circ h)\|_{\dot C^{r-i+1}(X;L_s^i)}\\
& \leq &  2^{k-1-i}  \|( \partial^{i+1} f) \circ g-(\partial^{i+1}f) \circ h\|_{C^{k-1-i}(X;L_s^{i+1})}
\,\cdotp \|
 g\|_{C^{k-i}(X)}\\
& &+2^{k-1-i} 
 \|
(\partial^{i+1} f) \circ h\|_{C^{k-1-i}(X;L_s^{i+1})}
\,\cdotp \| g- h \|_{C^{k-i}(X)}. \end{eqnarray*}
Using this and the claim (1) of the lemma, we see that  
\begin{eqnarray} \nonumber
 a_{i, k-i}
&=&\sup_{i-1 \leq r \leq k-1} \| ((\partial^i f) \circ g)-((\partial^i f) \circ h)\|_{\dot C^{r-i+1}(X;L_s^i)}
\end{eqnarray}
satisfy
\begin{eqnarray} \label{e:induction of a}
 a_{i, k-i}
&\leq & 2^{k-1-i}  a_{i+1, k-1-i}
\,\cdotp \|
 g\|_{C^{k-i}(X)}
+ 2^{k-1-i} b_k,
\end{eqnarray}
where 
 \begin{eqnarray}\label{eq:bi+1,k-i}  
 b_k = 2^{k(k-1)/2}
\|  f  \|_{C^{k}(Y)}(1+\| g\|_{C^{k}(X)}+\|  h  \|_{C^{k}(X)})^{k} \| g- h \|_{C^{k}(X)}. \hspace{-2cm}
\end{eqnarray} 
 Since $Y$  is convex, we have by  the Nesterov-Nemirovskii lemma
\begin{eqnarray*} \nonumber
 a_{k,0}
&=& \sup_{\|v\|=1}\sup_{x\in X} \| ((\partial^{k}_vf) \circ g(x))-((\partial^{k}_v f) \circ h(x))\|_{C^{0}(X;L_s^{k-1+1})}\\
&\leq &\sup_{\|v\|=\|w\|=1}\sup_{y\in Y} \| \partial_w\partial^{k}_vf(y)\|\, \| g- h \|_{C^{0}(X)}\\
&\leq &\|  f  \|_{C^{k+1}(Y)} \| g- h \|_{C^{k}(X)}. 
\end{eqnarray*}
Using this and induction for (\ref{e:induction of a}) in the index $i$, we see that 
\begin{eqnarray*} \label{e:resul to a}
 a_{0, k}
\leq 2^{k(k-1)}\|  f  \|_{C^{k+1}(Y)}(1+\| g\|_{C^{k}(X)}+\|  h  \|_{C^{k}(X)})^{k} \| g- h \|_{C^{k}(X)}.
 \end{eqnarray*}
This yields the claim (2).
\end{proof}
}

\noindent{\bf Acknowledgements.} The authors express their gratitude to the Mittag-Leffler Institute,
the Institut Henri Poincare, and the Fields Institute,
where parts of this work have been done.

Ch.F.\  was partly supported {by} AFOSR, grant   DMS-1265524, and NSF, grant FA9550-12-1-0425.
S.I.\ was partly supported {by} RFBR, {grants 14-01-00062 and 17-01-00128-A}, Y.K.\ was partly supported by 
EPSRC and the AXA professorship, 
M.L.\ was  supported by Academy of Finland, grants 273979, 284715, 312110 and H.N. was partly supported by NSF grant DMS-1620102 and a Ramanujan Fellowship.

\end{document}